\documentclass[11pt]{book}

\usepackage{latexsym,mathrsfs}
\usepackage{amsmath,amssymb} 
\usepackage{amsthm,enumerate,verbatim}
\usepackage{amsfonts}
\usepackage{graphicx}
\usepackage{algorithm}[chapter]
\usepackage{algorithmic}[chapter] 
\usepackage{url}
\usepackage[svgnames]{xcolor} 
\usepackage{color} 
\usepackage{ulem} 
\usepackage{nicefrac} 
\usepackage{lscape}  
\usepackage{bbding}
\usepackage{mhchem}
\usepackage{makeidx} 
\makeindex

\usepackage{tikz}
\usetikzlibrary{shapes,calc,matrix}
\usetikzlibrary{math}
\usetikzlibrary{shapes.misc}
\tikzset{cross/.style={cross out, very thick, draw=black, 
minimum size=12pt, inner sep=0pt, outer sep=-1pt},
cross/.default={2pt}}
\definecolor{brightpink}{rgb}{1.0, 0.0, 0.5}
\definecolor{orange}{rgb}{1.0, 0.49, 0.0}
\definecolor{vert}{rgb}{0.1,0.5,0.2} 

\definecolor{mediumelectricblue}{rgb}{0.01, 0.31, 0.59}

\usepackage{hyperref}
\usepackage{backref} 

\usepackage{hyperref}
\hypersetup{
    colorlinks = true,
    citecolor = {blue}
}

\newtheorem{lemma}{Lemma}[chapter]
\newtheorem{theorem}{Theorem}[chapter]
\newtheorem{corollary}{Corollary}[chapter] 
\newtheorem{definition}{Definition}[chapter]
\newtheorem{example}{Example}[chapter]

\newtheorem{remark}{Remark}[chapter]
\newtheorem{problem}{Problem}[chapter]

\numberwithin{equation}{chapter}
\numberwithin{figure}{chapter}

\DeclareMathOperator{\argmin}{argmin} 
 
\DeclareMathOperator{\rank}{rank}

\DeclareMathOperator{\tr}{tr} 
\DeclareMathOperator{\diag}{diag}

\def \R{{\mathbb R}}
\def \C{{\mathbb C}}
\def\real{\mathop{\mathrm{Re}}}
\def\imag{\mathop{\mathrm{Im}}}

\newcommand {\mat}  [1] {\left[\begin{array}{#1}}
\newcommand {\rix}      {\end{array}\right]}

\makeatletter
\newsavebox\myboxA
\newsavebox\myboxB
\newlength\mylenA
\newcommand*\xoverline[2][0.75]{%
    \sbox{\myboxA}{$\m@th#2$}%
    \setbox\myboxB\null
    \ht\myboxB=\ht\myboxA%
    \dp\myboxB=\dp\myboxA%
    \wd\myboxB=#1\wd\myboxA
    \sbox\myboxB{$\m@th\overline{\copy\myboxB}$}
    \setlength\mylenA{\the\wd\myboxA}
    \addtolength\mylenA{-\the\wd\myboxB}%
    \ifdim\wd\myboxB<\wd\myboxA%
       \rlap{\hskip 0.5\mylenA\usebox\myboxB}{\usebox\myboxA}%
    \else
        \hskip -0.5\mylenA\rlap{\usebox\myboxA}{\hskip 0.5\mylenA\usebox\myboxB}%
    \fi}

\def\ycite[#1#2#3#4#5]#6{\cite[$\mit{#1#2#3#4}$#5]{#6}}

\newtheorem{proposition}{Proposition}

\title{Lecture notes of the C.I.M.E. \\ 
- \\ 
Solving matrix nearness problems via \\ 
Hamiltonian systems, 
matrix factorization, 
and optimization} 

\date{}

\author{
Nicolas Gillis\footnote{
Department of Mathematics and Operational Research, 
University of Mons, 
Rue de Houdain 9, 
7000 Mons, 
Belgium. 
Email: nicolas.gillis@umons.ac.be. 
NG acknowledges the support by ERC starting grant No 679515, and by the Fonds de la Recherche Scientifique - FNRS
and  the  Fonds  Wetenschappelijk  Onderzoek - Vlaanderen (FWO) 
under EOS project O005318F-RG47. 
} 
\and 
Punit Sharma\footnote{Department of Mathematics, Indian Institute of Technology Delhi, Hauz Khas, New Delhi-110016, India. Email: \texttt{punit.sharma@maths.iitd.ac.in}.
PS acknowledges the support of the DST-Inspire Faculty Award (MI01807-G) by Government of India, and Institute SEED Grant (NPN5R) by IIT Delhi.} 
	}

\begin{document}

\normalem

\maketitle

\section*{Preface}

These notes were written for the summer school on ``Recent stability issues for linear dynamical systems - Matrix nearness problems and eigenvalue optimization'' organized by 
Nicola Guglielmi and Christian Lubich at the Centro Internazionale Matematico Estivo (CIME) in September 2021; 
see \url{http://php.math.unifi.it/users/cime/Courses/2021/course.php?codice=20216}. 

The aim of these notes is to summarize our recent contributions to compute nearest stable systems from unstable ones, namely 
\begin{itemize}

\item[\cite{GilS16}] N. Gillis and P. Sharma, ``A semi-analytical approach for the positive semidefinite Procrustes problem", Linear Algebra and its Applications 540, pp. 112-137, 2018.

\item[\cite{gillis2018computing}] N. Gillis, V. Mehrmann and P. Sharma, "Computing nearest stable matrix pairs", Numerical Linear Algebra with Applications 25 (5), e2153, 2018. 

\item[\cite{gillis2018finding}] N. Gillis and P. Sharma, "Finding the nearest positive-real system", SIAM J. on Numerical Analysis 56 (2), pp.\ 1022-1047, 2018. 

\item[\cite{gillis2019approximating}] N. Gillis, M. Karow and P. Sharma, "Approximating the nearest stable discrete-time system", Linear Algebra and its Applications 573, pp.\ 37-53, 2019. 

\item[\cite{gillis2020note}] N. Gillis, M. Karow and P. Sharma, ``A note on approximating the nearest stable discrete-time descriptor system with fixed rank", Applied Numerical Mathematics 148, pp.\ 131-139, 2020. 

\item[\cite{choudhary2020approximating}] N. Choudhary, N. Gillis and P. Sharma, ``On approximating the nearest $\Omega$-stable matrix", Numerical Linear Algebra with Applications 27 (3), e2282, 2020. 

\item[\cite{gillis2020minimal}] N. Gillis and P. Sharma, ``Minimal-norm static feedbacks using dissipative Hamiltonian matrices", Linear Algebra and its Applications 623, pp.\ 258-281, Special issue in honor of Paul Van Dooren, 2021.  

\end{itemize}

Hence most of the material of these notes take its roots in these papers, where the interested reader can find more details.

\paragraph{Matlab code} 

All Matlab codes used in these notes can be found at \url{https://sites.google.com/site/nicolasgillis/code}. 
During the course, we will show how to use the different algorithms with some numerical examples.

\paragraph{Slides}  The slides presented during the summer school are available from \url{https://www.dropbox.com/s/b33wd0j9pyiflar/CIME_Gillis_slides.pdf?dl=0}.

\paragraph{Acknowledgments}

We are grateful to Nicola Guglielmi and Christian Lubich for giving us the opportunity to present our work at the CIME. 

We thank our collaborators, 
Volker Mehrmann, 
Michael Karow and 
Neelam Choudhary for the fruitful and enjoyable moments spent working on these problems.

%
%
%
%
%
%
%
%
%
%

\newpage 

\tableofcontents 

\newpage 


\section*{Notation}

\subsection*{Sets of scalars, vectors, matrices}

\begin{tabular}{ll}
$\mathbb{R}$  &  set of real numbers \\
$\mathbb{R}_+$  &  set of nonnegative real numbers \\
$\mathbb{R}_{++}$  &  set of positive real numbers \\
$\mathbb{R}^n$ &  set of real column vectors of dimension $n$ \\
$\mathbb{R}^{m\times n}$ &  set of real $m$-by-$n$ matrices   \\
$\mathbb{R}^n_+$ &  set of nonnegative real column vectors of dimension $n$ \\
$\mathbb{R}^{m\times n}_+$ &  set of $m$-by-$n$  nonnegative real matrices  \\ 
$\mathbb{N}$  &  set of natural numbers (nonnegative integers) \\
$\mathbb{C}$  &  set of complex numbers \\
$\mathbb C_{-}$ & open left half of the complex plane, $\{\lambda \in \C:~\real{\lambda} < 0\}$ \\ 
$\mathcal{S}_+^n$ &  set of $n$-by-$n$ positive semidefinite (PSD) matrices \\
$\mathbb{S}^{n, n}$ &  set of $n$-by-$n$ continuous  stable matrices \\ & 
(eigenvalues in $\mathbb C_{-} \cup  i \mathbb{R}$, eigenvalues on $i \mathbb{R}$ are semisimple)  \\  
$\mathbb S_d^{n,n}$ & set of $n$-by-$n$ discrete stable matrices \\ 
&  (eigenvalues in the unit circle, eigenvalues of unit modulus are semisimple)  \\ 
\end{tabular}

\subsection*{Submatrices, transpose and inverse}

\begin{tabular}{ll}
$x_i$ of $x(i)$ & $i$th entry of the vector $x$ \\ 
$A_{i:}$ or $A(i,:)$  & $i$th row of $A$ \\
$A_{:j}$ or $A(:,j)$  & $j$th column of $A$ \\
$A_{ij}$ or $A(i,j)$  & entry at position $(i,j)$ of $A$ \\
$A(I,J)$  & submatrix of $A$ with row (resp.\ column) indices in $I$ (resp.\ $J$)  \\
$[A \ B; C \ D]$ & We use Matlab notation: $[A \ B; C \ D] = \left( \begin{array}{cc}
A & B \\ 
C & D 
\end{array} \right)$ \\ 
$A^\top$ & transpose of the matrix $A$, $(A^\top)_{ij} = A_{ji}$ \\
$A^{-1}$ & inverse of the square matrix $A$,  $A^{-1} A = AA^{-1} = I$\\ 
$A^{-\top}$ & inverse of the transpose of the square matrix $A$,  
$A^{-\top} A^\top = A^\top A^{-\top}  = I$ \\ 
\end{tabular}

\subsection*{Norms}


\begin{tabular}{ll}
$\|.\|_2$  &  vector $\ell_2$-norm,  $\|x\|_2 = \sqrt{\sum_{i=1}^n x_{i}^2}$, $x \in \mathbb{R}^{n}$  \\
      &  matrix $\ell_2$-norm,  $\|A\|_2 = \max_{x \in \mathbb{R}^n, \|x\|_2 = 1} {\|Ax\|_2}$, $A \in \mathbb{R}^{m \times n}$  \\
$\|.\|_F$  &  Frobenius norm,  $\|A\|_F = \sqrt{\sum_{i=1}^m \sum_{j=1}^n A_{ij}^2}$, $A \in \mathbb{R}^{m \times n}$  
\end{tabular}

\subsection*{Inequalities}

\begin{tabular}{ll}
$A \geq 0$  &  $A$ is a nonnegative matrix, that is, $A(i,j) \geq 0$ for all $i,j$  \\ 
$A \geq B$  &  This means $A-B \geq 0$  \\ 
$A \succeq 0$  &  $A$ is a PSD matrix \\  
$A \succeq B$  &  $A-B$ is a PSD matrix  \\ 
 $A \succ 0$   &  $A$ is a positive definite matrix  \\ 
 $A \succ B$   &  $A-B$ is a positive definite matrix  \\ 
\end{tabular}

%

\subsection*{Functions and sets on matrices}

\begin{tabular}{ll}
$\left \langle .,. \right \rangle$  &  Euclidean scalar product,  
$\left \langle A,B \right \rangle = \sum_{i=1}^m \sum_{j=1}^n A_{ij} B_{ij}, A, B \in \mathbb{R}^{m \times n}$\\
$\sigma_i(A)$  & $i$th singular values of matrix $A$, in non-decreasing order   \\
$\sigma_{\max}(A)$  & largest singular value of $A$, that is, $\sigma_1(A)$ \\ 
$\sigma_{\min}(A)$  & smallest singular value of $A \in \mathbb{R}^{m \times n}$, that is, $\sigma_{\min(m,n)}(A)$ \\
$\kappa(A)$ & condition number of $A$, $\kappa(A) = \frac{\sigma_{\max}(A)}{\sigma_{\min}(A)}$ \\ 
$\mathcal{P}_{\succ 0}(A)$ & projection of $A$ onto the set of PSD matrices  \\ 
$\mathcal{P}_{\bar S}(A)$ & projection of $A$ onto the set of skew-symmetric matrices  \\ 
$\det(A)$ & determinant of $A$ \\  
$\tr(A)$ & trace of $A$, that is, sum of its diagonal entries \\ 
$\Lambda(A)$ & set of eigenvalues of $A$\\
$\rho(X)$ & spectral radius of $A$,  $\rho(X)=\max_{\lambda \in \Lambda(X)}{|\lambda|}$\\
$\diag(.)$  & For $x \in \mathbb{R}^n$, 
$X = \diag(x) \in \mathbb{R}^{n \times n}$ is a diagonal matrix such that $X_{ii} = x_i$ for all $i$ \\ 
& For $X \in \mathbb{R}^{n \times n}$, $x = \diag(X) \in \mathbb{R}^{n}$ is the vector containing the diagonal entries of $X$ \\ 
$\rank(.)$  &  rank of a matrix \\
\end{tabular}

\subsection*{Special vectors and matrices}  

\begin{tabular}{ll}
$0$  &  matrix of zeros of appropriate dimension \\
$0_{m \times n}$  &  $m$-by-$n$ matrix of zeros  \\
$I_n$  &  identity matrix of dimension $n$ \\
$I$  &  identity matrix of appropriate dimension \\
\end{tabular}

\subsection*{Miscellaneous}  

\begin{tabular}{ll}
$\infty$ & infinity \\ 
$i$ & imaginary number, $i^2 = -1$ \\ 
$\real{\lambda}$ & real part of the complex number $\lambda \in \mathbb{C}$ \\ 
$\imag{\lambda}$ & imaginary part of the complex number $\lambda \in \mathbb{C}$ \\ 
$a$:$b$  &  set  $\{a,a+1,\dots,b-1,b\}$ (for $a$ and $b$ integers with $a \leq b$) \\
$[a,b]$  & closed interval for reals $a \leq b$ \\
$(a,b)$  & open interval for reals $a \leq b$ \\
$\nabla f$  & gradient of the function $f$   \\
$\nabla^2 f$  &  Hessian of the function $f$    \\
$\lceil . \rceil$  & $\lceil x \rceil$ is the smallest integer greater or equal to $x \in \mathbb{R}$   \\
$\lfloor . \rfloor$  & $\lfloor x \rfloor$ is the largest integer smaller or equal to $x \in \mathbb{R}$ \\
$\backslash$ & subtraction of two sets, that is, $R \backslash S$ is the set of elements in $R$ not in  $S$ \\
$|.|$ & cardinality of a set, $|S|$ is the number of elements in $S$   \\
$f(x) = O(g(x))$  
& Big $O$ notation:  \\ & there exists $K$ and $x_0$ such that 
  $f(x) \leq K g(x)$ for all $x \geq x_0$   \\ 
$\inf_{x \in \mathcal{X}} f(x)$  & infimum value of $f(x)$ over the feasible set $\mathcal{X}$ \\  
$\min_{x \in \mathcal{X}} f(x)$  & minimum value of $f(x)$ over the feasible set $\mathcal{X}$ \\ 
$\argmin_{x \in \mathcal{X}} f(x)$  & set of minimizers of $f(x)$ over the feasible set $\mathcal{X}$ \\
\end{tabular}

\subsection*{Abbreviations}

\begin{tabular}{ll}
a.k.a.    &  also known as \\ 
w.l.o.g.  &  without loss of generality  \\
w.r.t. & with respect to 
\end{tabular}

\subsection*{Acronyms}

\begin{tabular}{ll} 
BCD & block coordinate descent \\ 
DH  & dissipative Hamiltonian \\ 
ESPR& extended strictly positive real \\ 
FGM & fast gradient method  \\
IPM & interior-point method \\
LMI & linear matrix inequality \\
LTI & linear-time invariant  \\
PGM & projected gradient descent \\ 
PH  & port-Hamiltonian \\ 
PR  & positive real \\ 
PSD & positive semidefinite  \\
SDP & semidefinite program(ming) \\
SSDP& sequential semidefinite programming \\
SPR & strictly positive real   \\ 
\end{tabular}


\chapter{Introduction}   \label{chap:intro}

The main goal of these lecture notes is to survey a series of recent 
works~\cite{GilS16, 
gillis2018computing, 
gillis2018finding,  
gillis2019approximating, 
gillis2020note, 
choudhary2020approximating, 
gillis2020minimal} 
that aim at solving several nearness problems
for a given system. 
As we will see, these problems can be written as distance problems of matrices or matrix pencils. 
To solve them, this series of recent works rely on a two-step approach:  
\begin{enumerate}

\item {\it Parametrization}:~Parametrize the system using a Port-Hamiltonian representation where stability is guaranteed via convex constraints on the parameters. 

\item {\it Algorithmic solution}:~Apply standard non-linear optimization algorithms to optimize these parameters, minimizing the distance between the given system and the sought parametrized stable system.  

\end{enumerate}

Before we delve into the technical details, let us introduce and motivate the study of this problem. In the following, we recall 
what is a dynamical system, 
how its stability is characterized, and 
why finding a nearby system with a specific system property is useful.

 \section{Dynamical systems}

A general dynamical system in the state-space form consist of the following two equations~\cite{Kha92,KunM06,Dua10,Broc15,Kai80}:
\begin{align}
\label{eq1:sys}
\begin{split}
f(\dot{x}(t), x(t),u(t),t)=0,
\\
g( x(t),u(t),y(t),t) = 0& \; ,\quad t \in I,
\end{split}
\end{align}
and an initial condition $x(t_0)=x_0$, where $\dot{x}=\frac{dx}{dt}$ and 
\begin{itemize}
\item $I=[t_0,t_f]$ is the time interval of interest.
\item $x(t):I \mapsto \mathcal X$ is the the state function, where $\mathcal X$ is the state space ($\mathcal X \subseteq \R^{n}$ open set), 
\item $\dot{x}(t):I \mapsto \dot{ \mathcal X}$, with $\dot{\mathcal X} \subseteq \R^{n}$ being an open set,  
\item  $u(t):I \mapsto \mathcal U$ is the control or input function, with  
$\mathcal U$ being the control space ($\mathcal U \subseteq \R^{m}$ with a metric), 
\item $y(t):I \mapsto \mathcal Y$ is the output function, with  
$\mathcal Y$ being the output space ($\mathcal Y \subseteq \R^{p}$ with a metric), 
\item $f:\dot{\mathcal X} \times \mathcal X \times \mathcal U \times I \mapsto \R^n$ defines the dynamics of the system, 
\item $g:\mathcal X \times \mathcal U \times \mathcal Y \times I \mapsto \R^p$ is the output map, and 
\item $n$ is the order of the system (state space dimension).
\end{itemize}


Often systems such as~\eqref{eq1:sys} are visualized as a block diagram; see Figure~\ref{Blockdiagram}. 
\begin{figure}[ht!]
\begin{center}
\includegraphics[width=7.5cm]{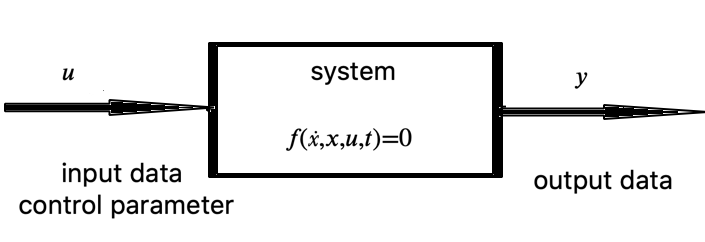}
\caption{Block diagram of a control system.   
\label{Blockdiagram}}
\end{center}
\end{figure} 

In the following examples and in many other situations, the system is not exactly known: the state $x$ and its derivative $\dot{x}$ cannot be measured. Therefore, the system is considered as a ``black box". 

\begin{example}{\rm
Influence of a change in the interest rate by the National Central Bank (of USA) on the stock market (e.g., Dow Jones Index), currency exchange rate (Euro vs US \$ ), and other financial markets. 
In terms of system theory, we have

$u$ is the interest rate, 

$x$ is the market parameters (stock, foreign exchange, capital, etc.), and 

$y$ contains the indices (Dow Jones, exchange rate, etc.)
}
\end{example}

\begin{example}{\rm
Cruise control of a car: it is a device designed to maintain vehicle speed at a constant desired speed provided by the driver.
In terms of system theory, we have

$u$ is the engines throttle position which determines how much power the engine delivers, 

$x$ contains the car parameters, and 

$y$ is the speed of the car. 
}
\end{example}

The main objective of control theory is to control a system, so its output follows a desired control signal, called the reference, which may be a fixed or changing value. Mathematically, given an initial value $x_0$ and target $x_1$, can we find an admissible input $\hat u \in \mathcal U$ such that there exists $t_1 \geq t_0$ with $x_1=x(t_1;\hat u)$, where $x(t;\hat u)$ is the solution trajectory of 
\[
f(\dot{x},x,\hat u,t)=0, \quad x(t_0)=x_0.
\]
Often, the target is $x_1=0$, that is, $x$ describes the deviation from a nominal path. 

The study of control systems can be divided into two branches:

\begin{itemize}
\item[] \emph{Linear Control Theory}:~ Linear control systems are governed by linear differential equations where the output is proportional to the input. They are divided into two subclasses; linear time-variant (if the input-output characteristics change with time) and linear time-invariant (if the input-output characteristics do not change with time).

\item[] \emph{Nonlinear Control Theory}:~ Nonlinear differential equations often govern nonlinear systems. These systems apply to more real-world systems because all real control systems are nonlinear. A few mathematical techniques have been developed to handle them, are more complicated and much less general. If only solutions near a stable point are of interest, the nonlinear system can often be linearized by approximating them using a linear system using perturbation theory. 
\end{itemize}

In these lecture notes, we mainly focus on the continuous-time and discrete-time linear time-invariant (LTI) systems. These systems can be well described, 
leading to solutions for system response and design techniques for most systems of interest~\cite{Antm97,Broc15,Kai80}.

\begin{definition} {\rm (Continuous-time LTI systems)}
A special case of~\eqref{eq1:sys} is the following continuous-time LTI  system 
\begin{align}
\label{eq2:sys}
\begin{split}
E\dot{x}(t) = & \; Ax(t)+Bu(t),
\\
y(t) = & \; Cx(t)+Du(t),
\end{split}
\end{align}
on the unbounded interval $t \in [t_0,\infty)$. Here, $A,E \in \R^{n,n}$, $B \in \R^{n,m}$, $C \in \R^{p,n}$, and $D \in \R^{p,m}$
are the coefficient matrices of the system, $x(t) \in \R^n$ is the state vector, $u(t) \in \R^m$ is the control input vector, and $y(t)\in \R^p$
is the measured output vector. 
\end{definition}

\begin{definition}{\rm (Discrete-time LTI systems)}
We also consider a class of discrete-time LTI systems of the form
\begin{align}
\label{eq3:sys}
\begin{split}
Ex(k+1) = & \; Ax(k)+Bu(k),
\\
y(k) = & \; Cx(k)+Du(k), \quad k \in \mathbb N,
\end{split}
\end{align}
where  $ \mathbb N$ is the set of nonnegative integers. 
The constant matrices $A,E \in \R^{n,n}$, $B \in \R^{n,m}$, $C \in \R^{p,n}$, and $D \in \R^{p,m}$
are the coefficient matrices of the system, $x(k) \in \R^n$ is the state vector, $u(k) \in \R^m$ is the control input vector, and $y(k) \in \R^p$
is the measured output vector. 
\end{definition}

The linear systems in~\eqref{eq2:sys} and~\eqref{eq3:sys} are  called \emph{standard systems} when $E=I_n$, where $I_n$ is the identity matrix
of size $n\times n$,
and a \emph{descriptor system} when $E$ is not identity.
We use the matrix quintuple $(E,A,B,C,D)$ to refer to a descriptor system and quadruple $(A,B,C,D)$ to refer to a standard system in the form~\eqref{eq2:sys} or~\eqref{eq3:sys}.

 \section{Stability}
 
Stability is an essential property in a control system which intuitively requires ``the output to converge to the desired value", as opposed to diverge from/oscillate around it. 
Consider a dynamical system of the form~\eqref{eq2:sys} for two different initial states, but the same input $u$. The system is called stable if the initial states are chosen close enough together, then the states remain close to each other for all time $t>0$. 

In this section, we investigate the stability of LTI systems. Since the theory of the stability of LTI descriptor systems is ambiguous, we consider the standard and descriptor LTI systems separately. A central figure in the study of stability of dynamical systems is Aleksandr Mikhailovich Lyapunov (1857-1918, student of Pafnuty Chebyshev in St.~Petersburg University).

\subsection{Standard systems}

The stability of a standard LTI system $(A,B,C,D)$ is well studied.
To study stability of LTI systems it is sufficient to consider $u(t)\equiv 0$, or equivalently the autonomous system $\dot{x}(t)=A x(t)$, $x(0)=x_0$.

\begin{definition}\label{def_stan_stab}
Let $\bar x$ and $\tilde x$ be two solutions of $\dot{x}=Ax$. Then
\begin{itemize}
\item An LTI system $(A,B,C,D)$/an autonomous system $\dot{x}(t)=Ax(t)$/a matrix $A \in \R^{n,n}$, is called continuous-time \emph{stable} (Lyapunov stable) 
if for every $\epsilon > 0$ there is a $\delta > 0$ such that 
\[
\|\bar x(0)-\tilde x(0)\| < \delta \text{ implies } 
\|\bar x(t)-\tilde x(t)\| < \epsilon \text{ for all } t > 0. 
\] 

\item It is called continuous-time \emph{asymptotically stable} if it is stable and  
\[
\lim_{t \rightarrow \infty} x(t)=0
\] 
for all solutions of $\dot{x}=Ax$.

\item It is called \emph{unstable}, if it is not stable.
\end{itemize}
The definitions of stability of discrete-time standard systems~\eqref{eq3:sys} are essentially identical to the corresponding definitions of stability of continuous-time systems described by ordinary differential equations~\eqref{eq2:sys}, replacing $t \in [0,\infty)$ by $k \in \mathbb N$.
\end{definition}

In the following, we state a variety of stability results that do not require explicit knowledge of the solutions of $\dot{x}=Ax$ and can be stated for an arbitrary $A \in \R^{n,n}$. The notation $\Lambda(A)$ denotes the set containing the eigenvalues of $A$. 

\begin{theorem} \label{th:stability}
Let $A \in \R^{n,n}$. Then
\begin{itemize}
\item $A$ is continuous-time asymptotically stable 
if and only if $\Lambda(A)\subset \mathbb C_{-}:=\{\lambda \in \C:~\real{\lambda} < 0\} $ if and only if  there exists $X \succ 0$ such that
$A^\top X+XA \prec 0$.
\item $A$  is continuous-time stable if and only if $\Lambda(A) \subset \mathbb C_{-} \cup\, i\R$ and all eigenvalues on the imaginary axis are semisimple (non-defective) if and only if
there exists $X \succ 0$ such that
$A^\top X+XA \preceq 0$. 
\item $A$ is discrete-time asymptotically stable if and only if 
$|\lambda| < 1$ for all $\lambda \in \Lambda(A)$ if and only if 
there exists $X \succ 0$ such that
$A^\top XA-X \prec 0$.
\item $A$ is discrete-time stable if and only if 
$|\lambda| \leq 1$ for all $\lambda \in \Lambda(A)$ and the eigenvalues with unit modulus are semisimple if and only if 
there exists $X \succ 0$ such that
$A^\top XA-X \preceq 0$.
\end{itemize}

\end{theorem}

\subsection{Descriptor systems}

Consider a continuous-time LTI descriptor system $(E,A,B,C,D)$ in the form~\eqref{eq2:sys}. 
Such a system is solvable if and only if there exists a unique solution for any given sufficiently differentiable control function $u(t)$ and any given admissible initial condition corresponding to an admissible $u(t)$~\cite{Yips81,Cam80}. It has been shown in~\cite{Cam80} that the system~\eqref{eq2:sys} is solvable if and only if the pencil $zE-A$ is regular, that is,
${\rm det}(\lambda E-A) \neq 0$ for some $\lambda \in \C$. 

Like the standard systems, while studying stability of descriptor systems~\eqref{eq2:sys}, we need only to consider the following homogeneous 
equation 
\begin{equation}\label{eq4:sys}
E\dot{x}(t) = Ax(t), 
\end{equation}
together with an initial condition
\begin{equation}\label{eq5:sys}
x(0) = x_0. 
\end{equation}
One can immediately extend
Definition~\ref{def_stan_stab} to regular systems~\eqref{eq4:sys}. However, one has to be careful with the initial conditions and inhomogeneities since they are restricted due to the algebraic constraints in the system. This is important, especially when one studies the stability under perturbations to the system, see, e.g., \cite[Example 1.1]{DuLM13} and~\cite[Example 1.2]{DuLM13} for possible difficulties in the stability concepts for descriptor system under small perturbations.

To characterize the stability for~\eqref{eq4:sys} under perturbations, let us introduce the following terminology~\cite{Gan59a,KunM06,Dua10}. A square matrix pair $(E,A)$ with $E,A \in \mathbb R^{n,n}$ is called \emph{regular} if the matrix pencil $z E-A$ is regular, that is, if
$\operatorname{det}(\lambda E-A)\neq 0$ for some $\lambda \in \mathbb C$, otherwise it is called \emph{singular}.
For a regular matrix pair $(E,A)$, the roots of the polynomial $\operatorname{det}(z E-A)$ are called \emph{finite eigenvalues} of the pencil
$zE-A$ or of the pair $(E,A)$, that is, $\lambda\in \mathbb C$ is a finite eigenvalue of the pencil
$zE-A$ if there exists a vector $x\in \mathbb C^n\setminus \{ 0\}$ such that
$(\lambda E-A)x=0$, and $x$ is called an \emph{eigenvector} of $zE-A$ corresponding to the eigenvalue $\lambda$.
A regular pencil $zE-A$ has \emph{$\infty$ as an eigenvalue} if $E$ is singular.

Any regular matrix pair $(E,A)$ (with $E,A\in \R^{n,n}$) can be transformed to \emph{Weierstra\ss\ canonical form} \cite{Gan59a}, that is, there exist nonsingular matrices $W, T \in \C^{n,n}$ such that
\[
E=W\mat{cc}I_q& 0\\0&N\rix T \quad \text{and}\quad A=W \mat{cc}J &0\\0&I_{n-q}\rix T,
\]
where $J \in \C^{q,q}$ is a matrix in \emph{Jordan canonical form} associated with the $q$ finite eigenvalues of
the pencil $z E-A$ and  $N \in \C^{n-q,n-q}$ is a nilpotent matrix in Jordan canonical form  corresponding
to  $n-q$ times the  eigenvalue $\infty$. If $q < n$ and $N$ has degree of nilpotency $\nu \in \{1,2,\ldots\}$, that is, 
$N^{\nu}=0$ and $N^i \neq 0$ for $i=1,\ldots,\nu-1$, then $\nu$ is called the \emph{index of the pair} $(E,A)$. If $E$
is nonsingular, then by convention the index is $\nu=0$.
A pencil $zE-A$ is of index at most one  if it is regular with exactly $r:=\text{rank}(E)$ finite eigenvalues, see, e.g.,  \cite{Meh91,Var95}. In this case the $n-r$ copies of the eigenvalue $\infty$ are semisimple.

The literature on (asymptotic) stability of constant coefficient DAEs is very ambiguous, see, e.g., \cite{BoyGFB94,ByeN93,Var95}, and the review in \cite{DuLM13}. This ambiguity arises from the fact that some authors consider only the finite eigenvalues in the stability analysis and allow the index of the pencil $(E,A)$ to be arbitrary, others consider regular high index pencils $zE-A$ as unstable by considering $\infty$ to be on the imaginary axis. We use the following definition.
%
\begin{definition}\label{prop1}
Consider the initial value problem~\eqref{eq4:sys}-\eqref{eq5:sys}.
\begin{itemize}
\item [i)] The linear DAE~\eqref{eq4:sys} is regular
if the matrix pair $(E,A)$ is regular~\cite{KunM06}.
\item [ii)] If the pair $(E,A)$
is regular and of index at most one, then the initial value problem~\eqref{eq4:sys} and~\eqref{eq5:sys} with a consistent initial value $x_0$ is \emph{stable} if all the finite eigenvalues of $zE-A$ are in the closed left half of the complex plane and those on the imaginary axis are semisimple~\cite{ByeN93,DuLM13}.
In this case $(E,A)$ is called a stable matrix pair.
\item [iii)] If the pair $(E,A)$ is regular and of index at most one, then the initial value problem~\eqref{eq4:sys} 
and~\eqref{eq5:sys} with a consistent initial value $x_0$ is \emph{asymptotically stable} if  all the finite eigenvalues of $zE-A$ are in the open left half of the complex plane~\cite{ByeN93,DuLM13}.
In this case $(E,A)$ is called an asymptotically stable matrix pair.
\end{itemize}
\end{definition}

The matrix pairs which are regular, of index at most one and asymptotically stable are also known as \emph{admissible} pairs~\cite{Dua10}. The admissibility of matrix pairs/LTI descriptor systems is also related to a type of generalized Lyapunov matrix equation.
\begin{theorem}[\cite{MasKAS97}] \label{thm:impulsefree_reference}
Consider a pair $(E,A)$ with $E,A \in \mathbb R^{n,n}$. The pair is regular, of index at most one and asymptotically stable
if and only if there exists a nonsingular $V\in \mathbb R^{n,n}$ satisfying
\begin{equation} \label{LMIcharact}
V^\top A+A^\top V \prec 0 \quad \text{ and } \quad E^\top V=V^\top E \succeq 0.
\end{equation}
\end{theorem}

\section{Passivity}\label{sec:defpass}

Another important property for a dynamical system is a conservation property that is termed as \emph{passivity}, which means system does not generate energy~\cite{AndV73,LozBEM13}. 

\begin{definition}\label{def:pass}
The LTI system $(E,A,B,C,D)$ is called \emph{passive} if there exists a nonnegative scalar valued function $\mathcal V(x)$ such that
$\mathcal V(0)=0$, and the dissipation inequality
\begin{equation}\label{eq:def_pass}
\mathcal V(x(t_1))-\mathcal V(x(t_0)) \leq \int_{t_0}^{t_1} y(t)^\top  u(t)dt
\end{equation}
holds for all admissible $u$, $t_0$, and $t_1 \geq t_0$.

The inequality~\eqref{eq:def_pass} has a natural interpretation as an assertion that the increase  in internal energy of the system, as measured by $\mathcal V$, cannot exceed the total work done on the system. The function $\mathcal V(x)$ is called a storage function associated with the supply rade, $y^\top (t)u(t)$. 

If for all $t_1 > t_0$, the inequality~\eqref{eq:def_pass}  is strict, then the
system is called \emph{strictly passive}. 
\end{definition}

In general for a system $(E,A,B,C,D)$, the lengths of vectors $u(t)$ and $y(t)$ are different. However, note that in~\eqref{eq:def_pass}, 
the energy is defined via the inner product of the input and output vectors, $u(t)$ and $y(t)$, of the system. Hence for passivity analysis,  these vectors need to be of the same length.  Passivity of an LTI dynamical system is equivalent to positive realness~\cite{AndV73}, which is defined next.

\section{Positive real systems}\label{sec:defprsys}

To define positive real (PR) systems, throughout this section we
assume that the system~\eqref{eq2:sys} is regular.
The system~\eqref{eq2:sys} can be described by its \emph{transfer function}
$G(s):\C\rightarrow (\C \cup \{\infty\})^{m,m}$, defined by
\begin{equation}\label{tran_fun}
G(s):=C(sE-A)^{-1}B+D,\quad s \in \C.
\end{equation}
Conversely, given a rational function $G(s):\C\rightarrow (\C \cup \{\infty\})^{m,m}$, any
representation of $G(s)$ in the form~\eqref{tran_fun} is called a realization  of
$G(s)$. A realization is called \emph{minimal} if the matrices $A$ and $E$ are of smallest
possible dimension.
In this case the poles  of the transfer function $G(s)$ are exactly the eigenvalues of the pencil $zE-A$.

Positive realness is a well-known concept in system, circuit and
control theory. In control theory, PR systems play a significant role
in stability analysis~\cite{Pop73,AndV73}, see also~\cite{Jos89} and the references therein for applications.
The PR systems have been defined in several different ways in the literature;
see \cite{AndV73,Wen88,LozJ90,SunKS94,IonT87,HadB91,HuaIMS99}
for  standard linear systems, \cite{WanC96,FreJ04,LozBEM13} for continuous-time descriptor systems, and \cite{ZhaLX02} for continuous- and discrete-time descriptor systems.
We follow~\cite{SunKS94} and define the positive realness in the frequency domain as follows.
\begin{definition}\label{def:PR_SPR_ESPR}
The system~\eqref{eq1:sys} is said to be
\begin{enumerate}
\item  \emph{positive real (PR)} if its transfer function $G(s)$ satisfies
\begin{enumerate}
\item $G(s)$ has no pole in $\real{s} >0$, and
\item $G(s)+G(s)^* \succeq 0$ for all $s$ such that $\real{s} > 0$.
\end{enumerate}
\item \emph{strictly positive real (SPR)} if its transfer function $G(s)$ satisfies
\begin{enumerate}
\item $G(s)$ has no pole in $\real{s} \geq0$, and
\item $G(jw)+G(jw)^* \succ 0$ for $w \in [0,\infty)$.
\end{enumerate}
\item \emph{extended strictly positive real (ESPR)} if it is SPR and
$G(j\infty)+G(j\infty)^* \succ 0$.
\end{enumerate}
\end{definition}

Note that the condition $(a)$ in the definition of SPR
is equivalent to the system being asymptotically stable.
An asymptotically stable system~\eqref{eq2:sys} with a minimal realization
is passive (resp.\ strictly passive) if and only if it is PR (resp.\ ESPR).
For more details, we refer to~\cite{AndV73} and \cite[pp.~174-175]{DesV75}. 
Furthermore, ESPR implies SPR, which further implies PR.

Note also that $G(s)=C(sE-A)^{-1}B+D$ is a rational function and has a power series expansion about $s=\infty$
of the form
\begin{equation}\label{eq:series_expan}
G(s)=C(sE-A)^{-1}B+D=\sum_{i=-p}^{\infty} \frac{H_i}{s^i},
\end{equation}
where $H_i$ are real matrices of size $m$. If $s=\infty$ is not a pole of $G(s)$ (that is, when $E$ is invertible),
then $p=0$ and $G(\infty)=D=H_0$. This implies that
for a standard system $(I_n,A,B,C,D)$ with $D+D^\top  \succ 0$, the notion of SPR and ESPR are the same, because
$G(j\infty)+G(j\infty)^* \succ 0$ if and only if $D+D^\top  \succ 0$.
 In the descriptor case (that is, when $E$ is not invertible), then the order of the pole at $s=\infty$
is larger than or equal to one (that is, $p\geq 1$ in~\eqref{eq:series_expan}).
In this case, $G(\infty)$ (if it exists) is not necessarily equal to $D$~\cite{gillis2018finding}.

\section{Nearness problems for LTI systems}

In real-world applications, one uses mathematical models to simulate, control or optimize a system or process. This mathematical model is infinite dimensional (e.g., the determination of the electric or magnetic field associated with an electronic device) and is approximated by finite element or finite difference model~\cite{IdaB97}, or the model is non-linear, and a linearization is used to obtain a linear model. The model may also be obtained by a realization or system identification~\cite{CoePS99,GusS01}, or it may result from a model order reduction procedure~\cite{Ant05}. These mathematical models, therefore, are typically inexact and contain uncertainties. Thus it is vital to study the following question~\cite{DuLM13}: 
\begin{center}
{\it How robust is a property of a dynamical system 
under perturbations of the coefficient matrices?}
\end{center}
The property of a system is called \emph{robust} if it is preserved under arbitrary, but sufficiently small, perturbations to the system. This gives a motivation to define nearness problems (distance problems) for a system. 

The nearness problem for LTI systems consists of finding, for a given system, the nearest system (with respect to some prescribed norm) within a given class of systems. The distance between the nominal system and the closest system from the given class is typically called the \emph{radius} of the system property 
(e.g., stability, instability, passivity, non-passivity, controllability, observability).
For better understanding, we divide these nearness problems into two types.

\subsection{Type-I distances (from good to bad systems)}

Consider an LTI system
 $(E,A,B,C,D)$ with a property $\mathcal P$. The Type-I distance  looks the smallest perturbation in the system matrices $(E,A,B,C,D)$ such that the perturbed system does not possess the property $\mathcal P$. More precisely, for  a given system $\Sigma=(E,A,B,C,D)$ compute the distance (radius of $\mathcal P$)
 \begin{align}\label{def_nearness_eq1}
 r_{\not \mathcal P}(\Sigma) \; = \; \inf \Big\{
 { \|\Delta_\Sigma\|}: & \;\; \Delta_\Sigma=(\Delta_E,\Delta_A,\Delta_B,\Delta_C,\Delta_D)~\text{such\, that} \nonumber \\
  & \;\; \Sigma+\Delta_\Sigma \, \text{does\, not\, posses \,property}\, \mathcal P
 \qquad  \Big\},
\end{align}  
where $\|\cdot\| $ is some norm defined on the set of matrix quintuple $(E,A,B,C,D)$. This is useful in the robustness analysis of control systems. If the radius is small, then the original problem is more likely to be ill-conditioned or more sensitive to pertubations, and some remedial actions need to be taken~\cite{Hig88a}.

Such nearness problems for systems have been a topic of research in the numerical linear algebra community. For example, the distance to instability (stability radius problem), where a stable system is given and one looks for the smallest perturbation that makes the system unstable; see~\cite{Bye88,HinP86} for standard systems, and~\cite{ByeN93,DuLM13} for descriptor systems. 
Similarly, the distance to non-passivity (passivity radius problem) is the smallest perturbation that makes a passive system non-passive. The passivity radius for complex standard systems was computed in~\cite{OveV05}. This problem is closely related to the Hamiltonian matrix nearness problem~\cite{Tal04,SchT07,WanZKPW10,VoiB11}. 

 Some other related distance problems of Type-I are, e.g., matrix nearness problems~\cite{Hig88a}, the structured singular value
problem~\cite{PacD93}, the robust stability problem~\cite{Zho11}, the distance to bounded realness for Hamiltonian matrices~\cite{AlaBKMM11}, the nearest defective matrix~\cite{Wil84}, the distance to singularity~\cite{ByeHM98,GugLM16}, and the distance to controllability~\cite{Eisr84}.

\subsection{Type-II distances (from bad to good systems)}\label{sec:bad2good}

Such distances are complementary to Type-I distances.  
 Consider the system $(E,A,B,C,D)$ and a given system property $\mathcal P$. The Type-I distance  looks for the smallest perturbation in the system matrices $(E,A,B,C,D)$ such that the perturbed system has property $\mathcal P$. More precisely, for  a given system $\Sigma=(E,A,B,C,D)$, compute the distance
 \begin{align}\label{def_nearness_eq2}
 r_{ \mathcal P}(\Sigma)  \; =  \;  \inf\Big\{
 { \|\Delta_\Sigma\|}  : &  \; \; \Delta_\Sigma=(\Delta_E,\Delta_A,\Delta_B,\Delta_C,\Delta_D)~ \text{such\, that}\nonumber\\
  &  \; \; \Sigma+\Delta_\Sigma \, \text{has \,property}\, \mathcal P
 \qquad \Big\},
\end{align}  
where $\|\cdot\| $ is some norm defined on the set of matrix quintuple $(E,A,B,C,D)$.  

This kind of problem occurs in system identification, where one needs to identify a system with property $\mathcal P$ from observations. Measurements being subject to several perturbations, truncation and noise, it may happen that the identified system does not possess property $\mathcal P$. One then approximates the system by a nearby system with property $\mathcal P$ by introducing small perturbations to system matrices $(E,A,B,C,D)$. This provides a tool that can correct measurement errors arising from the identification step by intrinsically modifying the identified system~\cite{ONV13,AlaBKMM11}; see also Section~\ref{sec:appldiscretesys}. 

Typically, one has an estimate or even a bound for the error (approximation, truncation and noise) in the original measurements. Then  one tries to keep the perturbations in $(E,A,B,C,D)$ within those bounds. So from the application point of view it may not be necessary to determine the minimal perturbation that makes the system achieve property $\mathcal P$; a perturbation that stays within the range of the already committed measurements errors is sufficient. But from the system theoretical point of view it is interesting to find a value or a bound for the smallest perturbation $( r_{ \mathcal P}(\Sigma))$ that makes the system achieves property $\mathcal P$. In general, determining the minimal perturbation for system properties, such as stability or passivity, is very challenging. Instead, one uses (non-convex) optimization methods to estimate $ r_{ \mathcal P}(\Sigma)$, see, e.g., \cite{ONV13,GilS16,guglielmi2017matrix,noferini2020nearest} for distance to stability (when $\mathcal P= {\rm stability}$) and~\cite{FreJ07,Tal04,SchT07,AlaBKMM11,gillis2018finding} for distance to passivity (when $\mathcal P= {\rm passivity}$). Another closely related problem is that of finding the closest stable polynomial to a given unstable one~\cite{MosL91}.

As mentioned in~\cite{Hig88a} for matrices, the choice of norm in~\eqref{def_nearness_eq1} and~\eqref{def_nearness_eq2} is usually guided by the tractability of the nearness problem. The two most useful norms are the squared $2$-norm, $\sum_{i} {\|\Delta_i\|}_2^2$, and the squared Frobenius norm, $\sum_{i} {\|\Delta_i\|}_F^2$, on the set of matrix quintuples 
$(\Delta_E,\Delta_A,\Delta_B,\Delta_C,\Delta_D)$. Both norms are unitarily invariant and differentiable. 

In these notes, we mostly focus on the Frobenius norm to measure distances, since (i)~it is arguably one of the most popular norms used to measure distances,
and (ii)~it is strictly convex and a smooth function of the matrix entries hence makes the optimization problem easier to tackle. However, the algorithms  in the following chapters can easily be extended to any other smooth objective function, e.g., any (weighted) $\ell_p$ norm with $1 < p < +\infty$ (only the computation of the gradient of the corresponding objective function will change).


%

\section{Organization of the lecture notes}  

In these lectures notes, we review our recent works addressing the problem of computing various type-II distances. 

In Chapter~\ref{chap:prelmin}, we provide some preliminary background, namely, 
defining Port-Hamiltonian systems (Section~\ref{sec_porthamsys}) and  
dissipative Hamiltonian systems (Section~\ref{diss_hamilt_sys}) and their properties, 
briefly discussing matrix factorizations (Section~\ref{sec:matfac}), and 
describing the optimization methods  that we will use in these notes (Section~\ref{sec:optimization}). 

In Chapter~\ref{chap:contsys}, we present our approach to tackle the distance to stability for standard continuous LTI systems. The main idea is to rely on the characterization of stable systems as  dissipative Hamiltonian systems. 
We show how this idea can be generalized to compute the nearest $\Omega$-stable matrix, where the eigenvalues of the sought system matrix $A$ are required to belong a rather general set $\Omega$ (Section~\ref{sec:omegastab}). 
In Section~\ref{sec:minnomfeed}, 
we show how these ideas can be used to compute minimal-norm static feedbacks, that is, stabilize a system by choosing a proper input $u(t)$ that linearly depends on $x(t)$ (static-state feedback), or on $y(t)$ (static-output feedback).   

In Chapter~\ref{chap:posrealsys}, we present our approach to tackle the distance to passivity.  The main idea is to rely on the characterization of stable systems as  port-Hamiltonian systems. We also discuss in more details the special case of computing the nearest stable matrix pairs in Section~\ref{sec:nearestdessys}. 
 
In Chapter~\ref{chap:discrete}, we focus on discrete-time LTI systems. Similarly as for the continuous case, we propose a  parametrization that allows efficiently compute the nearest stable system (for a matrix in Section~\ref{sec:nsm}, or for a matrix pair in Section~\ref{sec:nsmp}), allowing to compute the distance to stability. 
In Section~\ref{sec:appldiscretesys}, we show how this idea can be used in data-driven system identification, that is, given a set of input-output pairs, identify the system $A$.


\chapter{Preliminaries}   \label{chap:prelmin}

In this chapter, we briefly recall important concepts that will be used throughout these notes.

\section{Port-Hamiltonian systems } \label{sec_porthamsys}
 
\emph{Port-Hamiltonian (PH)} systems generalize the classical Hamiltonian systems and
recently have received a lot attention in energy based modeling;
see~\cite{Sch06,Sch13,SchM95,SchM02,SchM13,Sch2014port,BeaMXZ17_ppt,GolSBM03} 
for some major references. Although PH systems may be formulated in a general framework, we will restrict ourselves to LTI input-state-output PH systems, which have the form
\begin{align}
\label{eq:phsystem}
\begin{split}
E\dot{x}(t) = & \; (J-R)Qx(t) + (F-P)u(t), \\
y(t) = & \; (F+P)^{\top} Qx(t) + (S+N)u(t),
\end{split}
\end{align}
where the following conditions must be satisfied:
\begin{itemize}

\item The matrix $Q \in \R^{n,n}$ is invertible, $E \in \R^{n,n}$, and  $Q^{\top}E=E^{\top}Q \succeq 0$. The function $x\rightarrow \frac{1}{2}x^{\top} Q^{\top}E x$ is the \emph{Hamiltonian} and
describes the energy of the system.
\item The matrix $J^{\top}=-J \in \R^{n,n}$
is the structure matrix that describes flux among energy storage elements.
\item The matrix $R \in \R^{n,n}$
with $R\succeq 0$ is a positive semidefinite (PSD) matrix, and is the dissipation matrix that describes the energy dissipation/loss in the system.
\item The matrices $F\pm P \in \R^{n,m}$ are the port matrices describing the manner in which energy enters and exits
the system.
\item The matrix $S+N$, with $0 \preceq S \in \R^{m,m}$ and $N^{\top}=-N \in \R^{m,m}$,  describes
the direct feed-through from input to output.
\item The matrices $R$, $P$ and $S$ satisfy
\[
K=\mat{cc}R &P\\P^{\top} &S \rix \succeq 0.
\]
\end{itemize}

We will refer to $K$ as the \emph{cost matrix} of the PH system because
it corresponds to the cost matrix of an
infinite horizon linear-quadratic optimal control problem. 
We note that this definition of
PH systems is slightly more restrictive than that of PH systems in \cite{BeaMXZ17_ppt},
where it is not required for the matrix $Q$ to be invertible. 

PH systems form an important modelling tool in almost all areas of system and control, in particular, in network-based modelling of multi-physics multi-scale systems. They result in robust systems that can be easily interconnected.

\begin{example}{\rm
Finite element modelling of the acoustic field in the interior of a car, see, e.g., \cite{MehS11,BeaMXZ17_ppt}, leads to (after several simplifications) a large-scale constant-coefficient differential-algebraic equation system of the form
\begin{equation*}
M \ddot{p} +D \dot{p}+Kp =B_1 u,
\end{equation*}
where $p$ is the coefficient vector associated with the pressure in the air and the displacements of the structure, $B_1 u$ is an external force, $M \succeq 0$ is  a mass matrix, $D\succeq 0$ is a damping matrix, and $K \succ 0$ is a stiffness matrix. Here, 
$M$ is only semidefinite since small masses were set to zero. The first-order approximation leads to a PH system 
\[
E\dot{x}=(J-R)Qx+Bu,\quad y=B^{\top}Qx,
\]
 where
\begin{eqnarray*}
&E=\mat{cc}M & 0\\0& I\rix,~J=\mat{cc}0 & -I \\ I & 0\rix,
~R=\mat{cc}D& 0\\0&0\rix,~x=\mat{c}\dot{p}\\p\rix, \\
&Q=\mat{cc}I & 0\\0 &  K\rix,~B=\mat{c}B_1 \\ 0\rix,~ P=0,~S+N=0.
\end{eqnarray*}
The Hamiltonian in this case is given by $\mathcal H (x)=\frac{1}{2}\mat{c}x_1\\x_2\rix \mat{cc}M & 0\\0& K \rix \mat{c}x_1 \\x_2\rix$.
}
\end{example}

\subsection{Properties of PH systems}

One of the major advantages of PH modelling is that system properties are encoded algebraically~\cite{BeaMXZ17_ppt,MehMW18}, and they are robust under structured perturbations~\cite{MehMS16,MehMS16b,MehMW18,MehV20,MehMW21}.
The algebraic structure of PH systems guarantees that the system is automatically stable;~see~\cite{MehMS16} for
standard PH systems and~\cite{gillis2018computing,MehMW18} for descriptor PH systems. We review some of the stability results of PH systems~\eqref{eq:phsystem} in Section~\ref{subsec:dhprop}.

Another important property of the PH sytems~\eqref{eq:phsystem}  is that they are always passive. Indeed, the Hamiltonian $\mathcal H(x)=\frac{1}{2}x^{\top}Q^{\top}Ex$ defines a storage function with the supply rate $u^{\top}y$. Clearly, $\mathcal H(x)$ is nonnegative since $Q^{\top}E \succeq 0$, and
\begin{eqnarray*}\label{eq1:dissineq}
\frac{d\mathcal H}{dt} = u^{\top}y-\mat{c}x\\u\rix^{\top}
\mat{cc}Q^{\top}RQ & Q^{\top}P\\ P^{\top}Q & S\rix \mat{c}x\\u\rix \leq u^{\top}y,
\end{eqnarray*} 
since $K=\mat{cc}R&P\\P^{\top} &S \rix \succeq 0$ and $Q \succ 0$. This implies that for any $t_0,t_1 \in \R$, with $t_1 > t_0$, the dissipation inequality
\begin{equation} \label{HHyu}
\mathcal H(x(t_1))-\mathcal H(x(t_0)) \leq \int_{t_0}^{t_1} y(t)^{\top} u(t)dt
\end{equation}
holds. Thus by Definition~\ref{def:pass}, the system~\eqref{eq:phsystem} is passive. Moreover, if $K \succ 0$, then  \eqref{eq:phsystem} is strictly passive. We note that the inequality~\eqref{HHyu} holds even when the system matrices in~\eqref{eq:phsystem} depend on $x$ or explicitly on time $t$, see~\cite{MasSB92}, or when they are defined as linear operators acting on infinite-dimensional spaces~\cite{JacZ12}.

The PH system has many other properties; it is robust under structured  perturbations~\cite{MehMS16,MehMS16b,MehV20}. PH systems constitute a class of systems that is closed under power-conserving interconnections. This means the port connected PH systems produce an aggregate system that must also be PH. This aggregate system hence will be guaranteed to be both stable and passive~\cite{Kle13}. The model reduction of PH systems via Galerkin projection yields (smaller) PH systems~\cite{PolS10,GugPBS12}.

The PH modelling of dynamical systems is thus compelling; it encodes underlying physical principles as conservative laws directly into the structure of the system model. It is shown in~\cite{BeaMX15_ppt} that  every minimal passive system is equivalent to a PH system. It has also been recently shown to how robustly transform a passive and stable system into a robust PH system~\cite{MehV20,CheMH19}.

\section{Dissipative Hamiltonian systems} \label{diss_hamilt_sys}

 A \emph{dissipative Hamiltonian (DH)} descriptor system in the LTI case can be expressed as
\begin{equation}\label{eq:dhdescriptor}
E\dot{x}=(J-R)Qx,
\end{equation}
where $J=-J^{\top} \in \mathbb R^{n,n}$, $R \in \mathbb R^{n,n}$ with $R=R^{\top} \succeq 0$, and $E,Q \in \mathbb R^{n,n}$ such that $Q$ is invertible and satisfies  $E^{\top}Q=Q^{\top}E \succeq 0$. 
These systems arise in energy-based modeling of dynamical systems; see, e.g.,~\cite{BeaMXZ17_ppt,GolSBM03,Sch13,MehMW18} and are a special case of port-Hamiltonian descriptor systems  that were presented in Section~\ref{sec_porthamsys}. 

DH systems play an important role in a variety of applications~\cite{BeaMXZ17_ppt,MehMW18}. We briefly discuss the following examples from the literature.

\begin{example}{\rm
A simple RLC network (see, e.g.,~\cite{BeaMXZ17_ppt,FreJ11}) can be modelled by a DH descriptor system of the form
\begin{equation}\label{eq:examdh1}
\underbrace{\mat{ccc}G_c\mathcal CG_c^{\top} & 0&0 \\ 0 & \mathcal L & 0 \\ 0&0&0\rix}_{:=E}
\mat{c}\dot{v}\\ \dot{i}_l(t) \\  \dot{i}_v(t) \rix =
\underbrace{\mat{ccc} -G_r\mathcal R^{-1}G_r^{\top} & -G_l & -G_v \\ G_l^{\top} & 0 & 0\\ G_v^{\top} & 0&0 \rix}_{:=J-R}
\mat{c}{v}\\ {i}_l(t) \\  {i}_v(t) \rix, 
\end{equation}
with real symmetric matrices $\mathcal R \succ 0$, $\mathcal C \succ 0$,
$\mathcal L \succ 0$ incorporating the resistances of the resistors, capacitances of the capacitors, and inductances between the inductors, respectively. Here, $J$ and $-R$ are defined as the skew-symmetric and symmetric parts, respectively, of the matrix on the right-hand side of~\eqref{eq:examdh1}. 
The matrix $G_v$ is of full rank, and the subscripts $r,c,l,v$, and $i$ refer to edge quantities corresponding to the resistors, capacitors, inductors, voltage sources, 
and current sources, respectively, of the given $RLC$ network. We see that in this example, we have a form~\eqref{eq:dhdescriptor} system with the matrix $Q$ being the identity, $E=E^{\top}\succeq 0$, $J^{\top}=-J$, and $R\succeq 0$.
}
\end{example}

\begin{example}{\rm
Space discretization of the Stokes or Oseen equation in fluid dynamics~(see, e.g.,~\cite{EmmV13}) leads to a DH system $E \dot{x}=(J-R)Qx$, with
\[
E=\mat{cc}M & 0\\0 & 0\rix, \quad J=\mat{cc}0 & B \\ -B^{\top} & 0\rix,\quad
R=\mat{cc}A & 0 \\ 0 & 0\rix, \quad Q=I,
\]
where $A$ is a PSD discretization of the negative Laplace operator, $B$ is a discretized gradient, and $M$ is  a positive definite mass matrix. The matrix $B$ may be rank deficient, in which case the system is singular. 
}
\end{example}

Examples with singular $Q$ also arise in applications. The singular DH systems occur as a limiting case or when redundant system modelling is used, see~\cite{AstOV10,FujSS12}.

\subsection{Properties of DH systems}\label{subsec:dhprop}

To understand the algebraic properties of DH systems~\eqref{eq:dhdescriptor}, we analyze the matrix pairs $(E,(J-R)Q)$, 
where $J^{\top}=-J$, $R\succeq 0$, and $Q$ is invertible with $E^{\top}Q=Q^{\top}E \succeq 0$. We refer to such matrix pairs as \emph{DH matrix pairs}. 

The assumption that $Q$ is invertible and $Q^{\top}E$ is PSD results in many linear algebra properties of the DH systems. We note that a more general class of DH systems with singular $Q$ is investigated in~\cite{MehMW18}. 
 In the following, we review only those properties of DH systems from~\cite{MehMW18,gillis2018computing}, which will help solve distance problems related to the stability of LTI systems. 

A DH matrix pair is not necessarily regular, as the following example shows.
\begin{example}\label{ex:ex2}{\rm
The pair $(E,(J-R)Q)$ with
\[
E=\begin{bmatrix}1 & 0&0\\0& 1&0\\0&0&0\end{bmatrix}, \;
J=\begin{bmatrix}0 & 2&0\\-2&0&0\\0&0&0\end{bmatrix}, \;
R=\begin{bmatrix} 1&0&0\\0&1&0\\0&0&0\end{bmatrix}, \;
Q=\begin{bmatrix}1 &0&0\\0 & 1&0\\0&0&1\end{bmatrix},
\]
is a DH matrix pair, but $(E,(J-R)Q)$ is singular, since
$\operatorname{det}(zE-(J-R)Q)\equiv 0$.
}
\end{example}
It is easy to see that
if the matrices $E$ and $A$ have a common nullspace,
then the pair $(E,A)$ is singular, but the converse is in general not  true, see, e.g., \cite{ByeHM98}.
However, for singular DH matrix pairs, the property that $Q$ is invertible guarantees that the converse also holds, see~\cite{MehMW21,PraS21a} for more on the singularity of DH systems.
\begin{lemma} \label{Lemma1}
Let $(E,(J-R)Q)$ be a DH matrix pair. Then $(E,(J-R)Q)$ is singular if and only if
\[
\operatorname{null}(E)\cap\operatorname{null}((J-R)Q) \neq \emptyset.
\]
\end{lemma}
\begin{proof} The direction $(\Leftarrow)$  is immediate.  For the other direction, let $(E,(J-R)Q)$ be singular, that is, $\operatorname{det}(z E-(J-R)Q) \equiv 0$. Let $\lambda \in \mathbb C$ be such that $\real{(\lambda)} > 0$ and let  $x \in \mathbb C^n\setminus \{0\}$ be such that
\begin{equation}\label{eq:singular_DHpair_1}
(J-R)Qx=\lambda Ex.
\end{equation}
Since $Q$ is nonsingular, we have that $Qx\neq 0$, and we can multiply with $(Qx)^H$ from the left to obtain
\begin{equation}\label{eq:singular_DHpair_2}
x^HQ^{\top}JQx-x^HQ^{\top}RQx=\lambda x^HQ^{\top}Ex,
\end{equation}
where $x^H$ denotes the complex conjugate of  a vector $x$. This implies that $x^HQ^{\top}Ex=0$, because otherwise from~\eqref{eq:singular_DHpair_2} we would have
\[
\real{(\lambda)}=-\frac{x^HQ^{\top}RQx}{x^*Q^{\top}Ex} \leq 0,
\]
since  $Q^{\top}RQ \succeq 0$ (as $R \succeq 0$), and $Q^{\top}E \succeq 0$. But this is a contradiction to the fact that
$\real{(\lambda)} > 0$. Therefore $x^HQ^{\top}Ex=0$ and also $Q^{\top}Ex=0$, and this implies that
$Ex =0$ as $Q$ is invertible. Inserting this in~\eqref{eq:singular_DHpair_1}, we get
$(J-R)Qx=0$, i. e. $0\neq x\in \operatorname{null}(E)\cap\operatorname{null}((J-R)Q)$.
\end{proof}
Lemma~\ref{Lemma1} gives a necessary and sufficient condition for a DH matrix pair to
be singular. However, if the dissipation matrix $R$ is positive definite, regularity is assured, as shown in the following result.
\begin{corollary}\label{cor:DJ_regular}
Let $(E,(J-R)Q)$ be a DH matrix pair. If  $R$ is positive definite, then the pair is regular.
\end{corollary}
\begin{proof} By Lemma~\ref{Lemma1}, a necessary condition for the pencil $\lambda E-(J-R)Q$
to be singular is that neither $E$ nor $(J-R)Q$ is invertible. Thus the result follows immediately
by the fact that if $R \succ 0$ in a DH matrix pair $(E,(J-R)Q)$, then $(J-R)Q$ is invertible. Indeed,
suppose there exists $x\in \mathbb C\setminus \{0\}$ such that $(J-R)Qx=0$, then we have
$x^HQ^{\top}(J-R)Qx=0$. This implies $x^HQ^{\top}RQx=0$, since $J^{\top}=-J$. 
Since  $Q$ is invertible and thus $Qx\neq 0$, this  is a contradiction to the assumption that $R$ is positive definite.
\end{proof}

In the following lemma, which is the matrix pair analogue of~\cite[Lemma 2]{GilS16},  we localize the finite eigenvalues of a DH matrix pair.
\begin{lemma}\label{lem:DH_pair_prop}
Let $(E,(J-R)Q)$ be a regular DH matrix pair and let $L(z):=zE-(J-R)Q$. Then the following statements hold.
\begin{enumerate}
\item [1)] All finite eigenvalues of the pencil $L(z)$ are in the closed left half of the complex plane.
\item [2)] The pencil $L(z)$ has a  finite eigenvalue $\lambda$ on the imaginary axis if and only if $RQx=0$ for some eigenvector
$x$ of the pencil $zE-JQ$ associated with  $\lambda$.
\end{enumerate}
\end{lemma}
\begin{proof}
Let $\lambda \in \mathbb C$ be an eigenvalue of $z E-(J-R)Q$ and let $x\in \mathbb C^n\setminus\{0\}$
be such that
\begin{equation}\label{eqn:DH_eig_1}
(J-R)Qx=\lambda Ex.
\end{equation}
Multiplying~\eqref{eqn:DH_eig_1} by $x^HQ^{\top}$ from the left, we get
\begin{equation}\label{eqn:DH_eig_2}
x^HQ^{\top}(J-R)Qx=\lambda\, x^HQ^{\top}Ex.
\end{equation}
Note that $x^HQ^{\top}Ex \neq 0$, because if $x^HQ^{\top}Ex=0$, then we have
$Q^{\top}Ex=0$ as $Q^{\top}E \succeq 0$, and thus $Ex =0$ as $Q$ is invertible. Using this in~\eqref{eqn:DH_eig_1},
we have $(J-R)Qx=0$. This implies that $x \in \operatorname{null}(E)\cap \operatorname{null}((J-R)Q)$ which is,
by Lemma~\ref{Lemma1}, a contradiction to the regularity of the pair $(E,(J-R)Q)$.

Thus, by~\eqref{eqn:DH_eig_2} we have
\begin{equation}\label{eqn:DH_eig_3}
\real{(\lambda)}=-\frac{x^HQ^{\top}RQx}{x^HQ^{\top}Ex} \leq 0,
\end{equation}
because $Q^{\top}RQ \succeq 0$ as $R\succeq 0$, and $x^HQ^{\top}Ex >0$. This completes the proof of 1).

In the proof of 1), if  $\lambda \in i\mathbb R$, then from~\eqref{eqn:DH_eig_3} it follows that
$x^HQ^{\top}RQx=0$. This implies that $RQx=0$, since $R\succeq 0$. Using this in~\eqref{eqn:DH_eig_1} implies that
$(\lambda E-JQ)x=0$.

Conversely, let $\lambda \in i\mathbb R$ and $x\in \mathbb C^n\setminus\{0\}$ be such that
$RQx=0$ and $(\lambda E-JQ)x=0$. Then this trivially implies that $\lambda$ is also an
eigenvalue of the pencil $\lambda E-(J-R)Q$ with eigenvector $x$. This completes the proof of 2).
\end{proof}
Making use of these preliminary results, we have the following stability characterization.
\begin{theorem}\label{thm:stable_semidef_R}
Every regular DH matrix pair $(E,(J-R)Q)$ of index at most one is stable.
\end{theorem}
\begin{proof} In view of Lemma~\ref{lem:DH_pair_prop}, to prove
the result it is sufficient to show that if $\lambda \in i\mathbb R$ is an
eigenvalue of the pencil $z E-(J-R)Q$, then $\lambda$ is semisimple.

Let us suppose that $\lambda \in i\mathbb R$ is a defective eigenvalue of the pencil $zE-(J-R)Q$
and the set $\{x_0,x_1,\ldots,x_{k-1}\}$ forms a Jordan chain of length $k$ associated with $\lambda$, see e.~g. \cite{GohLR82}, i.~e. $x_0 \neq 0$ and
\begin{eqnarray}
(\lambda E-(J-R)Q)x_0=0,\quad (\lambda E-(J-R)Q)x_1&=&Ex_0,\nonumber\\
(\lambda E-(J-R)Q)x_2&=&Ex_1,\nonumber\\
\vdots \label{thm:stable_DH_1}\\
 (\lambda E-(J-R)Q)x_{k-1}&=&Ex_{k-2}.\nonumber
\end{eqnarray}
Note that by Lemma~\ref{lem:DH_pair_prop}, we have that $(\lambda E-(J-R)Q)x_0=0$ implies that
\begin{equation}\label{thm:stable_DH_2}
(\lambda E-JQ)x_0=0\ \mbox{\rm and }\ RQ x_0=0.
\end{equation}
By~\eqref{thm:stable_DH_1}, $x_0$ and $x_1$ satisfy
\begin{equation}\label{thm:stable_DH_3}
(\lambda E-(J-R)Q)x_1=Ex_0.
\end{equation}
Multiplying~\eqref{thm:stable_DH_3} by $x_0^HQ^{\top}$ from the left, we obtain
\[
x_0^H(\lambda Q^{\top}E-(Q^{\top}JQ-Q^{\top}RQ))x_1=x_0^HQ^{\top}Ex_0.
\]
This implies that
\begin{eqnarray}\label{thm:stable_DH_4}
-x_1^H(\lambda Q^{\top}E -Q^{\top}JQ)x_0 +x_1^HQ^{\top}RQ x_0 = x_0^HQ^{\top}Ex_0,
\end{eqnarray}
where the last equality follows by the fact that $Q^{\top}E=E^{\top}Q$, $J^{\top}=-J$, and $R^{\top}=R$.
Thus, by using~\eqref{thm:stable_DH_2} in~\eqref{thm:stable_DH_4}, we get
$x_0^HQ^{\top}Ex_0=0$. But this implies that $Q^{\top}Ex_0=0$ as $Q^{\top}E \succeq 0$ and
$Ex_0 =0$ as $Q$ is invertible. Since $x_0$ is an eigenvector of the pencil $zE-(J-R)Q$  to
$\lambda$, $Ex_0=0$ implies that $(J-R)Qx_0=0$. This means that
$0\neq x_0  \in \operatorname{null}(E)\cap \operatorname{null}((J-R)Q)$, which contradicts the regularity
of the pair $(E,(J-R)Q)$. Therefore there does not exist a vector $x_1 \in \mathbb C^n$ satisfying~\eqref{thm:stable_DH_1}. Hence
$\lambda$ is semisimple.
\end{proof}
We note that the proofs of Lemma~\ref{lem:DH_pair_prop} and Theorem~\ref{thm:stable_semidef_R}
highly depend on the invertibility of $Q$. In fact, we have used the fact that
$(Q^{\top}E,Q^{\top}(J-R)Q)$ is regular, because $Q$ invertible implies that
$(E,(J-R)Q)$ is regular if and only if $(Q^{\top}E,Q^{\top}(J-R)Q)$ is regular.
If  $Q$ is a singular matrix and $Q^{\top}E=E^{\top}Q \succeq 0$, then the pair
$(Q^{\top}E,Q^{\top}(J-R)Q)$ is always singular, but the pair $(E,(J-R)Q)$ may be regular.

For example, consider the matrices
\[
E=\mat{cc}1 &0\\0&2\rix,\quad J=\mat{cc}0 &1\\-1&0\rix,\quad R_1=\mat{cc}1 &0\\0&1\rix,\quad R_2=\mat{cc}1 &0\\0&0\rix,
\]
and the singular matrix $Q=\mat{cc}0 &0\\0&1\rix$. Then the matrix pair $(E,(J-R_1)Q)$ is regular, of index zero, and
has two simple eigenvalues $0$ and $-0.5$, which implies that $(E,(J-R_1)Q)$ is stable. On the other hand
the matrix pair $(E,(J-R_2)Q)$ is regular, of index zero, and has a defective eigenvalue of multiplicity two at the origin, which implies
that it is not stable. This shows that the invertibility of $Q$ is necessary
in Lemma~\ref{lem:DH_pair_prop} and Theorem~\ref{thm:stable_semidef_R}.

However, parts of Lemma~\ref{lem:DH_pair_prop} and Theorem~\ref{thm:stable_semidef_R}
also hold for singular $Q$ with an extra assumption on the eigenvectors of the pencil $zE-(J-R)Q$, as stated in Theorem~\cite[Theorem 2]{gillis2018computing}. For more properties of DH matrix pairs with singular $Q$, we refer to~\cite{MehMW18}.
 
\section{Matrix factorization}  \label{sec:matfac}

Matrix factorization is a central theme in numerical linear algebra (NLA), where it takes many different forms, 
e.g., QR decomposition, Cholesky decomposition, eigenvalue decomposition, the singular value decomposition, to cite a few.  
It is the key to some of the most fundamental NLA problems, such as solving linear systems of equations. Although these problems are classical, and old, they are still an active and important research topic. In the last 20 years, they have become a central tool for linear dimensionality reduction in many data analysis and machine learning tasks; 
see, e.g., 
\cite{singh2008unified}, 
\cite{Udell2016lowrank}, 
\cite{Markovsky2019},  
and \cite[Chapter 1]{gillis2020book}, 
and the references therein. 
 
 In these lecture notes, we will encounter very particular forms of matrix factorizations, in particular the factorization $A = (J-R)Q$ 
 where 
 $J$ is antisymmetric, that is, $J = -J^{\top}$, 
 $R$ is PSD, that is, $R \succeq 0$, and 
 $Q$ is positive definite, that is, $Q \succ 0$. 
This factorization takes its roots from  port-Hamiltonian systems, as described in the previous section. 


  \section{Optimization} \label{sec:optimization}

 In these lecture notes, we will rely mostly on four optimization strategies, briefly described in the next sections. 
 First, let us recall what is an optimization problem. To define an optimization problem, one first needs to choose 
 \begin{itemize} 

\item the variables, $x \in \mathbb{R}^n$, that are the degrees of freedom of the given problem, 

\item the feasible set, $\mathbb{E} \subseteq \mathbb{R}^n$,  that restrict the values of the variables and is typically defined via constraints, and 
 
 \item  the objective function, $f(x) : \mathbb{R}^n \mapsto \mathbb{R}$, that allows to compare feasible solutions:  $x_1  \in \mathbb{E}$ is better than $x_2  \in \mathbb{E}$ if $f(x_1) \leq f(x_2)$ when $f$ is minimized. 
 
\end{itemize} 
 Given these three objects, the corresponding optimization problem is written as 
 \[
 \min_{x \in \mathbb{E}} f(x). 
 \]

A famous example is the knapsack problem: you are given a set of $n$ objects, each one with a value $c_i$, and a weight $w_i$. You are allowed a fixed maximum weight, $W$, and want to maximize the value within your knapsack. 
The variables are $x \in \{0,1\}^n$, where $x_i=1$ if object $i$ is put in the knapsack, and $x_i=0$ otherwise. The optimization problem is written as 
\[
\max_{x \in \{0,1\}^n} \sum_{i=1}^n c_i x_i 
\quad \text{such that} \quad 
\sum_{i=1}^n w_i x_i  \leq W. 
\] 
Note that here the goal is to maximize the objective function. However, one could equivalently minimize $-\sum_{i=1}^n c_i x_i$. Note also that the feasible set is given by 
\[
\mathbb{E} = \left\{ x \in \{0,1\}^n \ \Big| \ \sum_{i=1}^n w_i x_i  \leq W \right\}, 
\]
and is a discrete set. 


\subsection{Block coordinate descent (BCD) methods} 

Let us consider the following optimization problem 
\begin{equation} \label{eq:BCD}
\min_{x_1 \in \mathbb{E}_1, 
x_2 \in \mathbb{E}_2, \dots, 
x_p \in \mathbb{E}_p} \quad f(x_1,x_2,\dots,x_p) , 
\end{equation} 
where $\mathbb{E}_i \subseteq \mathbb{R}^{n_i}$ for $i \in [p] = \{1,2,\dots,p\}$ are the feasible sets for the blocks of variables. 
The variable $x = (x_1,x_2,\dots,x_p) \in \mathbb{R}^n$ is split into $p$ blocks of variables.

\paragraph{Exact BCD} In exact BCD methods, each block of variables, $x_i \in \mathbb{E}_i$ for $i \in [p]$, is optimized exactly and alternatively, while the values of the other blocks are fixed; see Algorithm~\ref{algo:EBCD}. 
\algsetup{indent=2em} 
\begin{algorithm}[ht!]
\caption{Exact block coordinate descent \label{algo:EBCD}}
\begin{algorithmic}[1] 
\REQUIRE An optimization problem of the form~\eqref{eq:BCD}. 
\ENSURE An approximate solution to~\eqref{eq:BCD}.   
    \medskip  
\STATE Choose an initial point $x^{(0)} \in \mathbb{R}^n$.
\FOR {$k$ = 1, 2, \dots }  
  \FOR {$i = 1, 2, \dots, p$ }  

\STATE Update 
\[ 
x_i^{(k)} 
= \argmin_{y \in \mathbb{E}_i} 
f\left(x_{1}^{(k)}, \dots, 
x_{i-1}^{(k)}, 
y, 
x_{i+1}^{(k-1)}, 
\dots, x_p^{(k-1)}\right) 
\]  
  
\ENDFOR 
\ENDFOR 
\end{algorithmic}  
\end{algorithm}   

Note that Algorithm~\ref{algo:EBCD} cyclically updates the variables, but other strategies are possible (e.g., randomly shuffle the order of the blocks of variables before each outer loop); see~\cite{wright2015coordinate}. 

Problem~\eqref{eq:BCD} can be non-convex in general, and hence might have many local minima. In non-convex optimization when $f$ is differentiable, convergence of algorithms are usually studied in terms of convergence to a first-order stationary point. The point $x$ is a first-order stationary point if the first-order approximation of the function around $x$, 
 \[
 f(x+\delta x) \approx f(x) + \nabla f(x)^\top \delta x, 
 \] 
is larger than $f(x)$ in the domain, that is,  $\nabla f(x)^\top \delta x \geq 0$ for any feasible direction $\delta x$, that is, for any $\delta x$ pointing inside the domain (including $\delta x \rightarrow 0$ which might be tangent to the domain).  
For example, if there is no constraint, 
$\mathbb{E} = \mathbb{R}^n$, all directions are feasible and these conditions reduce to $\nabla f(x) = 0$. In the general constrained case, these conditions are referred to as Karush-Khun-Tucker (KKT) optimality conditions, see, e.g., 
\cite{wright1999numerical} for more details.

Convergence of Exact BCD methods are guaranteed under rather strong conditions, as stated by the following theorem. 
\begin{theorem}\cite[Proposition 2.7.1]{Bertsekas99, Bertsekas99b} \label{th:nblocks}
The limit points of the iterates of an exact
BCD algorithm are stationary points provided that the following  conditions hold:
\begin{enumerate}

\item the objective function is continuously differentiable,  

\item each block of variables is required to belong to a closed convex set,

\item the minimum computed at each iteration for a given block of variables is uniquely attained, and 

\item the objective function values in the interval between all iterates and the next (which is obtained by updating a single block of variables) is monotonically decreasing. 

\end{enumerate}
The condition 4 can be dropped if each block of variables belongs to a convex and compact set. 

The order in which the blocks are updated is arbitrary, as long as each block is updated at least once every $K$ iterations, where $K$ is a fixed constant; this is referred to as the essentially cyclic block update. 
\end{theorem}

\paragraph{Inexact BCD} In many situations, exact optimization of each block is either impossible (the subproblem does not admit a closed form) or too costly (we will encounter examples later on). It is therefore appropriate to apply a few iterations of a cheap iterative method that decreases the objective function, 
such as a gradient step (see the next section). 
We refer the interested reader to the 
proximal alternating linearized minimization (PALM) algorithm~\cite{bolte2014proximal}, and to the block successive upper-bound minimization (BSUM) framework~\cite{razaviyayn2013unified}, for important examples of inexact BCD schemes, with strong convergence guarantees.

\subsection{First-order methods} 

Given the problem 
 \begin{equation} \label{eq:minf}
 \min_{x \in \mathbb{E}} f(x),  
 \end{equation}
 where $f$ is differentiable, a workhorse approach to tackle it is projected gradient method (PGM); see Algorithm~\ref{algo:gradient}. 
\algsetup{indent=2em} 
\begin{algorithm}[ht!]
\caption{Projected gradient descent (PGM) \label{algo:gradient}}
\begin{algorithmic}[1] 
\REQUIRE An optimization problem of the form~\eqref{eq:minf} where $f$ is differentiable. 
\ENSURE An approximate solution to~\eqref{eq:minf}, $x^{(k)}$.   
    \medskip  
\STATE Choose an initial point $x^{(0)} \in \mathbb{R}^n$.
\FOR {$k$ = 0, 1, 2, \dots }  

\STATE Update 
\[ 
x^{(k+1)}  
= \mathcal{P}_{\mathbb{E}} 
\left( 
 x^{(k)} - \gamma_k \nabla f \left( x^{(k)} \right)  
 \right) , 
\] 
where  $\mathcal{P}_{\mathbb{E}}(x) = \argmin_{y \in \mathbb{E}} \| x - y \|_2$ is the Euclidean projection onto $\mathbb{E}$, and 
$\gamma_k$ is an appropriate step size, typically chosen such that $f\left( x^{(k+1)} \right)  \leq f \left( x^{(k)} \right)$.  
\ENDFOR 
\end{algorithmic}  
\end{algorithm}

Some remarks are in order: 
\begin{itemize}

\item The projection onto the feasible set, $\mathcal{P}_{\mathbb{E}}(\cdot)$, might not be easy to compute. 
If it is not possible/computationally too heavy, a possible approach is to put some constraints defining $\mathbb{E}$  in the objective function as penalties; 
see, e.g.,~\cite{wright1999numerical}.  

\item Computing the step sizes might be tricky. However, if $f$ is continuously differentiable, there always exists sufficiently small step sizes that guarantee the decrease of the objective function (proving this is a simple exercise, using the first-order Taylor expansion of $f$ around the current iterate). 

\end{itemize}

If $f$ is Lipschitz continuous, that is, 
$\| \nabla f(x) - \nabla f(y) \|_2 \leq L \| x - y \|_2$ for some $L$ for all $x,y \in \mathbb{E}$ where $\mathbb{E}$ is a convex set, 
the step size $\gamma_k = \frac{1}{L}$ guarantees the objective function to decrease. In fact, the descent lemma for a Lipschitz continuous function guarantees that 
\[
f(y) 
\; \leq \; 
g(y) = 
f(x) + \nabla f(x)^\top (y-x) 
+ \frac{L}{2}  \| y-x \|_2^2. 
\] 
We have that 
\begin{align}
y^* & = 
\argmin_{y \in \mathbb{E}} g(y) \nonumber \\ 
& = \mathcal{P}_{\mathbb{E}} 
\left( 
 x - \frac{1}{L} \nabla f \left( x \right)  
 \right), \label{eq:projLip}
\end{align} 
and hence $f(y^*)  \leq g(y^*) \leq g(x) = f(x)$. 

\begin{remark} 
The equality in~\eqref{eq:projLip} follows from the following two  facts: 
\begin{enumerate}
\item the optimal unconstrained solution is given by 
\begin{align*}
\argmin_{y} g(y) & = 
\argmin_{y} f(x) + \nabla f(x)^\top (y-x) 
+ \frac{L}{2}  \| y-x \|_2^2  \\  
& = x - \frac{1}{L} \nabla f \left( x \right), 
\end{align*} 
which follows by setting the gradient of $g$ to zero, since $g$ is a simple quadratic function, whose Hessian is a scaling of the  identity matrix. 

\item The function $g$ is isotropic (that is, the level sets are spheres around the unconstrained solution), as its Hessian is the identity matrix, and hence the optimal solution of $\min_{y \in \mathbb{E}} g(y)$ is given by the projection of the unconstrained solution, $x - \frac{1}{L} \nabla f \left( x \right)$, onto the feasible set. 

\end{enumerate}  
\end{remark}

Interestingly, in the unconstrained case, we can quantify the decrease as follows  
\begin{align*}
f(y^*) \leq g(y^*) 
& = 
g \left( x - \frac{1}{L} \nabla f \left( x \right)  \right)  \\ 
& = 
f(x) - \frac{1}{L} \nabla f(x)^\top  \nabla f \left( x \right) 
+ \frac{L}{2}  \left\| \frac{1}{L}  \nabla f ( x ) \right\|_2^2 \\ 
& = 
f(x) - \frac{1}{2L} \| \nabla f ( x ) \|_2^2. 
\end{align*}

\paragraph{Convex problems}  

If $f$ is convex and Lischitz continuous, 
PGM is guaranteed to decrease the objective function values at a rate $\mathcal{O}(1/k)$~\cite{Nes04}.  

It turns out PGM can be accelerated, to achieve an optimal rate  of $\mathcal{O}(1/k^2)$; it is optimal among methods only using the first-order information, that is, the gradient, at each iteration, and under only the convexity and Lipschitz continuity of $f$. 
This is achieved by introducing another sequence of iterates,  and ``pushing'' the iterates further in the descent direction; 
see Figure~\ref{extrapol} for an illustration.  
\begin{figure}[ht!]
\begin{center}
\includegraphics[width=0.7\textwidth]{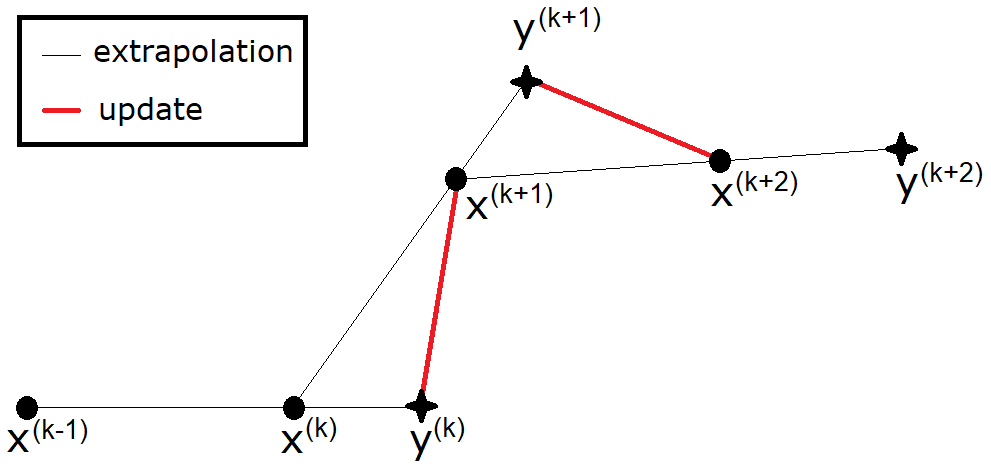}
\caption{Illustration of the use of extrapolation/inertia/momentum to accelerated the convergence of a sequence of iterates $\left\{ x^{(k)} \right\}_{k=0,1,2,\dots}$.  \label{extrapol}}
\end{center}
\end{figure}
This is known  in the literature as adding momentum, inertia or extrapolation. 
Algorithm~\ref{algo:fgm} provides a pseudocode for such a method, referred to as a fast gradient method (FGM). Note that the extrapolated sequence, 
$\{y^{(k)}\}$, might not be feasible. 
\algsetup{indent=2em} 
\begin{algorithm}[ht!]
\caption{Fast gradient method (FGM)~\cite{nes83, Nes04} \label{algo:fgm}}
\begin{algorithmic}[1] 
\REQUIRE An optimization problem of the form~\eqref{eq:minf} where $f$ is differentiable.  
\ENSURE An approximate solution to~\eqref{eq:minf}, $x^{(k)}$.   
    \medskip  
\STATE Choose an initial point $x^{(0)} \in \mathbb{R}^n$. 
Let $y^{(0)} = x^{(0)}$. 
\FOR {$k$ = 0, 1, 2, \dots }  

\STATE Update: \hspace{0.63cm}    $x^{(k+1)}   = 
 \mathcal{P}_{\mathbb{E}} 
\left( y^{(k)} - \gamma_k \nabla f \left( y^{(k)} \right)  \right)$, 
where $\gamma_k$ is a step size.  
\STATE Extrapolate: $y^{(k+1)} = x^{(k+1)} + \beta_k \left( x^{(k+1)} - x^{(k)} \right)$, 
 where $\beta_k$ is the extrapolation parameter, which can for example be chosen 
as 
$\beta_k = \frac{1-\alpha_k}{\alpha_{k+1}}$ where 
$\alpha_{k+1} = \frac{1 + \sqrt{1+ 4 \alpha_k^2}}{2}$ for some $\alpha_0 \in (0,1)$.  
\ENDFOR 
\end{algorithmic}  
\end{algorithm}

If $f$ is also strongly convex, that is, there exists a constant $\mu > 0$ such that 
\[
f(y) 
\; \geq  \; 
g(y) = 
f(x) + \nabla f(x)^\top (y-x) 
+ \frac{\mu}{2}  \| y-x \|_2^2, 
\] 
then the above rates become linear, namely 
$\mathcal{O}\left( \left( \frac{\kappa +1}{\kappa - 1} \right)^k \right)$ for PGM where $\kappa = \frac{L}{\mu} > 1$ is the conditioning of $f$, 
and 
$\mathcal{O}\left( \left( \frac{\sqrt{\kappa} +1}{\sqrt{\kappa} - 1} \right)^k \right)$ for FGM, 
so that FGM also provides a significant acceleration. 

Note however that FGM does not guarantee the decrease of the objective function at each iteration, and restarting strategies (that is, restarting the extrapolation sequence, $y^{(k)}$, and taking a standard gradient step)  might be useful to further accelerate convergence; 
see, e.g.,~\cite{o2015adaptive}.

\paragraph{Non-convex problems} 

In the non-convex setting, acceleration via extrapolation can also be used. It has first been used extensively as a heuristic acceleration, without theoretical guarantees, and more recently with convergence guarantees; 
see~\cite{xu2013block, xu2017globally, hien2019extrapolNMF, hien2020inertial} 
and the references therein. However, the variants with theoretical guarantees typically converge slower in practice, as they do not allow a very aggressive extrapolation strategy. 
In this lecture notes, we will use it as a heuristic, with a restarting procedure which guarantees the objective function to decrease at each step.

\subsection{Semidefinite programming} 

A semidefinite program (SDP)~\cite{Vandenberghe1996} has the form 
\begin{align*} 
\min_{X \succeq 0}  
& \left\langle C, X \right\rangle \\ 
\text{ such that } 
& \left\langle A_i, X \right\rangle = b_i \text{ for } 
i \in [m], 
\end{align*}
where $C,X,A_i \in \mathbb{R}^{n \times n}$,  
$\left\langle C, X \right\rangle = \tr(X^\top C) 
= \sum_{i,j} C_{i,j} X_{i,j}$, and 
$X \succeq 0$ means that $X$ is PSD, that is, $X$ is symmetric and its eigenvalues are nonnegative (equivalently, $x^\top X x \geq 0$ for all $x \in \mathbb{R}^n$). 

Semidefinite programming has been used successfully in many applications, e.g., 
in systems and control, 
combinatorial optimization, and 
structural design, 
see~\cite{wolkowicz2012handbook} 
and the references therein. SDP is a convex optimization problem, since the set of PSD matrices is convex. 
Note that semidefinite programming generalizes 
linear optimization (a.k.a.\ linear programming), by requiring the matrix $X$ to be diagonal, and 
second-order cone optimization, where a constraints of the type 
$\| Ax-b\|_2^2 \leq t$, where $A$ is a matrix, $x$ is a vector of variables, $b$ is a vector of parameters and $t \geq 0$ is a variable, 
can be modelled as 
\[
\left( \begin{array}{cc}
I & Ax-b \\ 
(Ax-b)^\top & t 
\end{array} 
\right) \succeq 0. 
\]
In fact, using the Schur complement, the above matrix is PSD if and only if $t - (Ax-b)^\top (Ax-b) \geq 0$. 

Usually, semidefinite programs are solved via interior-point methods~\cite{nesterov1994interior} which are expensive, as they rely on applying a Netwon step which requires, in general, $O(n^6)$ operations. However, they allow to obtain high-accuracy solutions within a few iterations, 
having quadratic convergence. 
In these lectures notes, 
we will rely on the solver SDPT3~\cite{toh1999sdpt3, tutuncu2003solving} using the CVX modeling tool~\cite{cvx}.  

However, an active direction of research is to develop faster SDP solvers, for example using 
\begin{itemize} 
\item The Burer-Monteiro approach~\cite{burer2003nonlinear} that factorizes the variable $X = UU^\top$ where $U$ has few columns (namely $r \ll n$) so as to reduce the number of variables. Although it makes the problem non-convex, one can show that all local minima are global under some appropriate conditions (in particular, the optimal solution must have low rank and $r$ needs to be sufficiently large); see also \cite{boumal2016non, waldspurger2020rank} and the references therein for recent results. 

\item First-order methods (see the next section for an example) which are particularly appropriate if high-precision solutions are not necessary for the application at hand; 
see~\cite{yurtsever2021scalable} for a recent paper on this topic. 

\end{itemize}

\subsection{Example: the semidefinite Procrustes problem} \label{sec:exSDPproc}

The semidefinite Procrustes problem is the following: 
Given $A, B \in \mathbb{R}^{m,n}$, solve
\begin{equation} \label{eq:procSDP}
\min_{X \in \mathbb{R}^{n \times n}} \| AX - B \|_F^2 
\; \text{ such that } \; X \succeq 0. 
\end{equation}  
This problem occurs for example in structure analysis~\cite{brock1968optimal}, signal processing~\cite{suffridge1993approximation}, and, as we will see, in the study of port-Hamiltonian systems.

The projection, $\mathcal{P}_{\succeq 0}(X) $, of $X$ onto the set of PSD matrices, $\mathcal{S}^n_+$, can be performed efficiently, in $\mathcal{O}(n^3)$ operations. In fact, 
\begin{equation} \label{eq:projPSD}
\mathcal{P}_{\succeq 0}(X) 
\; 
= 
\; 
\argmin_{Y \in \mathcal{S}^n_+} \| X-Y \|_F 
\; 
= 
\; 
U \max(\Lambda, 0) U^\top,  
\end{equation}
where $(U,\Lambda)$ is the eigendecomposition of the symmetric 
matrix $\frac{X+X^\top}{2} = U \Lambda U^\top$, that is, the columns of $U$ contains the eigenvectors of $\frac{X+X^\top}{2}$, and the diagonal entries of $\Lambda$ its eigenvalues~\cite{Hig88b}. 
The objective function is Lipschitz smooth, with constant 
$L = \sigma_{\max}(A)^2$. 
The gradient of $\| AX - B \|_F^2$ w.r.t.\ $X$ is $2A^\top (AX - B)$. 
We can therefore apply PGM (Algorithms~\ref{algo:gradient}) and FGM (Algorithms~\ref{algo:fgm}) to~\eqref{eq:procSDP}, with computational cost of $\mathcal{O}(n^3)$ operations per iteration. 

One can also model~\eqref{eq:procSDP} as a semidefinite program, and use interior-point methods, which will require $\mathcal{O}(n^6)$ operations per iteration. For well-conditioned problem, with $\mu = \sigma_{\min}(A) \gg 0$, FGM will converge linearly, and relatively fast, and hence should be preferred. Figure~\ref{wellcondwithzoom} illustrates this case on a  randomly generated semidefinite Procrustes  problem.
\begin{figure}[ht!]
\begin{center}
\includegraphics[width=\textwidth]{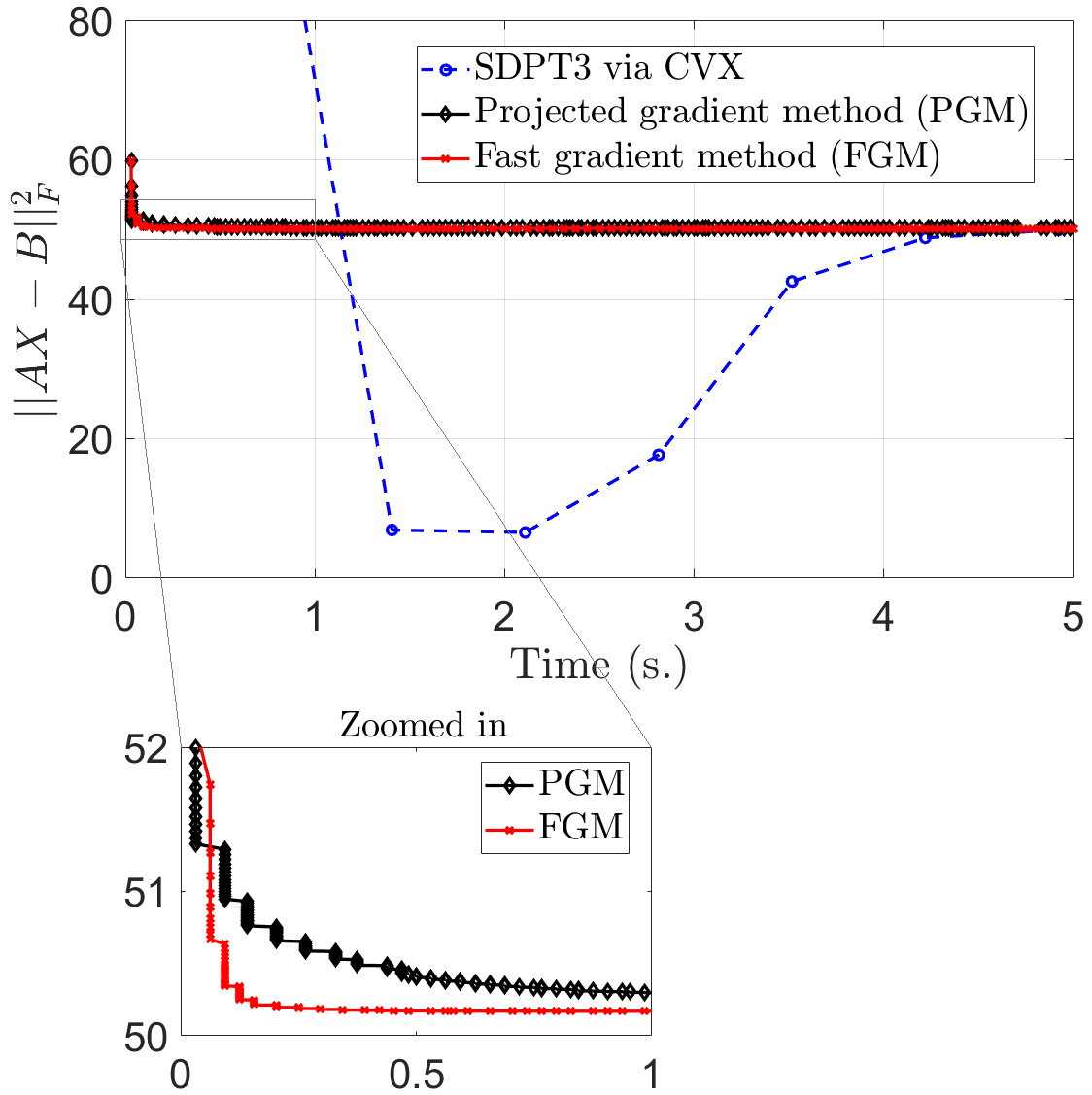}
\caption{Comparison of an IPM, PGM and FGM on a well-conditioned randomly generated semidefinite Procrusted problem, where the entries of $A$ and $B$ with $m = n = 60$ are generated using the normal distribution, 
\texttt{randn} in Matlab.  \label{wellcondwithzoom}}
\end{center}
\end{figure} 
We observe that IPM iterates take objective function values smaller than FGM at convergence. The reason is that SDPT3 uses infeasible intermediate solutions. 
We also observe, on the zoomed figure below, that FGM converges faster than PGM, as expected.

For a medium-scale ($n$ not much larger than 100) and ill-conditioned problem, with $\mu = \sigma_{\min}(A)$ close or equal to zero, IPM might be preferred. 
Figure~\ref{illcond} illustrates this case on a  randomly generated semidefinite Procrusted problem.
\begin{figure}[ht!]
\begin{center}
\includegraphics[width=\textwidth]{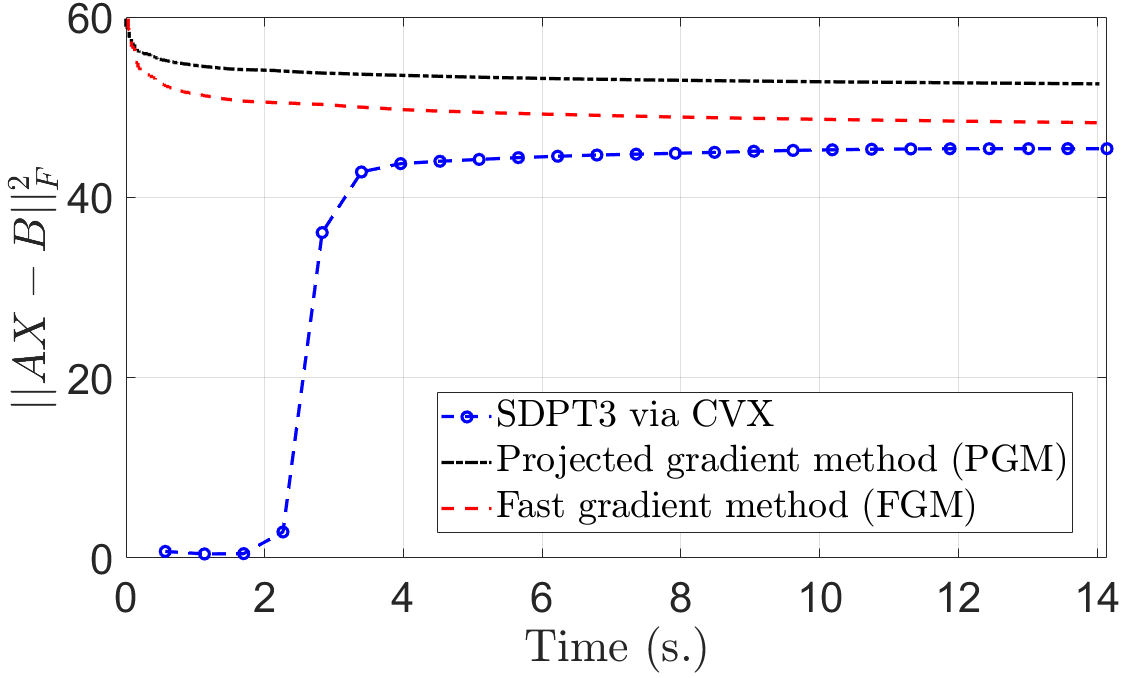}
\caption{Comparison of an IPM, PGM and FGM on an ill-conditioned randomly generated semidefinite Procrusted problem. 
The entries of $A$ and $B$ with $m = n = 60$ are generated using the normal distribution, 
\texttt{randn(n,n)} in Matlab. 
Then the SVD of $A$ is computed, $(U,\Sigma,V)$, and $A$ is replaced with $U \Sigma_{\text{ill}} V$ where the diagonal entries of $\Sigma_{\text{ill}}$ range from 1 to $10^{6}$ using log spaced values, so that $\kappa(A) = 10^6$.   \label{illcond}}
\end{center}
\end{figure}
We observe that IPM converges in about 20 iterations, while FGM was not able to converge within 8642 iterations (we stopped FGM when it attained the runtime of the IPM). 
Again, we observe that FGM converges faster than PGM, as expected.

\begin{remark}
The semidefinite Procrustes problem~\eqref{eq:procSDP} can be reformulated into an equivalent problem where $\sigma_{\min}(A) > 0$,  in which case first-order methods are often more effective than IPMs; 
see~\cite{GilS18} for the details. 
\end{remark}

\subsection{Trust-region methods}   \label{sec:trustregion}  
              
Another important class of optimization methods are trust-region methods~\cite{conn2000trust}. 
Since reviewing this rich class of methods is out of the scope of these notes, let us focus on a particular case  which will be useful when computing nearest stable matrices, and will illustrate the main idea behind trust-region methods.   

Assume you are given an optimization problem of the form 
\begin{equation} \label{eq:TR}
\min_x f(x) \quad \text{ such that } \quad x \in \mathbb{E},  
\end{equation} 
where $f$ is an non-convex and `hard' function to optimize, while $\mathbb{E}$ is nice convex set. 
Instead of trying to tackle~\eqref{eq:TR} directly (for example using projected gradient descent), 
trust-region methods will construct a model of $f(x)$ around the current iterate, $x^{(k)}$, and only trust this model in a neighbourhood around $x^{(k)}$. 
More precisely, let $x  = x^{(k)} + \Delta x$, and $f(x) \approx g( x^{(k)} + \Delta x)$ for $\Delta x$ sufficiently small. Typically, the model, $g(.)$, will be chosen as a quadratic function. Then, at iteration $k+1$, the following problem is solved, which is referred to as the trust-region subproblem, 
\[
\Delta x^* \quad = \quad 
\argmin_{\Delta x} g( x^{(k)} + \Delta x) 
\; \text{ such that } \;  \left(x^{(k)} + \Delta x\right) \in \mathbb{E} \, \text{ and } \, \| \Delta x \|_2 \leq \epsilon_k,  
\]
where $\epsilon_k > 0$ is a parameter that controls the size of the neighbourhood, referred to as the trust-region radius. The next iterate is obtained as $x^{(k+1)} = x^{(k)} + \Delta x^*$.   
If $f (  x^{(k+1)} ) \approx g( x^{(k+1)} )$, then the trust-region radius can be increased at the next iteration, otherwise it can be kept constant or decreased; many strategies exist to update the $\epsilon_k$'s. The step can also be rejected; in particular if 
$f( x^{(k+1)} ) > f(  x^{(k)} )$.


\chapter{Nearest stable matrix for continuous systems}   \label{chap:contsys}

In this chapter, we show how  dissipative Hamiltonian systems can be used to reformulate various nearest matrix problems for continuous-time LTI  systems.  

In Section~\ref{sec:reformcont}, we present a result from~\cite{GilS16} for continuous-time LTI system. 
In Section~\ref{sec:omegastab}, we explain how this result can be generalized to $\Omega$-stability, which requires the eigenvalues of the sought nearest system to belong to the set $\Omega$; this is the result from~\cite{choudhary2020approximating}. 
In Section~\ref{sec:otherapproaches}, we briefly discuss other approaches to tackle the nearest stable matrix problem. 
In Section~\ref{sec:minnomfeed}, we show how the ideas from Sections~\ref{sec:reformcont} and~\ref{sec:omegastab}  
can be used to solve the 
 static-state and 
 static-output feedback problems, which is the result from~\cite{gillis2020minimal}.

 \section{Introduction}   

Let us first consider the simplest case, a system of the form 
\[
\dot{x}(t)=Ax(t)+Bu(t),
\]
where $A\in \mathbb R^{n,n}$, $B\in \mathbb R^{n,m}$, $x$ is the state vector, and $u$ is the input vector. 
Such 
a system is stable if all eigenvalues of $A$ are in the \emph{closed} left half of the complex plane and all
eigenvalues on the imaginary axis are semisimple; see Theorem~\ref{th:stability}.  
We denote by $\mathbb S^{n,n}$ the set of stable matrices for continuous LTI systems. 

For a given unstable matrix $A$, the problem of finding the smallest perturbation that stabilizes $A$, or,
equivalently finding the nearest stable matrix $X$ to $A$ is an important problem~\cite{ONV13}, with application for example in system identification where one needs to identify a stable system from observations; see also Sections~\ref{sec:minnomfeed} and~\ref{sec:appldiscretesys}.  
More precisely,
we consider the following type-II distance problem. For a given unstable matrix $A$, compute
\begin{equation}\label{eq:prob_def}
\inf_{X\in \mathbb S^{n,n}} {\|A-X\|}_F^2,
\end{equation}
where ${\|\cdot\|}_F$ denotes the Frobenius norm of a matrix and $\mathbb S^{n,n}$ is the set of
all stable matrices of size $n\times n$. 

The set $\mathbb S^{n,n}$ is highly non-convex, and 
is not open nor closed~\cite{GilS16}, and hence~\eqref{eq:prob_def} is a difficult optimization problem. 

\section{Reformulation of the nearest stable matrix problem using DH systems} \label{sec:reformcont}

Inspired by the structure of standard DH systems $\dot{x}(t)=(J-R)
Qx(t)$, we define the following class of matrices.  
\begin{definition}[DH matrix] 
A matrix $A \in \mathbb R^{n,n}$ is said to be a dissipative Hamiltonian (DH) matrix if $A=(J-R)Q$ for some $J,R,Q \in \mathbb R^{n,n}$ such that
$J^\top=-J$, $R\succeq 0$ and $Q \succ 0$.
\end{definition}

The results presented in Section~\ref{sec_porthamsys} for DH pairs imply that DH matrices are stable (taking $E=I_n$). 
It turns out the converse is also true, that is, every stable matrix is a DH matrix. 
\begin{theorem}\cite[Lemma 2]{GilS16} \label{th:stableDHmat}
Every stable matrix is a {\em DH} matrix.
\end{theorem}
\begin{proof} 
 Let $A$ be stable. By Lyapunov's theorem~\cite{LanT85}, there exists $P \succ 0$ such that
 \begin{equation}\label{eq:lyapunov_cond}
 AP+P A^\top \preceq 0.
 \end{equation}
Let us define
 \begin{equation} \label{eq:lyapunov_form_JRQ}
 J:=\frac{AP - (AP)^\top}{2}, \quad 
 R:=-\frac{AP+(AP)^\top}{2}, 
 \; \text{ and } \; 
 Q:=P^{-1}.
 \end{equation}
 On can check that $A=(J-R)Q$, 
 while $J$ is skew symmetric, 
the inequality~\eqref{eq:lyapunov_cond}
implies that $R \succeq 0$, and $Q = P^{-1} \succ 0$. 
This implies that $A$ is a DH matrix. 
\end{proof} 

We can now reformulate \eqref{eq:prob_def} using DH matrices. 
\begin{theorem} \label{th_dhreform}
Let $A \in \mathbb{R}^{n ,n}$. Then  
$\inf_{S \in \mathbb{S}^{n, n}} {\| A - S \|}_F^2$ is equal to 
\begin{equation} \label{eq_dhreform}
 \inf_{J,R,Q \in \mathbb{R}^{n , n}} {\| A - (J-R) Q \|}_F^2 
 \; \text{ such that } J = -J^\top, R \succeq 0, Q \succeq 0. 
\end{equation} 
\end{theorem}
\begin{proof}
This follows directly from Theorem~\ref{th:stableDHmat}. 
Note that the condition $Q \succ 0$ is replaced with $Q \succeq 0$ which does not change the value of the infimum, 
but makes the feasible set of~\eqref{eq_dhreform} closed. 
\end{proof} 

Note that Theorem~\ref{th_dhreform} uses infimums, because the optimal value of these problem might not be attained, since $\mathbb{S}^{n , n}$ is not closed, while the feasible set of~\eqref{eq_dhreform}, that is, the set of DH matrices, $\{ (J,R,Q) \ | \ J = -J^\top, R \succeq 0, Q \succeq 0\}$, is not bounded. 

The advantage of the formulation~\eqref{eq_dhreform} over~\eqref{eq:prob_def} is that its feasible set is convex, and relatively easy to project onto. In fact, the projection onto the set of skew-symmetric matrices, $\bar{S}$, is given by 
\begin{equation} \label{eq:projskew}
\mathcal{P}_{\bar{S}}(Z)  
\; = \; 
\argmin_{J, J^\top = -J} \| J - Z\|_F^2 
\; = \; 
\frac{Z-Z^\top}{2},  
\end{equation} 
while the projection onto the set of PSD matrices requires an eigenvalue decomposition; see Section~\ref{sec:exSDPproc}.

\paragraph{Stable matrix and initialization} If the matrix $A$ is stable, then it can be written as $A = (J-R)Q$ for some
$J=-J^\top$, $R \succeq 0$ and $Q \succ 0$ (Theorem~\ref{th:stableDHmat}), 
and hence $AQ^{-1} = J-R$. 
In that case, we can solve the following
system to recover $(J,R,Q)$: denoting $P = Q^{-1}$,
\[
AP = J-R, \; P \succ 0, \; R \succeq 0, \; J = -J^\top.
\]
This is interesting because it provides a new (convex) way to check whether a matrix is stable. 

If $A$ is not stable so that the above system is infeasible, we can solve 
\begin{equation} \label{nearporthaminv}
\inf_{J=-J^\top,\, R \succeq 0,\,P \succeq I_n} {\| AP - (J-R) \|}_F^2,
\end{equation}
which provides an approximate solution to~\eqref{eq:prob_def} using $(J-R)P^{-1}$ as a stable approximation of $A$. 
The constraint $P \succeq I_n$ allows us to avoid the trivial solution $(J,R,Q) = (0,0,0)$. 
The solution of~\eqref{nearporthaminv} can be used as an initialization for iterative nearest stable matrix algorithms that try to tackle the difficult non-convex problem~\eqref{eq_dhreform}.
 
From the standard stability formulation~\eqref{eq:lyapunov_cond}, it is, as far as we know, not possible to extract a stable approximation from an unstable matrix. This is another advantage of our formulation.

\subsection{Optimization algorithms} 

The non-convex problem~\eqref{eq_dhreform} is hard in general, and there is no closed-form solution, as it is equivalent to the nearest stable matrix problem; see Theorem~\ref{th_dhreform}. 
Hence, it is standard to rely on iterative optimization algorithms, which  
have two key steps: 
\begin{enumerate}

\item Compute an initial solution, $Z^{(0)} = (J^{(0)},R^{(0)},Q^{(0)})$. 

\item From the $k$th iterate, $Z^{(k)}$, 
compute the next iterate, $Z^{(k+1)}$. 

\end{enumerate}

Let us describe a few approaches to tackle these two steps.

\paragraph{Initialization} 

For the initialization, we have already seen that solving the convex problem~\eqref{nearporthaminv} might be a good idea (in particular, it would give an exact solution if the input matrix is stable)--note that this also requires to resort to some iterative algorithms, such as interior-point methods of first-order methods; see Section~\ref{sec:optimization}. 

Another initialization that turns out to work well is to set $Q^{(0)} = I_n$, for which the corresponding optimal $(J^{(0)},R^{(0)})$ can be computed in closed form. 
\begin{lemma} \label{lem:sec_2_lem}
Let $A \in \mathbb R^{n,n}$. 
The optimal solution of 
\[
\min_{J,R \in \mathbb{R}^{n \times n}} 
 {\| A - (J-R) \|}_F^2
\quad \text{ such that }  \quad J=-J^\top \text{ and } R \succeq 0, 
\] 
is given by 
\[
\widehat J = \mathcal P_{\bar S}(A) = \frac{A-A^\top}{2} 
\quad 
\text{ and } 
\quad  
\widehat R = \mathcal P_{\succeq 0}\left(\frac{-A-A^\top}{2}\right), 
\]
where $\mathcal P_{\bar S}(.)$ is the projection on the set of skew-symmetric matrices, defined in~\eqref{eq:projskew}, and $\mathcal P_{\succeq 0}(.)$ is the projection on the set of PSD  matrices,  defined in~\eqref{eq:projPSD}.  
\end{lemma}
\begin{proof}
Observe that
\begin{align*}
\min_{R\succeq 0, J^\top=-J} & {\| A - (J-R)\|}_F^2  \\
& = \min_{R\succeq 0} \left(\min_{J^\top=-J}{\|(A+R)-J\|}_F^2\right) \\
& = \min_{R\succeq 0}{\left\|(A+R)-\mathcal P_{\bar S}(A+R)\right\|}_F^2 \\ 
& = \min_{R\succeq 0}{\left\|R + \frac{A+A^\top}{2}\right\|}_F^2, \\
\end{align*} 
where the second and third equalities follow from~\eqref{eq:projskew}, that is,  
\[
\underset{J^\top=-J}{\argmin} \|(A+R)-J\|_F^2
= 
\mathcal P_{\bar S}(A+R) 
= \frac{(A+R)-(A+R)^\top}{2} 
= \frac{A-A^\top}{2}  
= \mathcal P_{\bar S}(A).
\]
The closed form for $R$ follows from the projection on the set of PSD matrices~\eqref{eq:projPSD}. 
\end{proof}

Another strategy for initialization is to simply generate $(J,R,Q)$ randomly, and project them onto the feasible set.

\paragraph{Updating iterates}

To update the iterates $(J^{(k)},R^{(k)},Q^{(k)})$, any standard optimization scheme can be used. 
For example, one can use a BCD method, where $J$, $R$ and $Q$ are updated alternatively. In fact, the subproblems in one block of variables is a convex constrained least squares problem. 
These subproblems can be solved via IPM, or gradient descent; the latter being more appropriate for large-scale problems as explained in Section~\ref{sec:optimization}. 
Another approach is to use PGD (Algorithm~\ref{algo:gradient}) or FGM (Algorithm~\ref{algo:fgm}) on all variables simultaneously. 
FGM turns out to perform best among PGM and BCD; see~\cite{GilS16} for numerous numerical experiments.  

Since~\eqref{eq_dhreform} is not Lipschitz smooth with respect to $(J,R,Q)$, there is no clear choice for the step length $\gamma_k$. 
The Lipschitz constant of the gradient of the objective function with respect to $Q$  (for fixed $J,R$) is given by $\lambda_{\max}( (J-R)^\top (J-R) )$
while the Lipschitz constant of the gradient with respect to $(J,R)$ (for fixed $Q$)  is given by $\lambda_{\max}( QQ^\top )$. 
Therefore, it makes sense to scale $(J,R)$ and $Q$ such that  $L = \lambda_{\max}( (J-R)^\top (J-R) ) = \lambda_{\max}( QQ^\top )$ while
choosing an initial steplength $\delta = 1/L$. Note that this allows to remove the scaling degree of freedom since this imposes
 ${\|J-R\|}_2={\|Q\|}_2$. 
In order to avoid computing the maximum eigenvalues of $(J-R)^\top (J-R)$ and $QQ^\top$ from scratch at each step,
we use a few steps of the power method to update the initial value (since $J-R$ and $Q$ do not change too much between two iterations).
We combined this with a backtracking line search: if the objective function has not decreased, the step is divided by a fixed constant larger than one until decrease is achieved. It turns out that in most cases, especially when getting closer to stationary points,
the steplength of $1/L$ allows to decrease the objective function.

\subsection{Numerical example with the AC7 system} \label{sec:numexpAC7stablematrix}

A useful library containing many dynamical systems is the  COnstrained Matrix-optimization Problem library, 
COMPleib~\cite{leibfritz2004compleib}. 
It contains for example the so-called AC7 system 
from~\cite{gangsaas1986application} 
(case study III~2), 
which comes from an  aircraft 
stability and control problem. 
For now, we only consider the matrix $A$ of this system, given by, with three digits of accuracy, 
\small 
\[
A =  
\left( \begin{array}{ccccccccc} 
 -0.063 &  0.019 &  0  &  -0.561 &  -0.028 &  0  &  0.063 &  -0.001 &  0  \\ 
 0.011 &  -0.993 &  0.998 &  0.001 &  -0.071 &  0  &  -0.011 &  0.064 &  0  \\ 
 0.077 &  1.675 &  -1.311 &  0 &  -4.250 &  0  &  -0.077 &  -0.109 &  0  \\ 
 0  &  0  &  1 &  0  &  0  &  0  &  0  &  0  &  0  \\ 
 0  &  0  &  0  &  0  &  -20  &  20  &  0  &  0  &  0  \\ 
 0  &  0  &  0  &  0  &  0  &  -30  &  0  &  0  &  0  \\ 
 0  &  0  &  0  &  0  &  0  &  0  &  -0.882 &  0  &  0  \\ 
 0  &  0  &  0  &  0  &  0  &  0  &  0  &  -0.882 &  0.009 \\ 
 0  &  0  &  0  &  0  &  0  &  0  &  0  &  -0.009 &  -0.882 \\ 
\end{array} \right)
\] 
\normalsize 
whose eigenvalues are given by 
\[
\{  
   0.172 \pm 0.236i, 
  -0.25, 
     -0.882, 
  -0.882 \pm 0.009i, 
  -2.458, 
 -20, 
 -30 
 \},  
\]
with two eigenvalues with a positive real part. 

Using the $(J-R)Q$ representation and FGM initialized with $Q=I_n$ and $(J,R)$ as the optimal solution of the corresponding problem (Lemma~\ref{lem:sec_2_lem}), 
we obtain, in less than a second,  a stable approximation of $A$, namely $(J-R)Q$,  equal to 
\[ 
\left( \begin{array}{ccccccccc} 
 -0.091 &  0.019 &  -0.001 &  -0.561 &  -0.028 &  0&  0.063 &  -0.001 &  0\\ 
 -0.027 &  -1.008 &  0.980 &  0.010 &  -0.071 &  0&  -0.011 &  0.064 &  0\\ 
 0.053 &  1.665 &  -1.324 &  -0.006 &  -4.250 &  0&  -0.077 &  -0.109 &  0\\ 
 0.011 &  0&  1 &  -0.029 &  0 &  0 &  0 &  0&  0 \\ 
 0.005 &  0.002 &  0.003 &  0.001 &  -20 &  20 &  0 &  0 &  0 \\ 
 0.004 &  0.001 &  0.002 &  0.001 &  0 &  -30 &  0 &  0 &  0 \\ 
 0&  0.001 &  0.001 &  -0.001 &  0 &  0 &  -0.882 &  0&  0 \\ 
 0.001 &  0 &  0 &  0.001 &  0 &  0&  0&  -0.882 &  0.009 \\ 
 0 &  0 &  0 &  0 &  0 &  0 &  0&  -0.009 &  -0.882 \\ 
\end{array} \right) 
\] 
with relative error $\frac{\| A - (J-R) Q\|_F}{\|A\|_F} = 0.168\%$, and 
 \[
\Lambda((J-R)Q) = \{  
0, 
\pm 0.037i, 
    -0.882,  
  -0.882 \pm 0.009i,  
      -2.453,
 -20,  -30.001 
 \}.  
\]
Figure~\ref{AC7_nearestJRQ} illustrates this approximation. 
\begin{figure}[ht!]
\begin{center}
\includegraphics[width=0.7\textwidth]{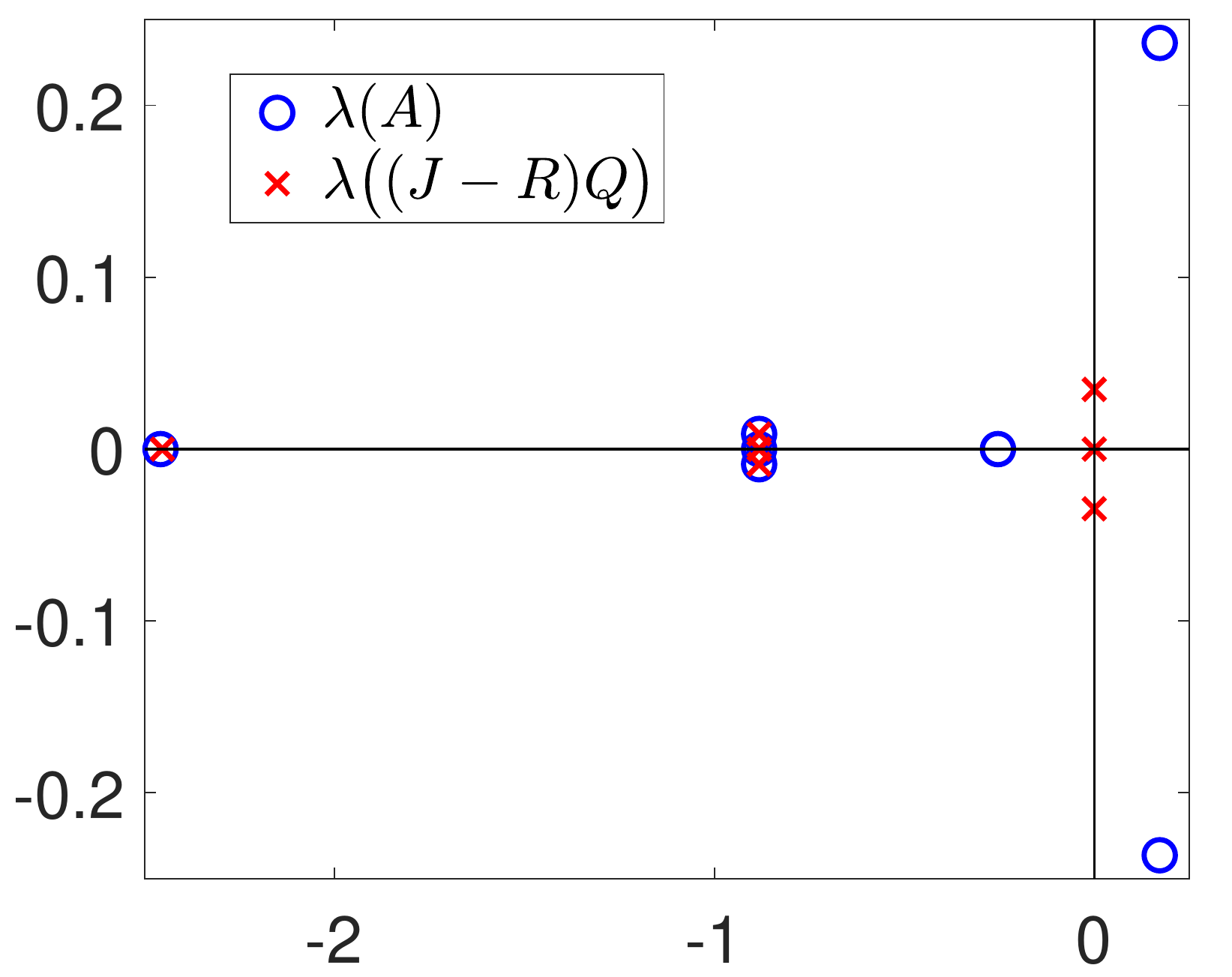}
\caption{Eigenvalues of $A$, and its stable approximation $(J-R)Q$ computed by FGM. Note that the eigenvalues -20 and -30 of $A$ are not displayed, but $(J-R)Q$ has the same eigenvalues.   \label{AC7_nearestJRQ}}
\end{center}
\end{figure}

The approximation of $A$ is not asymptotically stable as it has three eigenvalues on the imaginary axis. 
This is expected since the open left half of the complex plane is an open set. 

If one wishes to obtain an asymptotically stable matrix, there are (at least) two possibilities: 
\begin{enumerate}

\item One can change the feasible set $Q \succeq 0$ and 
$R \succeq 0$  to $Q \succeq \delta I$ and $R \succeq \delta I$ for some parameter $\delta > 0$. 
For example, using $\delta = 10^{-3}$, we obtain an approximation $(J-R)Q$ with relative error of 1.44\% with eigenvalues 
\[
\{ 
-0.15 \pm 0.43i, 
  -0.31, 
  -0.88,  
  -0.88 \pm 0.01i, 
  -2.46, 
 -20, 
 -30 
 \}. 
\]

\item One can apply the same algorithm on the matrix $A+\epsilon I$ for some parameter $\epsilon > 0$. 
This will give $A + \epsilon I \approx (J-R)Q$, and hence 
$A \approx  (J-R)Q - \epsilon I$ where the real part of the eigenvalues of 
$(J-R)Q - \epsilon I$ are guaranteed to be smaller than $-\epsilon$, since the eigenvalues the real part of the eigenvalues of 
$(J-R)Q$ are nonpositive.  
For example, using $\epsilon = 10^{-3}$, we obtain an approximation $(J-R)Q$ with relative error of 0.1705\% (while the case $\epsilon=0$  gives 0.168\%) with eigenvalues 
\[
\{ 
-0.001, 
  -0.001 \pm 0.034i, 
  -0.88,  
  -0.88 \pm 0.01i, 
  -2.45, 
 -20, 
 -30 
 \} . 
\] 
This second approach allows to control directly the maximum real part of the approximation of~$A$. 

\end{enumerate}

Another possibility would be to impose additional constraints on $(J,R,Q)$ to ensure that the maximum real part of the eigenvalues of $(J-R)Q$ is smaller than some given constant; this is discussed in the next section.

 \section{Generalization to $\Omega$-stability} \label{sec:omegastab} 

At first sight, the strategy proposed in Section~\ref{sec:reformcont} to reformulate the nearest stable matrix is only useful for continuous-time systems; see Theorem~\ref{th_dhreform}. 
However, using appropriate constraints on $J$, $R$ and $Q$, it is possible constraint the eigenvalues of $(J-R)Q$ to belong to other subsets of the complex plane.   
In fact, it is possible to represent three types of sets via additional convex constraints on $J$, $R$ and $Q$, namely: 
\begin{itemize}
\item Conic sector: the conic sector region of parameters $a,\,\theta \in \R$ with $0 < \theta < \pi/2$, denoted by $\Omega_C(a,\theta)$, is defined as
\[
\Omega_C(a,\theta)
:=
\left\{ x+iy \in \C \ \big| \ \sin(\theta) (x-a) < \cos(\theta) y < -\sin(\theta) (x-a),\,
x < a \right\}.
\]

\item Vertical strip: the vertical strip region of parameters $ h < k$, denoted by $\Omega_V(h,k)$, is defined as
\[
\Omega_V(h,k)
:= \left\{ x+iy \in \C \ \big| \ -k < x < -h \right\}.
\]
Note that $h$ (resp.\ $k$) can possibly be equal to $-\infty$ (resp.\ $+\infty$)  in which case $\Omega_V$ is a half space.
In particular, $\Omega_V(0,+\infty)$ is the open left half of the complex plane, corresponding to stable matrices for continuous LTI systems.

\item Disks centred on the real line:
the disk centred  at $(-q,0)$ with radius $r>0$, denoted by $\Omega_D(-q,r)$, is defined as
\[
\Omega_D(-q,r)
:= \left\{ z\in \C \ \big| \ |z+q|<r\right\}.
\]
In particular, $\Omega_D(0,1)$ is the unit disk, corresponding to stable matrices for discrete LTI systems.

\end{itemize}

For a given region $\Omega \subseteq \C$, the matrix $A\in \R^{n,n}$ is called $\Omega$-stable if all its eigenvalues lie inside the region $\Omega$. We consider the following analogue of~\eqref{eq:prob_def}, and called it as \emph{nearest $\Omega$-stable matrix problem}: 

\begin{equation}\label{eq:prob_def_omega}
\inf_{X\in \mathbb S_{\Omega}^{n,n}} {\|A-X\|}_F^2,
\end{equation}
where $\mathbb S_{\Omega}^{n,n}$ is the set of
all $\Omega$-stable matrices of size $n\times n$. 

We consider $\Omega$ as either any of $\Omega_C$, $\Omega_V$, $\Omega_D$, or the intersection of such sets; see Figure~\ref{fig1} for an illustration. Note that $\Omega$ is symmetric with respect to the real line. 
\begin{figure}
\begin{center}
\includegraphics[width=5cm]{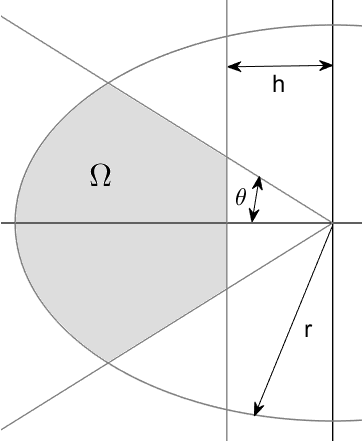}
\caption{Illustration of 
$\Omega =
\Omega_C(0,\theta)
\cap \Omega_V(h,+\infty)
\cap \Omega_D(0,r)$.
\label{fig1}}
\end{center}
\end{figure} 

\subsection{Generalizing Lyapunov LMI to $\Omega$ stability}

In~\cite{choudhary2020approximating}, we relied on the results of Chilali and Gahinet~\cite{ChilG96} to constrain $J$, $R$ and $Q$ such that the eigenvalues of $(J-R)Q$ belong to sets $\Omega$ as described in the previous section.   Let us recall this result.  

 A subset $\Omega$ of $\mathbb{C}$ is called an LMI region if there exist a real symmetric matrix $B$ and  a real matrix $C$ such that 
 \[
 \Omega \, = \, \left\{ z \in \mathbb{C} \ | \ 
 f_\Omega(z) := B + zC + \bar{z} C^\top  \prec 0 \right\}. 
 \] 
Note that such sets are symmetric with respect to the real line, since 
 $f_\Omega(\bar z) = \overline{f_\Omega(z)}$. 
 Chilali and Gahinet~\cite[Theorem 2.2]{ChilG96} showed that $A$ is $\Omega$-stable if and only if there exists $P \succ 0$ such that 
 \[
 B \otimes P + X \otimes (AP) + C^\top \otimes (AP)^\top \prec 0,  
 \]
 where $\otimes$ denotes the Kronecker product. 
 
 Let us illustrate this result when $\Omega$ is a circle centred at $(-q,0)$ of radius $r > 0$, namely $\Omega = \Omega_D(-q,r)$. We have 
 \[
 f_\Omega(z) = 
 \mat{cc} 
 -r & z+q \\ 
 \overline{z+q} & -r 
 \rix 
 = 
  \underbrace{\mat{cc} 
 -r & q \\ 
  q & -r 
 \rix}_{=B}
 + 
 z \underbrace{\mat{cc} 
 0 & 1 \\ 
 0 & 0
 \rix}_{=C} 
 + 
 \bar z \underbrace{\mat{cc} 
 0 & 1  \\ 
 0 & 0 
 \rix^\top}_{=C^\top} . 
 \]
 In fact, $f_\Omega(z) \prec 0$ if and only if the trace of $f_\Omega(z)$ is negative (the sum of the eigenvalues is negative) and the determinant is negative (the product of the eigenvalues is positive). 
 The trace is always negative since $r < 0$, while the determinant is given by $r^2 - |z+q|^2 < 0$, which gives the result. 
 Then, to obtain a Lyapunov-like LMI for this set, it suffices to use the above result: a matrix $A$ is $\Omega_D(-q,r)$-stable if there exists $P \succ 0$ such that 
 \[
   \mat{cc} 
 -r P & q \\ 
  q & -r P  
 \rix 
 + 
   \mat{cc} 
 0 &  AP \\ 
0 & 0 
 \rix 
  + 
   \mat{cc} 
 0 & 0 \\ 
 (AP)^\top & 0 
 \rix 
 \prec 0. 
 \]
 Interestingly, for the particular case of discrete stability, with $q = 0$ and $r = 1$, we obtain the standard Lyapunov LMI: 
 \begin{align*}
   \mat{cc} 
 - P & AP \\ 
 (AP)^\top & - P  
 \rix 
\prec 0 & \quad 
\iff 
\qquad 
\mat{cc} 
 P & -AP \\ 
 -(AP)^\top &  P  \rix 
\succ 0 \\ 
& \quad 
\iff 
\qquad 
P \succ 0 \text{ and }  P - (AP) P^{-1} (AP)^\top = P - APA^\top \succ 0, 
 \end{align*}
 where the second equivalence follows from the Schur complement. \\

 Before we provide the constraints on $(J,R,Q)$ to have the eigenvalues of $(J-R)Q$ belong to various LMI regions, 
 let us provide a useful lemma. 
 
 \begin{lemma}\cite[Lemma 1]{choudhary2020approximating}\label{lem:dhevalues}
Let $A=(J-R)Q$, where $J,R,Q \in \R^{n,n}$ is such that $J^\top=-J$, $R^\top=R$, and $Q^\top=Q$ is invertible.
Let $\lambda \in \C$, and $v \in \C^{n}\setminus \{0\}$  be such that
$v^*A=\lambda v^*$. Then
\[
\real{(\lambda)}= -\frac{v^*Rv}{v^*Q^{-1}v} \quad \text{and}\quad \imag{(\lambda)}= -i\frac{v^*Jv}{v^*Q^{-1}v}.
\]
\end{lemma}
\begin{proof}
 Let $v$ be a left eigenvector of $A$ corresponding to
 eigenvalue $\lambda$, i.e., $v^*A=\lambda v^*$. Then
\begin{eqnarray}\label{eq1}
v^*(J-R)Q=\lambda v^* \Longrightarrow v^*(J-R)v=\lambda v^*Q^{-1}v, 
\end{eqnarray}
and by taking the conjugate of~\eqref{eq1}, we get
\begin{eqnarray}\label{eq2}
 v^*(-J-R)v=\bar \lambda v^*Q^{-1}v.
\end{eqnarray}
Substracting~\eqref{eq2} from~\eqref{eq1}, we obtain 
\begin{eqnarray*}\label{eq3}
2v^*Jv &=&(\lambda -\bar \lambda)v^*Q^{-1}v =
2 i \imag(\lambda) v^*Q^{-1}v \Longrightarrow v^*Jv = i \imag( \lambda) v^*Q^{-1}v,
\end{eqnarray*}
and, summing~\eqref{eq2} with~\eqref{eq1}, we obtain 
\begin{eqnarray*}\label{eq4}
-2v^*Rv &=&(\lambda +\bar \lambda)v^*Q^{-1}v =
2 \real(\lambda) v^*Q^{-1}v \Longrightarrow v^*Rv = -\real( \lambda) v^*Q^{-1}v.
\end{eqnarray*}
\end{proof}

\subsection{Conic sectors, $\Omega_C$}  \label{paramC}

Consider the region $\Omega_C(a,\theta)$ with parameters
$a\in \R$ and $0 < \theta < \pi/2$ and
let $\alpha := \sin(\theta)$ and $\beta := \cos(\theta)$.
To parametrize $\Omega_C(a,\theta)$ in terms of DH matrices, let us first prove the following lemma.
\begin{lemma}\cite[Lemma 2]{choudhary2020approximating} \label{lem1}
Let $\lambda =\lambda_1+i\lambda_2$, where $\lambda_1$, $\lambda_2 \in \R$. Then  $\mat{cc} \alpha \,(\lambda_1-a) & \beta \,i\lambda_2 \\ -\beta\, i\lambda_2 & \alpha \,(\lambda_1-a) \rix \prec 0$
if and only if $\lambda \in \Omega_C(a,\theta)$.
\end{lemma}
\begin{proof} The proof follows using the fact that $\mat{cc} \alpha \,(\lambda_1-a) & \beta \,i\lambda_2 \\ -\beta\, i\lambda_2 & \alpha \,(\lambda_1-a) \rix$
is Hermitian and therefore it is negative definite if and only if both eigenvalues $\mu_1=\alpha (\lambda_1-a) +\beta \lambda_2$
and $\mu_2=\alpha (\lambda_1-a) -\beta \lambda_2$ are negative which is true if and only if $\alpha (\lambda_1-a) < \beta \lambda_2 < -\alpha (\lambda_1-a)$, i.e.,
$\lambda \in \Omega_C(a,\theta)$.
\end{proof}

\begin{theorem}\cite[Theorem 1]{choudhary2020approximating} \label{mainthm}
Let $A\in \R^{n,n}$. Then $A$ is $\Omega_C(a,\theta)$-stable if and only if $A=(J-R)Q$ for some $J,R,Q \in R^{n,n}$ such that
$J^\top=-J$, $R^\top=R$,  $Q$ is symmetric positive definite, and
\begin{equation}\label{eq:tem1}
\mat{cc}\alpha  (R+aQ^{-1}) & -\beta J \\
\beta   J & \alpha (R+aQ^{-1})
\rix
\succ 0.
\end{equation}
\end{theorem}
\begin{proof}
First suppose that $A=(J-R)Q$ for some $J,R,Q$ satisfying $Q\succ 0$ and~\eqref{eq:tem1}.
Let $\lambda =\lambda_1 +i \lambda_2$ be an eigenvalue of $A$ and let $v \in \C^n \setminus \{0\}$ be a left eigenvector of $A$ corresponding to eigenvalue~$\lambda$.
Since 
\[ 
\mat{cc}\alpha\, (R+aQ^{-1}) & -\beta\,J \\\beta\, J & \alpha\, (R+aQ^{-1})\rix \succ 0 
\] 
and $v \neq 0$, we have that
\begin{eqnarray}\label{eq51}
-2\mat{cc} v^* & 0 \\ 0 & v^*\rix
\mat{cc}\alpha\, (R+aQ^{-1}) & -\beta\,J \\\beta\, J & \alpha\, (R+aQ^{-1})\rix
\mat{cc} v & 0 \\ 0 & v\rix &\prec& 0 \nonumber\\
\Longrightarrow\quad  2\mat{cc} -\alpha v^*(R+aQ^{-1})v & \beta v^*Jv \\-\beta v^*Jv & -\alpha v^*(R+aQ^{-1})v \rix &\prec& 0 \nonumber\\
\Longrightarrow\quad
\mat{cc} -\alpha v^*Rv  & \beta v^*Jv \\-\beta v^*Jv & -\alpha v^*Rv \rix -
\alpha \mat{cc}av^*Q^{-1}v & 0 \\ 0 & av^*Q^{-1}v \rix&\prec& 0.
\end{eqnarray}
Thus by using Lemma~\ref{lem:dhevalues} in~\eqref{eq51}, we obtain
\begin{eqnarray}
v^*Q^{-1}v\mat{cc} \alpha \,(\lambda_1-a) & \beta\,i\lambda_2 \\-\beta\,i \lambda_2 & \alpha \,(\lambda_1-a) \rix \prec 0.
\end{eqnarray}
This implies that $\mat{cc} \alpha \,(\lambda_1-a) & \beta\,i\lambda_2 \\-\beta\,i \lambda_2 & \alpha \,(\lambda_1-a) \rix \prec 0$ since $Q$ is positive definite.
Thus Lemma~\ref{lem1} implies that $A$ is $\Omega_C(a,\theta)$-stable.

For the `only if' part, since $A$ is $\Omega_C(a,\theta)$-stable, by~\cite[Theorem~2.2]{ChilG96}, there exists $X \succ 0$ such that
\begin{equation}\label{eqt}
\mat{cc} \alpha (AX+XA^\top -2aX) & \beta (AX-XA^\top) \\
\beta (XA^\top-AX) & \alpha (AX+XA^\top-2aX)
\rix \prec 0.
\end{equation}
Let
\begin{equation}\label{eq:JRQ}
R=-\frac{AX+(AX)^\top}{2},\quad J=\frac{AX-(AX)^\top}{2},\quad \text{and}\quad Q=X^{-1}.
\end{equation}
Then $(J-R)Q=A$ and it follows from~\eqref{eqt} that
\[
\mat{cc}\alpha\, (R+aQ^{-1}) & -\beta\,J \\\beta\, J & \alpha\, (R+aQ^{-1})\rix =- \frac{1}{2} \mat{cc} \alpha (AX+XA^\top-2aX) & \beta (AX-XA^\top) \\
\beta (XA^\top-AX) & \alpha (AX+XA^\top-2aX) \rix
\]
is positive definite.
\end{proof}

As a consequence of~\eqref{eq:tem1} in Theorem~\ref{mainthm}, the matrix $J$ is skew-symmetric. However, the matrix $R$ may not be positive definite (when $a >0$)
and therefore the $\Omega_C(a,\theta)$-stable matrix $A$ need not be a DH matrix. But
when $a \leq 0$, then~\eqref{eq:tem1} implies that $R+aQ^{-1} \succ 0$, or equivalently, $R \succ -aQ^{-1}$ since $a \leq 0$ and $Q\succ 0$. As a result
$R$ is positive semidefinite. Therefore in this case $A$ is $\Omega_C(a,\theta)$-stable if and only if $A$ is a DH matrix satisfying~\eqref{eq:tem1}.

\subsection{Vertical strips, $\Omega_V$} \label{paramV}

We can characterize $\Omega_V$-stability as follows.

\begin{theorem}\cite[Theorem 2]{choudhary2020approximating} \label{thm:vert}
Let $A \in \R^{n,n}$ and $ h < k$. Then $A$ is $\Omega_V(h,k)$-stable if and only if $A=(J-R)Q$
for some $J,R,Q \in \R^{n,n}$ such that $J^\top=-J$, $R^\top=R$,
$Q$ is symmetric positive definite, and
\begin{equation} \label{condVert}
k Q^{-1} \succ R \succ h Q^{-1}.
\end{equation}
\end{theorem}
\begin{proof} First suppose that $A=(J-R)Q$, where $J^\top=-J$, $R^\top=R$, $Q\succ 0$ such that  $k Q^{-1} \succ R\succ h Q^{-1}$.
Let $\lambda$ be an eigenvalue of $A$ and $x \in \C^{n}\setminus \{0\}$
be such that
$Ax=\lambda x$ or $(J-R)Qx=\lambda x$. Since $Q$ is invertible, this implies that
\begin{equation}\label{eq:1}
x^*Q(J-R)Qx=\lambda x^*Q x \quad \Longrightarrow \quad \real{(\lambda)}=-\frac{x^*QRQx}{x^*Qx}.
\end{equation}
Since $x^*Qx > 0$ as $Q\succ 0$ and $R$ satisfies $k Q^{-1} \succ R\succ h Q^{-1}$, we have
$k x^*Qx > x^*QRQx > h x^*Qx$. This implies that
\begin{equation}\label{eq:2}
k  > \frac{x^*QRQx}{x^*Qx} > h.
\end{equation}
From~\eqref{eq:1} and~\eqref{eq:2}, we have that $-k<\real{(\lambda)} <-h$.

Conversely, let $A$ be $\Omega_V(h,k)$-stable. Then from~\cite{ChilG96}, there exists $X \succ 0$
such that
\begin{equation}\label{eq:3}
AX+XA^\top+2hX \prec 0\quad  \text{and}\quad AX+XA^\top+2kX \succ 0.
\end{equation}
Define $J$, $R$, and $Q$ as in~\eqref{eq:JRQ}.
Then clearly $A=(J-R)Q$. Also in view of~\eqref{eq:3} the matrix $R$ satisfies $k Q^{-1} \succ R\succ h Q^{-1}$.
\end{proof}

It is easy to see that in Theorem~\ref{thm:vert} when $h \geq 0$  the matrix $A$ is $\Omega_V$-stable
if and only if $A$ is a DH matrix since $R \succ h Q^{-1}$.

\subsection{Disks centred on the real line, $\Omega_D$} \label{paramD}

The disk $\Omega_D(-q,r)$ of radius $r$ and center $(-q,0)$ is
an LMI region with characteristic function $f_D(z)=\mat{cc} -r & q+\lambda \\ q+\overline{\lambda} & -r \rix$~\cite[Definition 2.1]{ChilG96}. More precisely,
we have the following lemma.
\begin{lemma}\cite[Lemma 3]{choudhary2020approximating} \label{lem:omega_stab}
Consider the region $\Omega_D(-q,r)$  where $q \in \R$ and $r > 0$, and let $\lambda \in \C$.
Then  $\lambda \in \Omega_D(-q,r)$ if and only if
$\mat{cc} -r & q+\lambda \\ q+\overline{\lambda} & -r \rix \prec 0$.
\end{lemma}
We can characterize $\Omega_D$-stability as follows.
\begin{theorem}\cite[Theorem 3]{choudhary2020approximating} \label{thm:omega2}
Let $A \in \R^{n,n}$, $q \in \R$ and $r > 0$.
Then $A$ is $\Omega_D(-q,r)$-stable if and only if  $A=(J-R)Q$ for some $J,R,Q \in \R^{n,n}$ such that
$J^\top = -J$, $R^\top=R$, $Q$ is symmetric positive definite, and
\begin{equation}\label{eq:t1}
\mat{cc} rQ^{-1}& -qQ^{-1} \\ -qQ^{-1} & r Q^{-1} \rix \succ \mat{cc} 0& J-R \\ (J-R)^\top &0
\rix.
\end{equation}
\end{theorem}
\begin{proof} 
First suppose that $A=(J-R)Q$ with $J^\top=-J$, $R^\top=R$ and $Q\succ 0$ satisfying~\eqref{eq:t1} holds.
Let $\lambda=\lambda_1+i\lambda_2$ with $\lambda_1,\lambda_2 \in \R$ be an eigenvalue of $A$ and $v \in \C^{n}\setminus \{0\}$ be a corresponding left eigenvector.
Since~\eqref{eq:t1} holds, we have that
\begin{eqnarray*}
\mat{cc} v^* & 0 \\ 0 & v^*\rix\mat{cc} rQ^{-1}& -qQ^{-1} \\ -qQ^{-1} & r Q^{-1} \rix
\mat{cc} v & 0 \\ 0 & v\rix \succ
\mat{cc} v^* & 0 \\ 0 & v^*\rix\mat{cc} 0& J-R \\ (J-R)^\top &0
\rix \mat{cc} v & 0 \\ 0 & v\rix.
\end{eqnarray*}
This implies that
\begin{eqnarray*}
\mat{cc} r & -q \\ -q & r  \rix v^*Q^{-1}v \succ \mat{cc} 0& v^*(J-R)v \\ v^*(J-R)^\top v &0
\rix.
\end{eqnarray*}
Since $v^*Q^{-1}v > 0$ as $Q \succ 0$, we obtain
\begin{eqnarray*}
 \mat{cc} r & -q \\ -q & r  \rix  \succ \mat{cc} 0& \frac{v^*(J-R)v}{v^*Q^{-1}v} \\ \frac{v^*(J-R)^\top v}{v^*Q^{-1}v} &0
\rix .
\end{eqnarray*}
Thus in view of Lemma~\ref{lem:dhevalues}, we have
\begin{eqnarray*}
 \mat{cc} r & -q \\ -q & r  \rix \succ \mat{cc} 0& i\lambda_2+\lambda_1 \\ -i\lambda_2+\lambda_1 &0
\rix \quad \Longrightarrow \quad  \mat{cc} r & -q-\lambda \\ -q-\overline{\lambda} & r  \rix \succ 0.
\end{eqnarray*}
This implies by using Lemma~\ref{lem:omega_stab} that $\lambda \in \Omega_D(-q,r)$ and therefore $A$ is
$\Omega_D(-q,r)$-stable.

Conversely, suppose $A$ is $\Omega_D(-q,r)$-stable.  Then by \cite[Theorem~2.2]{ChilG96} there exists $X \succ 0$ satisfying
\begin{equation}\label{eqtr1}
\mat{cc} -rX & qX+AX \\ qX +XA^\top & -rX
\rix \prec 0.
\end{equation}
Define $J$, $R$, and $Q$ as in~\eqref{eq:JRQ}.
Then clearly $A=(J-R)Q$ with $J^\top = -J$, $R$ symmetric and $Q \succ 0$.
Moreover, by~\eqref{eqtr1}, we have
\begin{eqnarray*}
0 \succ \mat{cc} -rX & qX+AX \\ qX +XA^\top & -rX \rix &=& \mat{cc} -rQ^{-1} & qQ^{-1}+AQ^{-1} \\ qQ^{-1} +Q^{-1}A^\top & -rQ^{-1}
\rix \\
&=&\mat{cc} -r Q^{-1}  & q Q^{-1}  \\ q Q^{-1}  & -r Q^{-1}  \rix  +
\mat{cc}0 & J-R \\ (J-R)^\top & 0\rix.
\end{eqnarray*}
This implies that
\begin{eqnarray*}
\mat{cc} r Q^{-1}  & -q Q^{-1}  \\ -q Q^{-1} & r Q^{-1}  \rix \succ
\mat{cc}0 & J-R \\ (J-R)^\top & 0\rix.
\end{eqnarray*}
This completes the proof.
\end{proof}

We note that in the above theorem, the matrix $R$ need not be positive semidefinite and thus a $\Omega_D$-stable matrix need not be a DH matrix. However, if the disc $\Omega_D$ completely lies in the left half of the complex plane,
then $A$ is a DH matrix.

\subsection{Reformulation of the nearest $\Omega$-stable matrix} 

The set $\mathbb S_{\Omega}^{n,n}$ can be reformulated in terms of matrix triplets with symmetric and PSD constraints. 
For this, we introduce the auxiliary variable $P = Q^{-1} \succ 0$.
In view of Theorems~\ref{mainthm},~\ref{thm:vert} and~\ref{thm:omega2}, this allows us to parametrize the sets
$\mathbb S_{\Omega_C(a,\theta)}^{n,n}$,
$\mathbb S_{\Omega_V(h,k)}^{n,n}$, and
$\mathbb S_{\Omega_D(-q,r)}^{n,n}$ as convex sets via the DH form. This is done as follows: 
\begin{eqnarray} \label{eq:omegaC}
\mathbb S_{\Omega_C(a,\theta)}^{n,n} = \Big\{(J-R)P^{-1} \ \big| && J,R,P \in \R^{n,n}, P \succ 0, \nonumber\\
&&\mat{cc}
\sin(\theta) \, (R+aP) & -\cos(\theta) \,J \\
\cos(\theta) \, J & \sin(\theta) \, (R+aP)
\rix \succ 0
\Big\},
\end{eqnarray}
\begin{equation} \label{eq:omegaV}
\mathbb S_{\Omega_V(h,k)}^{n,n} = \left\{(J-R)P^{-1} \ \big| ~ J,\,R,\,P \in \R^{n,n},\, P \succ 0,\, k P \succ R\succ h P
\right\},
\end{equation}
and
\begin{eqnarray} \label{eq:omegaD}
\mathbb S_{\Omega_D(-q,r)}^{n,n} = \Big\{(J-R)P^{-1} \ \big|&& \ J,R,P \in \R^{n,n}, J^T = -J,\nonumber \\
&&\mat{cc} rP & -qP-(J-R) \\ -qP-(J-R)^{\top} & r P \rix \succ 0
\Big\},
\end{eqnarray}
where $0 < \theta < \pi/2$, $ h < k$ and $r > 0$.
Note that these sets are non-convex and open.  From an optimization point of view, it does not make much sense to optimize on such sets since the optimal solution(s) may not be attained. Therefore, we will consider the closure of these sets: this amounts to replacing all constraints involving a positive definite constraint with a positive semidefinite  constraint, that is,
replace $\succ 0$ with $\succeq 0$,
in the definition of the sets~\eqref{eq:omegaC}, \eqref{eq:omegaV} and~\eqref{eq:omegaD}.
We will denote the corresponding sets as
 $\bar{\mathbb S}_{\Omega_C(a,\theta)}^{n,n}$,
$\bar{\mathbb S}_{\Omega_V(h,k)}^{n,n}$, and
$\bar{\mathbb S}_{\Omega_D(-q,r)}^{n,n}$, respectively. Note that by considering the closure of these sets, as done in~\cite{GilS16}, we do not change the value of the infimum of~\eqref{eq:prob_def_omega}. 

Finally, given
 $a_j$ and $0 < \theta_j < \pi/2$ for $1 \leq j \leq p$,
$h < k$,
and several disks of parameters $(q_i,r_i)$ $1 \leq i \leq k$,  we tackle~\eqref{eq:prob_def_omega} by solving
\begin{equation} \label{finalform}
\inf_{J,R,P} {\|A - (J-R)P^{-1}\|}_F^2
\quad
\text{ such that }
\quad
 (J-R)P^{-1}
\in  \bar{\mathbb S}_\Omega^{n \times n},
\end{equation}
where
\begin{equation} \label{omega:finalform}
\bar{\mathbb S}_\Omega^{n \times n}
\; = \;
\cap_{j=1}^p
\bar{\mathbb S}_{\Omega_C(a_j,\theta_j)}^{n,n}
\cap
\bar{\mathbb S}_{\Omega_V(h,k)}^{n,n}
\cap_{i=1}^k
\bar{\mathbb S}_{\Omega_D(q_i,r_i)}^{n,n}.
\end{equation}
 The feasible set of the above optimization problem only involves convex linear matrix inequality constraints. Of course, the objective function is non-convex and the problem remains difficult,
 but it is arguably easier, from an algorithmic point of view, to handle a non-convex objective function rather than a non-convex feasible set.

\paragraph{Implementation} 

There is a key difference when considering the general $\Omega$-stability problem: as opposed to the simpler continuous-time stability case, 
the projection onto the feasible set does not have closed form (in terms of eigenvalue decompositions of symmetric matrices), and hence the corresponding optimization problem is more difficult to handle. 
In particular, first-order methods that use projection onto the feasible set become much more expensive. 
In fact, as far as we know, to obtain a high-precision projection onto such general PSD matrix sets, only IPMs are available, running in $\mathcal{O}(n^6)$  operations, instead of the $\mathcal{O}(n^3)$ requires for eigenvalue decompositions.  
In~\cite{choudhary2020approximating}, we used a BCD scheme based on IPMs to solve the subproblems in $(J,R)$ and $Q$ alternatively. 

Several initialization are possible, in particular the identity initialization where $Q = I_n$, while $(J,R)$ are optimally computed. For other initializations and discussions, we refer to~\cite{choudhary2020approximating}.

\subsection{Numerical example with the AC7 matrix} 

Let us illustrate this with the AC7 matrix, 
and the set  
\[
\Omega \quad = \quad 
\Omega_C(0,7 \pi / 8) 
\; \cap \; 
\Omega_C(-0.5,+\infty) 
\; \cap \; 
\Omega_D(1,1). 
\] 
Using BCD and the identity initialization, the obtained approximation is displayed on 
Figure~\ref{AC7_nearestJRQ_Omega}.  
\begin{figure}[ht!]
\begin{center}
\includegraphics[width=0.9\textwidth]{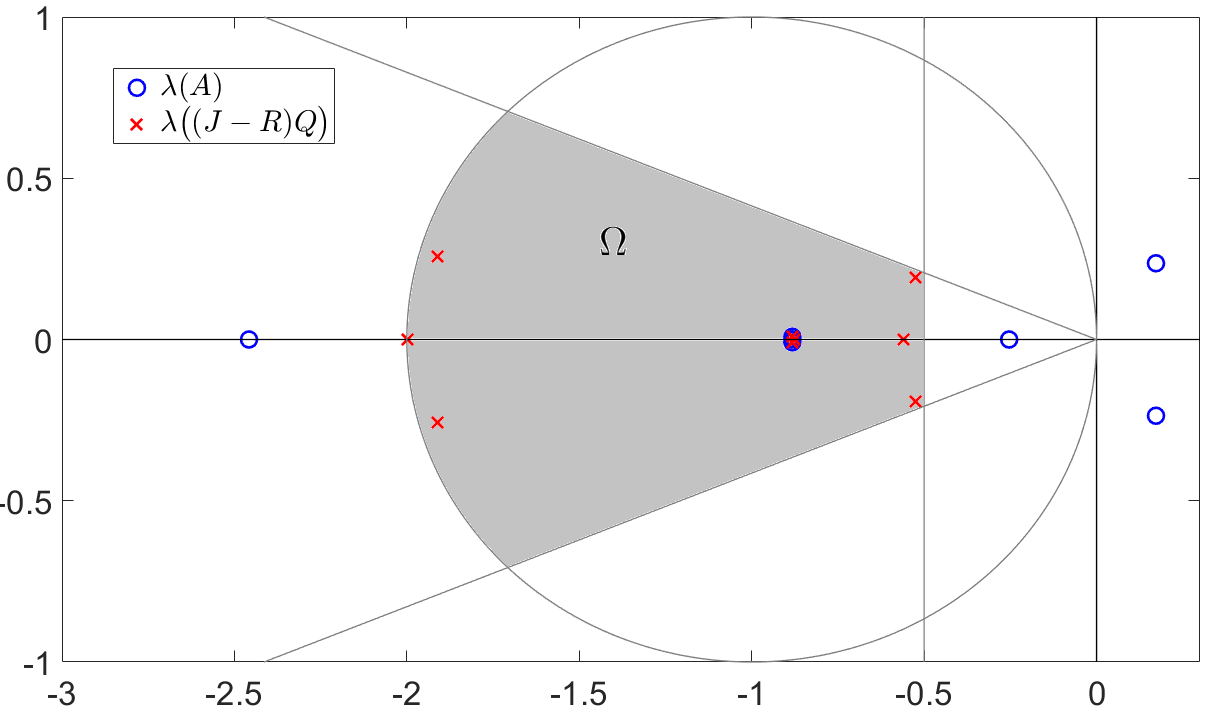}
\caption{Eigenvalues of $A$, and its $\Omega$-stable approximation $(J-R)Q$. Note that the eigenvalues -20 and -30 of $A$ are not displayed.  \label{AC7_nearestJRQ_Omega}}
\end{center}
\end{figure}

\section{Other approaches to tackle the nearest stable matrix problem} \label{sec:otherapproaches} 

In this section, we briefly mention other approaches to tackle the nearest stable matrix problem.

\paragraph{Successive convex approximation} 

In~\cite{ONV13}, authors propose an iterative approach. The main idea is as follows: 
at iteration $k$, the algorithm constructs an SDP-representable set (namely, an ellipsoid) around the current stable approximation $X^{(k)} \in \mathbb S^{n,n}$ of $A$, such that this set is contained within $\mathbb S^{n,n}$. 
This convex set relies on the Lyapunov equation: if $X$ is stable, then there exists $P \succ 0$ such that $XP + PX^\top \preceq 0$. 
The next iterate is computed as the nearest matrix to $A$ within that set. 

A drawback of this approach is that it is computationally expensive, requiring to solve an SDP in $O(n^3)$ variables at each step.

\paragraph{Matrix stabilization using differential equations} 

Guglielmi and Lubich~\cite{guglielmi2017matrix} proposed yet another completely different approach using differential equations. 
They optimize directly the norm of the perturbation, $\Delta A$, such that $A+\Delta A$ is stable.  

A main advantage of their approach is that it can easily handle structure of the stable approximation, $A+\Delta A$, that is, 
it can directly enforce some structure, 
 such as symmetry or a sparsity pattern, 
 on the stable approximation, which is not trivial when using the DH form, $(J-R)Q$. 
 However, in the unstructured case, 
 the two approaches perform similarly, 
 as reported in~\cite{guglielmi2017matrix}.

 \paragraph{Riemannian optimization}  
 
 Noferini and Poloni \cite{noferini2020nearest} recently proposed 
 a highly efficient approach to compute the nearest $\Omega$-stable matrix. They parametrize $X$ with its complex Schur factorization/decomposition, $X = U T U^*$ 
 where $U$ is unitary (that is, $U U^* = I$) 
 and $T$ is upper triangular. They observe that, if $U$ is fixed, then there is an easy solution to the simplified problem
in the variable $T$ only. 
As a consequence, finding an
$\Omega$-stable matrix nearest to $A$ is equivalent to minimizing a certain function (depending both on $\Omega$ and on $A$) over the matrix Riemannian manifold of unitary matrices $U(n)$. 
After the reformulation, the authors rely on the software for  optimization over manifolds, Manopt~\cite{boumal2014manopt}.  
 The code is available from~\url{https://github.com/fph/nearest-omega-stable}. 
 
 Their approach outperforms the previously introduced methods; in particular that based on DH matrices. 
 However, as for DH matrices, this algorithm cannot easily handle structure.

 \section{Application: minimal-norm static feedbacks} \label{sec:minnomfeed}  

In this section, we consider a continuous linear-time invariant (LTI) system in the form
\begin{eqnarray*}
\dot{x}(t) &=& A x(t) + Bu(t), \\
y(t) &=& Cx(t), 
\end{eqnarray*}
and discuss two ways to stabilize it, depending on the choice of the input $u(t)$.

 \subsection{Static-state feedback}
 
 Stabilizing the system pair $(A,B)$
using feedback controllers is a fundamental one, and is referred to as the static-state feedback (SSF) problem. 
In this setting, the feedback is chosen as $u(t) = - K x(t)$, so that 
$\dot{x}(t) = (A - B K) x(t)$.   Therefore, it requires to find $K \in \R^{m,n}$ such that $A-BK$ is stable.

Note that, in the SSF problem, the state must be measured and this is not always the case. Otherwise the state must be estimated from measurements 
of $y(t)$ and $u(t)$.   
In practice, it often makes more sense to control using the output, referred to as the static-output feedback problem, using $u(t) = - K y(t)$; see for example the discussion in~\cite{Ferrarinotes}. 
However, this is a numerically much more difficult problem to solve, and is discussed in the next section.

Typically, one requires to minimize the norm of the feedback matrix, that is,  solve
\begin{equation} \label{eq:ssf} 
\inf_{K} \, {\|K\|} \quad \text{ such that } \quad A-BK \text{ is stable},
\end{equation} 
where ${\|\cdot\|}$ is a given norm such as the $\ell_2$ norm, ${\|\cdot\|}_2$, or the Frobenius norm, ${\|\cdot\|}_F$.

In~\cite{gillis2020minimal}, we used the DH form to solve this problem. In view of Theorem~\ref{th:stableDHmat}, the following theorem is relatively straightforward. 
\begin{theorem}\cite[Theorem 4]{gillis2020minimal} \label{th:SOFth}
Let $A\in \R^{n,n}$ and $B\in \R^{n,m}$. Then the following are equivalent.
\begin{enumerate}

\item There exists $K$ such that $A-BK$ is stable. 

\item There exists a DH matrix $(J-R)Q$ such that $A-BK=(J-R)Q$
for some $K\in \R^{m,n}$.

\item There exists a DH matrix $(J-R)Q$ such that $(I_n-BB^{\dagger})(A-(J-R)Q)=0$.

\end{enumerate}
\end{theorem}

 Theorem~\ref{th:SOFth} allows to find a feasible solution to the SSF problem, solving 
\begin{equation}\label{eq:algo1ab}
\mu \; := \; \inf_{J,R,Q\in \R^{n,n}, J^\top=-J,R\succeq 0,Q\succ 0}{\left \|(I_n-BB^\dagger)(A-(J-R)Q)\right \|}, 
\end{equation} 
and checking whether $\mu = 0$. However, this problem is non-convex, because of the product $(J-R) Q$ in the objective. 
However, checking whether $\mu = 0$ can be done via solving an SDP: using the 
change of variable $P = Q^{-1}$, $\mu$ is equal to zero if and only if the followinh infimum is equal to zero 
\[
\inf_{J,R,Q\in \R^{n,n}, J^\top=-J,R\succeq 0, P\succeq I}{\left \|(I_n-BB^\dagger)(AP-(J-R))\right \|}, 
\]
where we imposed, w.l.o.g., that $P\succeq I$ to avoid the trivial solution (namely $J=R=0$). 
Note that this provides a new way to check whether the pair $(A,B)$ is stabilizable. 
 
Given an optimal solution, with $\mu = 0$, one can then solve for example (see~\cite{gillis2020minimal})
\[
\inf_{K,J=-J^\top,R\succeq 0,Q \succeq 0} \, {\|B^\dagger ( A - (J-R) Q )\|} \quad \text{ such that } \quad A-BK = (J-R) Q,  
\]
using a block coordinate descent method. 
It is however possible to reformulate the problem to obtain a convex feasible set: using $P = Q^{-1}$ again, we obtain  
\[
\inf_{J=-J^\top, R\succeq 0, P \succ 0} 
\, {\| B^\dagger ( A - (J-R) P^{-1} ) \|} 
\quad \text{ such that } \quad (I_n-BB^\dagger)(AP-(J-R)) = 0.   
\]  
To solve this new problem, we rely on a trust-region approach (see Section~\ref{sec:trustregion}). The model of the objective is obtained by linearizing  $B^\dagger ( A - (J-R) P^{-1} )$ at each step, so that the objective is quadratic when using the Frobenius norm. More precisely,  given an initial solution $(J,R,P)$, 
we look for $(\Delta J,\Delta R, \Delta P)$ such that
$(J+\Delta J,R+\Delta R, P+\Delta P)$ is a better solution than $(J,R,P)$.
To do so, we linearize the term $( A - (J+\Delta J-(R+\Delta R))  (P+\Delta P)^{-1})$ by using 
\[
(P+\Delta P)^{-1} \approx P^{-1} - P^{-1} \Delta P P^{-1},
\]
and removing the non-linear terms appearing in the product, that is, we use the following approximation:
\begin{eqnarray*}
 A - (J+\Delta J-(R+\Delta R))  (P+\Delta P)^{-1}
\approx
 A - (J+\Delta J-(R+\Delta R)) P^{-1} +(J-R)  P^{-1} \Delta P P^{-1}.
\end{eqnarray*}
This results in the following optimization problem
\begin{align}
\inf_{\Delta J, \Delta R, \Delta P\in \R^{n,n}} &
{\left \|
B^\dagger \left(
A - (J+\Delta J)P^{-1}+(R+\Delta R) P^{-1} +(J-R)  P^{-1} \Delta P P^{-1}
\right)
\right \|}  \nonumber  \\
  \quad  \text{ such that }
& \quad
\Delta J^\top=- \Delta J, R+\Delta R\succeq 0,P+\Delta P\succ 0, \label{eq:algo2}  \\
&  \quad
( I - B B^{\dagger} ) ( A \Delta P - (\Delta J - \Delta R) ) = 0 ,
\nonumber \\
&  \quad
\|\Delta J\| \leq \epsilon \| J\|,
\|\Delta R\| \leq \epsilon \| R\|,
\|\Delta P\| \leq \epsilon \| P\|. \nonumber
\end{align}
Similar to a trust-region method, the value of $\epsilon$ is updated in the curse of the algorithm. As long as
the error of $(J+\Delta J,R+\Delta R,P+\Delta P)$ is larger than that of $(J,R,P)$, $\epsilon$ is decreased. For the next step, $\epsilon$ is increased to allow a larger trust-region radius.
Since~\eqref{eq:algo2} is an SDP, we refer this this approach to as the sequential SDP (SSDP) method. 
Numerical experiments showed that  
SSDP performs significantly better than BCD on this problem; see~\cite[Table A.1]{gillis2020minimal}.

 \subsection{Static-output feedback}

 A closely related problem is that of stabilizing the system triplet $(A,B,C)$, referred to as the static-output feedback (SOF) problem. 
The goal is to stabilize the system via the output, that is, to take $u(t) = - Ky(t)$ so that 
\[
 \dot{x}(t) = A x(t) + B  u(t) 
  = A x(t) -  B K y(t) 
  = A x(t) -  B K C x(t) 
  = (A - BKC) x(t), 
\] 
is stable. The SOF problem therefore requires to find $K \in \R^{m,p}$ such that
$A-BKC$ is stable, if possible; 
see~\cite{SyrADG97} for a survey on the SOF problem. 
This decision problem is believed to be NP-hard as no polynomial-time algorithm is known; let us quote~\cite{blondel2000survey}: 
\begin{quote} The SOF problem is widely studied and still unsolved\dots a satisfactory answer to this problem has yet to be
found. This problem is often cited as one of the difficult open
problems in systems and control. Still, despite various attempts, 
it is unclear whether the problem is NP-hard. 
\end{quote}

 Note that the difficulty is that even finding a feasible 
solution is hard, as opposed to the SSF problem. 
Similarly as for the SOF problem (see Theorem~\ref{th:SOFth}), finding $K$ such that $A-BKC$ is stable is equivalent to find a DH form for it, that is, $A-BKC = (J-R)Q$.

Using similar derivations as for the SSF problem,  
  finding a feasible solution of the SOF problem is equivalent to finding an optimal solution with objective function equal to zero of the following  optimization problem 
\begin{equation}\label{eq:refabc1}
\inf_{J,R,Q\in \R^{n,n}, J^T=-J,R\succeq 0,Q\succ 0}{\left \|(I_n-BB^\dagger)(A-(J-R)Q)\right \|}
+ {\left \|(A-(J-R)Q)(C^\dagger C-I_n)\right \|}. 
\end{equation} 
To tackle this problem, we use an SSDP approach. 

If a feasible solution is found, which has the form $K = B^\dagger(A-(J-R)Q)C^\dagger$, we refine it by  considering 
\begin{align}
\inf_{(J,R,Q), (J-R)Q \text{ is a DH matrix}} & \quad {\left\|B^\dagger(A-(J-R)Q)C^\dagger\right\|}, \nonumber \\
 \text{such that} & \quad
 (I_n-BB^\dagger)(A-(J-R)Q)=0,  \text{ and }   \label{eq:formulationPabc} \\
 & \quad (A-(J-R)Q)(C^\dagger C-I_n)=0. \nonumber
\end{align}
To solve~\eqref{eq:formulationPabc}, we cannot use SSDP because the  constraints cannot be linearized exactly (we would obtain an infeasible solution after one step). 
 Instead, we resort to BCD: alternatively solve \eqref{eq:formulationPabc} for $(J,R)$ with $Q$ fixed, and then for $Q$ with $(J,R)$ fixed.

 \subsection{Numerical example with the AC7 matrix} \label{AC7:BC}

For the AC7 system, the matrix $B$ is given by 
\[ 
B = [0, 
     0, 
     0, 
     0, 
     0, 
    30, 
     0, 
     0, 
     0]^\top. 
\]  
 Using the SSDP approach, we obtain 
 \[
 K = \left( \begin{array}{ccccccccc} 
 -0.027 &  -0.018 &  -0.013 &  -0.066 &  -0.005 &  -0.005 &  -0.001 &  0 &  0 \\ 
\end{array} \right), 
 \]
 with $\|K\|_F = 0.075$, 
 and the eigenvalues of $A-BK$ are given by 
\[
[  -0.005 + 0.34i, 
  -0.005 - 0.34i, 
  -0.017, 
  -0.88 + 0.001i, 
  -0.88 - 0.001i, 
  -0.88,   
  -2.37,  
 -19.56,  
 -30.24  ]. 
\]

For the static-output feedback, we have 
\[ 
C = \left( \begin{array}{ccccccccc} 
 -0.005 &  0.476 &  0.001 &  -0 &  0.034 &  0 &  0.005 &  -0.031 &  0 \\ 
 0 &  0 &  1 &  0 &  0 &  0 &  0 &  0 &  0 \\ 
\end{array} \right). 
\] 
Using the SSDP approach, we obtained $K = [-0.3635 \,   -0.3379]$  with $\|K\|_F = 0.36$, for which the eigenvalues of $A-BKC$ are given by 
\[
[  
  -0.0004 \pm 0.1931i, 
-0.51, 
  -0.8821 \pm 0.0088i, 
  -0.8821, 
  -3.18, 
   -16.54, 
  -32.14]. 
\]
 
 \begin{remark} 
 Using the results from Section~\ref{sec:omegastab} to find the nearest $\Omega$-stable matrix, it is possible to design SSF (resp.\ SOF) such that $A-BK$ (resp.\ $A-BKC$) is $\Omega$ stable, adding proper constraints on $J$, $R$ and $Q$.  
 \end{remark}


\chapter{Nearest positive-real system and nearest stable matrix pair}   \label{chap:posrealsys}

This chapter shows that the set of linear PH systems can be exploited to compute a nearby positive real (PR) system to a given non PR system $(E,A,B,C,D)$, as we have done in~\cite{gillis2018finding}. 
This framework is then used to find a nearby regular, stable, and index one system to a given descriptor system $E\dot{x}=Ax$, which is our result from~\cite{gillis2018computing}. 

\begin{remark} 
Note that \cite{gillis2018computing} appeared before \cite{gillis2018finding}, as we first worked on the simpler descriptor system $E\dot{x}=Ax$, before extending this result to general systems \eqref{eq2:sys}. 
However, as the result from \cite{gillis2018computing} is a special case of  \cite{gillis2018finding}, we first present the result from \cite{gillis2018finding} in Section~\ref{sec:nearestPR}, and then explain the specificities of the descriptor system $E\dot{x}=Ax$ in Section~\ref{sec:nearestdessys}. 
\end{remark}

\section{Introduction} \label{sec:theproblem}

The nearest system problems can be formulated in a generic way as follows: 
\begin{problem}\label{prob_g}
For a given system $(E,A,B,C,D)$ and a given set $\mathcal{D}$, find the nearest system $(\tilde E,\tilde A,\tilde B,\tilde C, \tilde D) \in \mathcal{D}$ to $(E,A,B,C,D)$, that is, solve
\begin{equation*}\label{def_F}
\inf_{(\tilde E,\tilde A,\tilde B,\tilde C, \tilde D) \in \mathcal{D}}
\mathcal{F}(\tilde A,\tilde B,\tilde C,\tilde D,\tilde E),
\end{equation*}
where
\begin{equation}\label{eq:def_F}
\mathcal{F}(\tilde A,\tilde B,\tilde C,\tilde D,\tilde E) =
{\|A-\tilde A\|}_F^2+{\|B-\tilde B\|}_F^2
+{\|C-\tilde C\|}_F^2+{\|D-\tilde D\|}_F^2
+ {\|E-\tilde E\|}_F^2.
\end{equation}
\end{problem}

By choosing the set $\mathcal D$, one can define various nearness problems for the system~$(E,A,B,C,D)$.   The goal of this chapter is to consider the following variants of this problem for continuous-time systems (see Section~\ref{sec:defprsys} for the definitions of PR, ESPR and admissible systems): 
\begin{enumerate}
\item \underline{Nearest PR system} ($\mathcal P$):~
$\mathcal{D} = \mathbb S$
where $\mathbb S$ is the set of all PR systems
$(\tilde E,\tilde A,\tilde B,\tilde C, \tilde D)$.

\item \underline{Nearest ESPR system} ($\mathcal P_e$):~
$\mathcal{D} = \mathbb S_e$ where $\mathbb S_e$ is the set of all admissible ESPR systems
$(\tilde E,\tilde A,\tilde B,\tilde C, \tilde D)$ with $\tilde D+\tilde D^T \succ 0$.

\item \underline{Nearest admissible system} ($\mathcal P_a$):~$\mathcal{D} = \mathbb S_a$ where $\mathbb S_a$ is the set of all admissible descriptor systems
$(\tilde E,\tilde A,\tilde B,\tilde C,\tilde D)$. \label{problem:Pa}
\end{enumerate}

We will also consider the variants of ($\mathcal P$) and ($\mathcal P_e$) for standard systems with the additional constraints that $\tilde{E} = E = I_n$. The corresponding variant of ($\mathcal P_a$) is the nearest stable matrix problem considered in Chapter~\ref{chap:contsys}.
These problems are challenging because the feasible sets $\mathbb S$, $\mathbb S_e$, and $\mathbb S_a$ are unbounded, highly nonconvex, and neither open nor closed~\cite{gillis2018finding,choudhary2020approximating}.

\section{Nearest PR system problem} \label{sec:nearestPR}

As mentioned in Section~\ref{sec:defprsys},  the positive realness of an LTI dynamical system is equivalent to passivity, which means that the system does not generate energy. 
Since passivity and positive realness are equivalent for LTI systems, the distance to positive realness has
direct  applications in  passive model approximations (see Section~\ref{sec:bad2good}). 

The nearest PR system problem is complementary with the distance to
nonpassivity for control systems; see~\cite{OveV05} for complex standard systems. These problems are closely related to the Hamiltonian matrix nearness problems~\cite{AlaBKMM11,GugKL15}.
Several algorithms tackle this problem using the spectral
properties of the related Hamiltonian/skew-Hamiltonian matrices  or pencils
for the input systems that are asymptotically stable, controllable,
observable, and almost passive; see~\cite{Tal04,SchT07,WanZKPW10,VoiB11,BruS13} and the references therein.

As far as we know, except~\cite{gillis2018finding}, no other algorithm exists for the nearest PR system problem that does not make any assumption on the input system and that allows perturbations to all matrices $(E,A,B,C,D)$ describing the system.
In the following, we explain the algorithm proposed in~\cite{gillis2018finding}, which is based on the generalization of the results from~\cite{GilS16}, where authors used the structure of PH systems to find a nearby stable standard system to an unstable one.  As opposed to the previously proposed methods, this algorithm is not based on the spectral properties of Hamiltonian matrices or pencils. It can be applied to any given LTI dynamical system.


\subsection{Reformulation of  ($\mathcal P_e$) using PH systems}

We first discuss the link between PR systems and PH systems.
The positive realness of a system~\eqref{eq2:sys} can be characterized in terms
of solutions $X$ to the following linear matrix inequalities (LMIs):
\begin{equation}\label{eq:LMI1}
\mat{cc} A^\top X +X^\top  A & X^\top B-C^\top  \\B^\top X-C & -D-D^\top  \rix \preceq 0
\quad \text{and} \quad E^\top X=X^\top E \succeq 0.
\end{equation}
\begin{theorem}[\cite{FreJ04}, Theorem~3.1] \label{thm:suff_PR}
Consider a regular system $(E,A,B,C,D)$ in the form~\eqref{eq2:sys}. If the LMIs~\eqref{eq:LMI1} have a solution $X \in \R^{n,n}$, then
$(E,A,B,C,D)$ is PR.
\end{theorem}

The converse of Theorem~\ref{thm:suff_PR} is true with some additional assumptions.
In fact, the positive real lemma for standard systems~\cite{AndV73} proves that
if a system is PR and minimal, then a solution to the LMIs~\eqref{eq:LMI1} is also necessary.
Similarly, with an additional condition, the positive real lemma for descriptor systems~\cite{FreJ04} proves that  the existence of a solution to the LMIs~\eqref{eq:LMI1} is also necessary for positive realness.

Theorem~\ref{thm:suff_PR} gives an alternative way, compared to the one described in~\eqref{HHyu}, 
to show that every PH system is positive real by providing an explicit solution $X$ to~\eqref{eq:LMI1}~\cite[Theorem 3.2]{gillis2018finding}. Similarly, if the LMIs~\eqref{eq:LMI1} have an invertible solution $X$, then the system $(E,A,B,C,D)$ can be written as a PH system~\cite[Theorem 3.6]{gillis2018finding}.

In the following, a necessary and sufficient condition for a system in the form~\eqref{eq2:sys}
to be ESPR is obtained in terms of the existence of a solution of the LMIs~\eqref{eq:LMI1}. 

\begin{theorem}[\cite{ZhaLX02}, Theorem~2] \label{thm:espr_equi}
Let $(E,A,B,C,D)$ define a system~\eqref{eq2:sys}. Then
it is admissible, ESPR and satisfies $D+D^\top  \succ 0$
if and only if
there exists a solution $X$ to the LMIs
\begin{equation}\label{eq:LMI2}
\mat{cc} A^\top X +X^\top  A & X^\top B-C^\top  \\B^\top X-C & -D-D^\top  \rix \prec 0
\quad \text{and} \quad E^\top X=X^\top E \succeq 0.
\end{equation}
\end{theorem}	

Theorem~\ref{thm:espr_equi} will be used to  characterize the set of all admissible ESPR systems in terms of PH systems. For this, let us define the \emph{PH-form} for a system~\eqref{eq1:sys}.

\begin{definition}{\rm
A system $(E,A,B,C,D)$ is said to admit a \emph{port-Hamiltonian form (PH-form)} if
there exists a PH system as defined in~\eqref{eq:phsystem} such that
\[
A=(J-R)Q,\quad B=F-P, \quad C=(F+P)^\top Q, \quad \text{and}\quad D=S+N ,  
\] 
see page~\pageref{eq:phsystem} for the constraints on these matrices. 
} 
\end{definition}

In the following, we state and prove several equivalent characterizations of a system to be admissible and ESPR. 

\begin{theorem}\label{thm:mainprequi}
Let $\Sigma=(E,A,B,C,D)$ be a system in the form~\eqref{eq2:sys}. Then the following are equivalent.
\begin{enumerate}
\item[(i)]  $\Sigma$ is admissible and ESPR with $D+D^\top  \succ 0$.
\item [(ii)]There exists a solution $X$ to the LMIs~\eqref{eq:LMI2}.
\item [(iii)] $\Sigma$ admits a  PH-form with positive definite cost matrix $K = \mat{cc}  R & P \\ P^\top & S  
\rix$. 
\end{enumerate}
\end{theorem}
\begin{proof} $(i) \Longleftrightarrow (ii)$ follows from Theorem~\ref{thm:espr_equi}. 

$(iii)\Longrightarrow (ii)$.~
Suppose  $\Sigma$ admits a PH-form  with positive definite cost matrix, and let \mbox{$A=(J-R)Q$}, $B=F-P$, $C=(F+P)^\top Q$ and $D=S+N$, where $J^\top =-J$,
$N^\top =-N$, $K=\mat{cc}R & P\\P^\top  & S \rix \succ 0$, $Q$ is invertible and $E^\top Q=Q^\top E \succeq 0$. Then $X=Q$ satisfies the LMIs in~\eqref{eq:LMI2}.  In fact, we have
\begin{align*}
&\mat{cc} A^\top Q +Q^\top  A & Q^\top B-C^\top  \\B^\top Q-C & -D-D^\top  \rix\\
&=\mat{cc} ((J-R)Q)^\top Q +Q^\top  (J-R)Q & Q^\top (F-P)-((F+P)^\top Q)^\top  \\(F-P)^\top Q-(F+P)^\top Q & -(S+N)-(S+N)^\top  \rix\\
&= -2 \mat{cc} Q^\top RQ & Q^\top P \\P^\top Q & S \rix = -2\mat{cc}Q^\top  &0\\0 &I_m\rix
\mat{cc} R & P \\P^\top  & S \rix
\mat{cc}Q & 0\\0 &I_m\rix \prec 0,
\end{align*}
because $K \succ 0$ and $Q$ is invertible.

$(ii)\Longrightarrow (iii)$.~ Suppose there exists a solution $X$ to the LMIs~\eqref{eq:LMI2}. This implies that $A^\top X +X^\top  A \prec 0$, and therefore $X$ is invertible. Define
\begin{eqnarray}
 &J:=\frac{AX^{-1}-(AX^{-1})^\top }{2}, \quad R:=-\frac{AX^{-1}+(AX^{-1})^\top }{2}, \quad Q:=X, \quad
S:=\frac{1}{2}(D+D^\top ), \nonumber \\
& N:=\frac{1}{2}(D-D^\top ),\quad F:=\frac{1}{2}(B+X^{-1}C^\top ), \quad \text{and}\quad
P:=\frac{1}{2}(-B+X^{-1}C^\top ). \label{DHformcstr}  
\end{eqnarray} 
 Let us show that the matrices $J,R,Q,F,P,N$ and $S$ provide a PH-form for $\Sigma$. We have
\[
(J-R)Q=A,\quad F-P=B, \quad (F+P)^\top Q =C,\quad \text{and}\quad S+N=D.
\]
Further, we have that $E^\top Q \succeq 0$ (using the second LMI in~\eqref{eq:LMI2}), $J^\top =-J$, $N^\top =-N$, and
\begin{eqnarray*}
K&=&\mat{cc}R &P\\P^\top  & S \rix = -\frac{1}{2}\mat{cc} AX^{-1}+X^{-1}A^\top  & -B+X^{-1}C^\top  \\ -B^\top +CX^{-1} & -D-D^\top  \rix \\
&=& -\frac{1}{2}\mat{cc} -X^{-1} & 0 \\0& I_m \rix
\mat{cc} A^\top X +X A & XB-C^\top  \\B^\top X-C & -D-D^\top  \rix
\mat{cc} -X^{-1} & 0 \\0& I_m \rix
\succ 0,
\end{eqnarray*}
which follows from the first LMI in~\eqref{eq:LMI2}.
\end{proof}

We reformulate the nearest ESPR system problem $(\mathcal P_e)$ using the
PH-form for an admissible ESPR system
$(E,A,B,C,D)$ with $D+D^\top  \succ 0$.
For a standard ESPR system $(I_n,A,B,C,D)$ we have that $D+D^\top  \succ 0$,
thus the condition $D+D^\top  \succ 0$ for standard systems is redundant.
However, the PH-form characterization of an admissible ESPR descriptor system depends on the existence of a solution of the LMIs~\eqref{eq:LMI2} when $D+D^\top  \succ 0$.
This justifies the restriction $D+D^\top  \succ 0$ on defining the set $\mathbb S_e$ for the nearest ESPR system problem
in Section~\ref{sec:theproblem}.
Let us define the following two sets:
\begin{itemize}
\item The set $\mathbb S_{PH}$
containing all systems $(E,A,B,C,D)$ in PH-form, that is,
\begin{eqnarray*}
\mathbb S_{PH} &:=& \left\{(E,A,B,C,D) \; | \ (E,A,B,C,D) \text{ admits a PH-form}\right\} \\
&=& \Big\{ (E,(J-R)Q,F-P,(F+P)^\top Q,S+N) \; \Big|  \;
J^\top =-J, N^\top =-N, \\
&& \hspace{3cm} E^\top Q \succeq 0, Q \text{ invertible},
K=\mat{cc}R &P \\
P^\top  & S \rix \succeq 0 \Big\}.
\end{eqnarray*}

\item The set $S_{PH}^{\succ 0} \subset \mathbb S_{PH}$
containing all systems $(E,A,B,C,D)$ in strict PH-form, that is,
\begin{eqnarray*}
   \mathbb S_{PH}^{\succ 0}  := \Big\{(E,(J-R)Q,F-P,(F+P)^\top Q,S+N) \in \mathbb S_{PH}
	\; \Big| \; K 
	\succ 0 \Big\}.
\end{eqnarray*}
By Theorem~\ref{thm:mainprequi}, $\mathbb S_e =  \mathbb S_{PH}^{\succ 0}$.
\end{itemize}

The sets $\mathbb S_{PH}$ and $\mathbb S_{PH}^{\succ 0}$ are neither closed (due to the constraint that $Q$ is invertible)
nor open (due to the constraint $E^\top Q \succeq 0$). This gives another way to see that the set $\mathbb S_e$ of all
ESPR systems is neither open nor closed. 
Consider the closure $\overline{\mathbb S_{PH}}$ of $\mathbb S_{PH}$, which is equal to the set $\mathbb S_{PH}$ except that $Q$ can be singular. Moreover, we have that $\overline{\mathbb S_{PH}} = \overline{\mathbb S_{PH}^{\succ 0}}$.
Therefore the values of the infimum over the sets $\mathbb S_{PH}$, $\mathbb S_{PH}^{\succ 0}$, and
$\overline{\mathbb S_{PH}}$ are the same. We have the following result.
\begin{theorem}\label{thm:reform_in_ph}
Let $(E,A,B,C,D)$ be a system in the form~\eqref{eq1:sys} and $\mathcal{F}$ be defined as in~\eqref{eq:def_F}.
Then
\begin{equation} \label{eq:reform_in_ph}
\inf_{ (M,(J-R)Q,F-P,(F+P)^\top Q,S+N) \in \overline{\mathbb S_{PH}} }  \quad
\mathcal{F}((J-R)Q,F-P,(F+P)^\top Q,S+N,M)
\end{equation}
coincides with the infimum of $(\mathcal P_e)$  while it is
is an upper bound for the infimum of $(\mathcal P)$.
\end{theorem}
\begin{proof}
 This follows directly from the fact that $\mathbb S_e =  \mathbb S_{PH}^{\succ 0}$ and $\mathbb S_e \subseteq \mathbb S$.
\end{proof}
We will refer to~\eqref{eq:reform_in_ph} as the \emph{nearest PH system problem}.
The same result holds for the variants of $(\mathcal P)$ and $(\mathcal P_e)$ for standard systems
since the only difference is that $M$ is imposed to be equal to $E=I_n$.

Although the value of the infimum in~\eqref{eq:reform_in_ph} coincides with the infimum of $(\mathcal P_e)$,
the solution of~\eqref{eq:reform_in_ph} may not solve the problem $(\mathcal P_e)$,
as the solution found may not even be PR since $\overline{\mathbb S_{PH}}$ could contain systems which are not regular.
To rule out such situations,
one can impose the matrix $R$ to satisfy $R\succeq \delta I_n$ for some fixed
small $\delta >0$, because in this case $(E,(J-R)Q)$ is a DH matrix pair with positive definite
$R$ and therefore the system is guaranteed to be regular by Corollary~\ref{cor:DJ_regular}.
This does not complicate the problem as
the projection is still straightforward but gives a nearby regular
descriptor PH system (hence a PR system, see~\cite[Theorem 3.2]{gillis2018finding})
to a given system.

In view of~\eqref{eq:reform_in_ph}, solving problem ($\mathcal P_e$) is equivalent to solving the nearest PH system problem. We briefly explain this separately for standard systems when $E=I_n$ and $E$ is not subject to perturbation
and for general systems when $E$ is subject to perturbation.

\subsubsection{Standard systems} \label{stansys}

For standard systems, we have $M=E=I_n$ and thus~\eqref{eq:reform_in_ph} can be simplified as follows
\begin{align}
\inf_{J,R,Q,F,P,S}
& \mathcal{G}(J,R,Q,F,P,S) \quad \text{ such that }
\quad J^\top =-J, Q\succeq 0 \text{ and } \mat{cc}R &P\\P^\top &S\rix\succeq 0 ,  \label{def:dist_sph_2}
\end{align}
where
\begin{align*}
\mathcal{G}(J,R,Q,F,P,S)
& = {\|A-(J-R)Q\|}_F^2 + {\|B-(F-P)\|}_F^2 \\
& \quad \quad + {\|C-(F+P)^\top Q\|}_F^2+{\left\|\frac{D+D^\top }{2}-S\right\|}_F^2, 
\end{align*}
since the optimal $N$ in~\eqref{eq:reform_in_ph} is given by $\frac{D-D^\top }{2}$ as $S$ is symmetric.

\subsubsection{General systems}  \label{sec:gen_algo}

Similarly as for standard systems in~\eqref{def:dist_sph_2}, \eqref{eq:reform_in_ph} can be simplified to
\begin{align}
\inf_{J,R,Q,M,F,P,S} & \mathcal{G}(J,R,Q,F,P,S) + {\|E - M\|}_F^2 \label{def:dist_dph_1} \\
& \quad \text{ such that } J^\top =-J,
M^\top  Q \succeq 0
\text{ and }
\mat{cc}R &P\\P^\top &S\rix\succeq 0 . \nonumber
\end{align}
As opposed to~\eqref{def:dist_sph_2}, it is difficult to project on the feasible domain of~\eqref{def:dist_dph_1}
because of the coupling constraint $M^\top  Q \succeq 0$. Moreover, this constraint was observed to get standard optimization schemes stuck in suboptimal solutions; see~\cite[Example 3]{gillis2018computing} for an example.
To overcome this issue, one can introduce a new variable   $Z = M^\top  Q$ so that $M^\top  = Z Q^{-1}$~\cite{gillis2018computing}.
This leads to a reformulation of~\eqref{def:dist_dph_1} into an equivalent optimization problem with a
simpler feasible set:
\begin{align}\label{eq:equivalent1}
\inf_{J,R,Q,Z,F,P,S} 
& \mathcal{G}(J,R,Q,F,P,S) + {\|E^\top  - Z Q^{-1} \|}_F^2 \\
& \text{ such that } J^\top =-J, Z \succeq 0 \text{ and } \mat{cc}R &P\\P^\top &S\rix\succeq 0 . \nonumber
\end{align}

\subsection{Optimization algorithms}\label{sec:proptim}

As it was done in Chapter~\ref{chap:contsys} to find the nearest stable matrix to an unstable one, methods like PGD (Algorithm~\ref{algo:gradient}) or FGM (Algorithm~\ref{algo:fgm}) can be used to estimate~\eqref{def:dist_sph_2} and~\eqref{eq:equivalent1}, see~\cite{gillis2018computing} for more details.

\subsubsection{Initialization}\label{sechap4:init}

The simplified optimization problems~\eqref{def:dist_sph_2} or~\eqref{eq:equivalent1} are nonconvex. This makes choosing good initial points crucial to obtain good solutions. 

\paragraph{Identity initialization}
The identity initialization uses $Q = I_n$ and $P = 0$. For these values of $Q$ and $P$, the optimal solutions for the other variables can be computed explicitly:
\[
J = \big(A-A^\top \big)/2,
R = \mathcal{P}_{\succeq} \big((-A-A^\top )/2\big),
S = \mathcal{P}_{\succeq} \big((D^\top +D)/2\big),
F = \big(B+C^\top \big)/2,
\]
and $Z = \mathcal{P}_{\succeq }(E^\top )$ for general systems. 
This initialization has the advantage of being very simple to compute while working reasonably well in many cases;
see~\cite{gillis2018finding} for numerical experiments.

\paragraph{LMI-based initializations}

Given a system that does not admit a PH-form, the LMIs~\eqref{eq:LMI1} will not have a solution. However, since we are looking for a nearby system that will admit a solution to these LMIs, it makes sense to find a solution $X$ to nearby LMIs.
We propose the following to relax the LMIs~\eqref{eq:LMI1}:
\begin{align}
\min_{\delta, X} & \quad \delta^2 \nonumber \\
\text{ such that } & \quad
\mat{cc} -A^\top X -X^\top  A & C^\top -X^\top B \\ C-B^\top X & D+D^\top  \rix + \delta I_{n+m} \succeq 0,  \label{initLMIs} \\
& \quad  E^\top  X  + \delta I_n \succeq 0. \nonumber
\end{align}
Let us denote $(\widehat \delta,\widehat X)$ an optimal solution of~\eqref{initLMIs}.
If $\widehat \delta = 0$ and $\widehat X$ is invertible,
then the system $(E,A,B,C,D)$ admits a PH-form; see~\cite[Theorem 3.6]{gillis2018finding}. 
Moreover, as long as $\widehat X$ is invertible, the matrices $(J,R,Q,S,N,P,Z)$ can be constructed using~\eqref{DHformcstr}
and projected onto the feasible set $\overline{\mathbb S_{PH}}$ to obtain an initial system in PH-form. 

If one wants to obtain a better initial point, given $Q = \widehat X$, it is possible to compute the matrices $(J,R,S,N,P)$ by solving a semidefinite program (SDP):
\begin{equation}
\min_{J,R,S,N,P} \mathcal{G}(J,R,Q,F,P,S)  \quad \text{ such that } \quad  J^\top =-J \text{ and } \mat{cc}R &P\\P^\top &S\rix\succeq 0,
\label{initLMIs2}
\end{equation}
while taking $Z = \mathcal{P}_{\succeq }(E^\top Q)$ (as $Q=\widehat X$ can be ill-conditioned).

It is observed that the LMI-based initializations work well when the initial system is close to being passive (that is, when $\widehat \delta$ is small); otherwise, it may provide rather bad initial points;  see~\cite{gillis2018finding}  for some examples.
However, in most applications, the systems of interest are usually close to being passive (cf.~\@ Section~\ref{sec:bad2good}); hence these initializations may be particularly useful.  
An interesting direction of research would be to provide theoretical guarantee for a relaxation such as~\eqref{initLMIs} to recover a nearby passive system to a system which is close to being passive.

\subsection{Numerical Example on the AC7 system} \label{numexp}

The AC7 system is a standard system, with $E=I_n$. 
Since the dimension of $y$ (=2) is not equal to that of $u$ (=1), we artificially add an input of zero, adding a column of zeros to $B$ and $D$. Recall from Section~\ref{sec:defpass} that for passivity the vectors $u$ and $y$ need to be of same length. 
The original matrices $B$ and $C$ are given in Section~\ref{AC7:BC}, while $D = [0 \, 0; 0 \, 0]$. 

Imposing the passive system to remain standard (that is, $M = I_n$), we obtain, running our code with the default initialization (identity matrix), the following errors: 
\[
\frac{\| A  - (J-R) Q\|_F}{\| A \|_F} = 1.08\%,  \; 
\frac{\| B  - (F-P) \|_F}{\| B \|_F} = 0.22\%, 
\]
\[ 
\frac{\| C  -  (F+P)^\top Q \|_F}{\| C \|_F} = 60.08\%, \; 
{\| D  - (S+N) \|_F} = 1.08, 
\]
while the global error, 
$\mathcal{G}(J,R,Q,F,P,S)$, is 1.81. 

We can also approximate this system using a descriptor system, removing the constraint that $M=E$, for which our code provides a rather different solution, with errors 
$\| M-E \|_F / \| E\|_F = 34.91\%$, 
\[
\frac{\| A  - (J-R) Q\|_F}{\| A \|_F} = 1.64\%,  \; 
\frac{\| B  - (F-P) \|_F}{\| B \|_F} = 0.50\%, 
\]
\[ 
\frac{\| C  -  (F+P)^\top Q \|_F}{\| C \|_F} = 18.62\%, \; 
{\| D  - (S+N) \|_F} = 0.30. 
\]
Giving freedom in the variable $M$ allows to reduce significantly the approximation error for $C$ and $D$, and reduce the global error, from 1.81 to 1.72.

\section{Nearest stable matrix pairs} \label{sec:nearestdessys}

Recall that a system $\Sigma=(E,A,B,C,D)$ is called admissible if it is regular, asymptotically stable, and of index at most one. Since the admissibility of the system depends solely on the matrix pair $(E,A)$, with additional constraints  that $\tilde B=B$, $\tilde C=C$, and $\tilde D=D$, the problem $(\mathcal P_a)$ is equivalent to the \emph{nearest admissible matrix pair problem}; see page~\pageref{problem:Pa}. 
More precisely, we consider the following problem: 
\begin{problem}\label{prob:contstabpair}{\rm
For a given  pair $(E,A) \in \R^{n,n}\times \R^{n,n}$
find the nearest admissible matrix pair $(M,X)$. In other words,  if $\mathbb S$ is the set of
matrix pairs $(M,X)\in \R^{n,n}\times \R^{n,n}$ that are regular, of index at most one, and have all finite eigenvalues in the open left half plane,
then we wish to compute
\begin{equation*}
\inf_{(M,X) \in \mathbb S} \{{\|E-M\|}_F^2+{\|A-X\|}_F^2\}.
\end{equation*}
}
\end{problem}

Note that this problem is a special case of the nearest PR system problem discussed in the previous section, taking $B$, $C$ and $D$ as empty matrices. However, we provide in this section some additional insight on this case; in particular regarding the characterization of DH matrix pairs (Theorem~\ref{eq:main_result}). 

In~\cite{gillis2018computing}, the authors used the \emph{nearest stable matrix pair problem} to refer to the above problem.
This problem is the complementary problem to the distance to instability for matrix pairs; see~\cite{ByeN93} for complex pairs and~\cite{DuLM13} for a survey on this problem.
Since we require a stable pair to be regular, it also complements the distance to the nearest singular pencil, which is a long-standing open problem \cite{ByeHM98,GugLM16,MehMW15,PraS21a}. 
The nearest stable matrix pair problem occurs in system identification, where one needs to
identify a stable matrix pair depending on observations (see Section~\ref{sec:bad2good}).

 As demonstrated in~\cite{gillis2018computing}, the feasible set $\mathbb S$ is not open, not closed, non-bounded, and highly nonconvex, thus it is very difficult to work directly with the set $\mathbb S$. For this reason, we reformulate the nearest stable matrix pair problem into an equivalent optimization problem with a simpler feasible set using DH matrix pairs.

\subsection{Formulation using DH matrix pairs}

Let us recall the definition of DH matrix pairs from Section~\ref{subsec:dhprop}.
\begin{definition}
A matrix pair $(E,A)$, with $E,A \in \mathbb R^{n,n}$, is called a \emph{dissipative Hamiltonian (DH) matrix pair} if there exists an invertible matrix $Q\in \mathbb R^{n,n}$ such that
$Q^\top E=E^\top Q \succeq 0$, and $A$ can be expressed as $A=(J-R)Q$  with $J^\top =-J$, $R^\top =R  \succeq 0$.
\end{definition}
The matrix $R$ in a DH matrix pair $(E,(J-R)Q)$ is called the dissipation matrix. We have seen in Theorem~\ref{thm:stable_semidef_R} that every regular, index at most one DH matrix pair $(E,(J-R)Q)$ is stable. 
The additional constraint that the dissipation matrix $R$ is positive definite guarantees that the DH matrix pair is asymptotically stable, that is, regular, of index at most one, and has all finite eigenvalues in the open left half of the complex plane. The converse of this statement that every asymptotically stable pair $(E,A)$ is a DH matrix pair with positive definite dissipation matrix is also true. 

\begin{theorem}\label{eq:main_result}
Let $(E,A)$ be a matrix pair, where $E,A \in \R^{n,n}$. Then the following statements are equivalent.
\begin{enumerate}
\item [1)] $(E,A)$ is a DH matrix pair with positive definite dissipation matrix.
\item [2)] $(E,A)$ is regular, of index at most one, and asymptotically stable.
\end{enumerate}
\end{theorem}
\begin{proof} $1)\Rightarrow 2)$ Let $(E,A)$ be a DH matrix pair with positive definite dissipation matrix, that is,
$A$ can be expressed as $A=(J-R)Q$ for~some $R\succ 0$, $J^\top =-J$, and nonsingular $Q$ with $Q^\top E \succeq 0$.
Clearly, by Corollary~\ref{cor:DJ_regular} $(E,(J-R)Q)$ is regular. Furthermore,  $(E,(J-R)Q)$ has all its finite eigenvalues in the open left half plane. To see this,
 let $\lambda \in \mathbb C$ be a finite  eigenvalue of the pencil $zE-(J-R)Q$. Then by
Lemma~\ref{lem:DH_pair_prop} it follows that $\real{(\lambda)} \leq 0$, and $\real{(\lambda)} =0$
if and only if there exists $x\neq 0$ such that $(\lambda E-JQ)x=0$ and $0\neq Qx \in \operatorname{null}(R)$. But
$\operatorname{null}(R) =\{0\}$ as $R \succ 0$. 

To show that $(E,(J-R)Q)$ is of index at most one, we
set $r:=\text{rank}(E)$ and assume that $U \in \R^{n,n-r}$ is an orthogonal matrix whose column spans $\text{null}(E)$.
Then, see~\cite{KauNC89}, $(E,(J-R)Q)$ is of index at most one if and only if $\text{rank}(\mat{cc}E&(J-R)QU\rix)=n$.
Suppose that $x \in \C^n\in \setminus \{0\}$ is such that $x^H\mat{cc}E&(J-R)QU\rix =0$. Then we have the two conditions
\begin{equation}\label{eq:equivalent_proof_1}
x^HE=0,\ x^H(J-R)QU=0.
\end{equation}
Since $Q$ is invertible, we have $x^HEQ^{-1}=0$ and hence
$(EQ^{-1})x=0$ because $EQ^{-1} \succeq 0$ as $E^\top Q \succeq 0$. This shows that $Q^{-1}x \in \text{null}(E)$,
and thus there exists $y \in \C^{n-r}$ such that $Q^{-1}x=Uy$, or, equivalently $x=QUy$.
Using this in~\eqref{eq:equivalent_proof_1}, we obtain that $x^H(J-R)x=0$. This implies that
$x^HJx=0$ and $x^HRx=0$ as $J$ is skew-symmetric and $R$ is symmetric. But this is a contradiction to the assumption that $R \succ 0$. This completes the proof of $1)\Rightarrow 2)$.

$2)\Rightarrow 1)$ Consider a  pair $(E,A)$, with $E,A\in \mathbb R^{n,n}$, that is regular, asymptotically stable, and of index at most one. Then by
Theorem~\ref{thm:impulsefree_reference}, there exist an nonsingular $V\in \mathbb R^{n,n}$ such that
$V^\top A+A^\top V \prec 0$ and $E^\top V=V^\top E \succeq 0$.
Setting
\begin{equation} \label{DHexplicit}
Q=V,\quad J=\frac{(AV^{-1})-(AV^{-1})^\top}{2},\quad \text{and}\quad R=-\frac{(AV^{-1})+(AV^{-1})^\top}{2},
\end{equation}
we have $J^\top = -J$, $E^\top Q = Q^\top E \succeq 0$, and $R \succ 0$, as $V$ is invertible. Applying the Lyapunov inequality
\[
V^\top RV=-\frac{V^\top((AV^{-1})+(AV^{-1})^\top)V}{2}=-\frac{V^\top A+A^\top V}{2} \succ 0,
\]
the assertion follows.
\end{proof}
An important consequence of the proof of Theorem~\ref{eq:main_result} is an \emph{explicit construction} of the DH characterization of a matrix pair $(E,A)$:
(i)~solve the LMIs~\eqref{LMIcharact} (if the LMIs do not admit a solution, the pair is not regular, of index at most one, and asymptotically stable),
and (ii)~use~\eqref{DHexplicit} to construct $(J,R,Q)$. 

By Theorem~\ref{eq:main_result}, the set $\mathbb S$ of all asymptotically stable matrix pairs can be expressed as the set of all DH matrix pairs with positive definite dissipation, that is,
\begin{eqnarray*}
\mathbb S\hspace{-.2cm}&=& \hspace{-.2cm}\left\{(M,(J-R)Q)\in \R^{n,n}\times \R^{n,n}\,:J^\top =-J, R\succ 0, Q\,\text{invertible s.t.}\,Q^\top M \succeq 0\right\}\\
\hspace{-.1cm}&=:&\hspace{-.2cm} \mathbb S_{DH}^{\succ 0}.
\end{eqnarray*}
This characterization changes the feasible set and also the objective function in the nearest stable matrix pair problem as
\begin{eqnarray}\label{eq:reformulation_1}
&\inf_{(M,X) \in \mathbb S} \{{\|E-M\|}_F^2+{\|A-X\|}_F^2\}\nonumber\\
&=\inf_{(M,(J-R)Q) \in \mathbb S_{DH}^{\succ 0}} \{{\|E-M\|}_F^2+{\|A-(J-R)Q\|}_F^2\} \nonumber\\
&=
\inf_{(M,(J-R)Q) \in \mathbb S_{DH}^{\succeq 0}} \{{\|E-M\|}_F^2+{\|A-(J-R)Q\|}_F^2\}, 
\end{eqnarray}
where  the set $\mathbb S_{DH}^{\succeq 0}$ containing all pairs of the form $(M,(J-R)Q)$ with
$J^\top = -J$, $R \succeq 0$ ($R$ can be singular), and $Q$ ($Q$ can be singular) such that $M^\top Q \succeq 0$, that is, $\mathbb S_{DH}^{\succeq 0}$ is the closure $\overline{\mathbb S_{DH}^{\succ 0}}$ of $\mathbb S_{DH}^{\succ 0}$.
Note that the set $\mathbb S_{DH}^{\succeq 0}$ is  not bounded, and hence the infimum in the right hand side of~\eqref{eq:reformulation_1} may not be attained.

\subsection{Optimization algorithms}

As mentioned in Section~\ref{sec:gen_algo}, the coupling constraint $Q^\top M \succeq 0$ in~\eqref{eq:reformulation_1} seems to prevent standard optimization schemes to converge to good solutions, as demonstrated in~\cite[Example 3]{gillis2018computing}. 
Similarly as in the previous section, let us introduce a new variable  $Z = M^\top Q \succeq 0$ in~\eqref{eq:reformulation_1} to obtain the following optimization problem with a modified  feasible set and objective function
\begin{equation} \label{reformoptprob}
\inf_{J = -J^\top, R \succeq 0, Q~{\rm invertibe}, Z \succeq 0} {\|A - (J-R) Q\|}_F^2 + {\| E^\top - ZQ^{-1} \|}_F^2 .
\end{equation}
Note that the values of the infimum in~\eqref{eq:reformulation_1} and~\eqref{reformoptprob} coincide. In fact,
$(M,J,R,Q)$, where $Q$ is invertible is a solution for~\eqref{eq:reformulation_1} with the optimal value $\mu$ if and only if
$(Z=M^\top Q,J,R,Q)$, where $Q$ is invertible is a solution for~\eqref{reformoptprob} with the optimal value $\mu$. This implies that the infimum in \eqref{reformoptprob} is given by
\begin{equation}\label{reformoptprob_temp}
 \inf_{J = -J^\top, R \succeq 0, Q~{\rm invertibe}, M,Q^\top M \succeq 0} {\|A - (J-R) Q\|}_F^2 + {\|E - M\|}_F^2.
\end{equation}
Furthermore, the closeness of the set $\mathbb S_{DH}^{\succeq 0}$ implies that~\eqref{reformoptprob_temp} coincides
with~\eqref{eq:reformulation_1}.

The feasible set~\eqref{reformoptprob} is rather simple, with no coupling of the variables, and it is relatively easy to project onto it. 
As it was done in Section~\ref{sec:proptim} for the nearest PR system problem, methods    like PGD (Algorithm~\ref{algo:gradient}) or FGM (Algorithm~\ref{algo:fgm}) can be used to solve the nearest stable matrix pair problem~\eqref{reformoptprob}~\cite{gillis2018computing}. In fact, the same algorithms and initializations as proposed in Section~\ref{sec:proptim} for the nearest PR system problem can be used on the system
$(E,A,[\,],[\,],[\,])$, where $[\,]$ is the empty matrix, to recover a stable approximation of $(E,A)$.

\subsection{Numerical Examples}

Again, let us consider the standard AC7 system. 
As done in Section~\ref{sec:numexpAC7stablematrix}, 
finding the nearest stable matrix to $A$ leads to a stable pair, $(I_n, (J-R)Q)$ of $(I_n,A)$. 
In this case, the matrix $E$ is untouched and remains the identity matrix. It turns out that this solution is also a stationary point of the descriptor system: relaxing the constraint that $E=I_n$ in the approximation by solving~\eqref{reformoptprob}, and starting the algorithm at the same solution does not modify the solution. 

However, for some other matrices, this is not the case; see some numerical examples in~\cite{gillis2018computing}.


\chapter{Nearest stable discrete-time systems} \label{chap:discrete}

The aim of this chapter is to derive a characterization for the discrete-time systems, with a similar spirit as the results in the previous chapters. 
We provide a factorization of $A$ whose factors belong to simple sets onto which is it easy to project, allowing us to design optimization algorithms, such as fast gradient methods, for computing the nearest stable matrix  and the nearest stable matrix pair in the discrete-time case. We note that the matrix case can be tackled using the $\Omega$-stability
results from Section~\ref{sec:omegastab} with $\Omega:=\Omega_D(0,1)$. However, here we provide a different, and numerically more efficient, parametrization for the set of stable matrices in the discrete-time case (Section~\ref{sec:nsm}). This idea is then generalized to compute a nearby descriptor system with a fixed rank (Section~\ref{sec:nsmp}).

\section{Nearest stable matrix} \label{sec:nsm}

Consider a discrete-time linear system described by the following difference equation
\begin{equation}\label{dessys}
x(t+1)=Ax(t),\quad t \in \mathbb N,
\end{equation}
where $A\in \R^{n,n}$ and $\mathbb N$ is the set of nonnegative integers,
$x(t)$ denotes the $n$-dimensional state vector. 
Recall from Theorem~\ref{th:stability} that, 
if $\lambda_1,\ldots,\lambda_n$ are the eigenvalues of $A$, then such a system
is called stable (resp.\ asymptotically stable) if $|\lambda_i|\leq 1$ (resp.\ $|\lambda_i|<1$) for all $i=1,\ldots,n$, and the eigenvalues
with unit modulus are semisimple; otherwise, it is called unstable.

Analogously to the continuous-time case~\eqref{eq:prob_def}, the \emph{nearest stable matrix problem} in the discrete-time case is the following optimization problem
\begin{equation}\label{eq:probdef}
\inf_{X\in \mathbb S_d^{n,n}}{\|A-X\|}_{F}^2,
\end{equation}
where $\mathbb S_d^{n,n}$ is the set of all stable matrices of size $n \times n$. This problem in discrete-time case has received much less attention, and to the best our knowledge, only~\cite{ONV13} considered this problem without any assumption on the entries of the matrix.
For the class of positive systems of the form~\eqref{dessys}, where the matrix $A$ is component-wise nonnegative, the problem of computing the nearest stable nonnegative matrix has been studied very recently in~\cite{GugP18, NesP20}.

The problem~\eqref{eq:probdef} is notoriously difficult, with the existence of many local minima, 
up to $2^n$ in dimension $n$,~\cite{GugP18}. The set $\mathbb S_d^{n,n}$ of stable matrices is highly nonconvex~\cite{ONV13}, and neither open nor closed.

\subsection{A new characterization for discrete-time stable matrices}

The principle strategy in~\cite{gillis2019approximating} for solving the problem~\eqref{eq:probdef} is to reformulate it into into an equivalent problem with a simpler feasible set onto which points can be projected relatively easily. This is achieved by deriving a factorization of stable matrices into symmetric and orthogonal matrices. To see this, let us define the SUN form of a matrix. 

\begin{definition}
A matrix $A \in \mathbb R^{n,n}$ is said to admit a \emph{SUN form} if there exist $S,U,N \in \mathbb R^{n,n}$
such that $A=S^{-1}UNS$ where $S\succ 0$, $U$ is orthogonal, $N\succeq 0$ and ${\|N\|}_2\leq 1$.
\end{definition}

\begin{theorem}\label{thm:mainresult}
A matrix is stable (resp.\  asymptotically stable) if and only if it admits a SUN form (resp.\ a SUN form with $\|N\| < 1$). 
\end{theorem}
\begin{proof}
The proof follows by the following two facts: 
\begin{enumerate} 
\item  The Lyapunov criterion of the Schur stability~\cite{Gan59a}. 

\item The polar decomposition~\cite{HorJ85}: given a square matrix $X$, a polar decomposition of $X$ is a factorization $X = WH$ where $W$ is a unitary matrix and $H$ is PSD, both square and of the same size. 
The polar decomposition of a square matrix $X$ always exists. If $X$ is invertible, the decomposition is unique, and the factor $H$ is positive definite. Given an SVD of $X = U\Sigma V$, a polar decomposition is given by $W = U V^\top$ and $H = V^\top \Sigma V$. 
\end{enumerate} 

By the Lyapunov theorem, $A$ is stable (asymptotically stable) if and only if there exists
an ellipsoid $E$ such that $AE \subseteq E$ (respectively, $AE \subseteq \text{int}E$). This is equivalent to say that there
exist matrices $C$ and $L$ such that ${\|L\|}_2 \leq 1$ (respectively, ${\|L\|}_2 < 1$) and $A = C^{-1}LC$. Now we write
the polar decomposition $C = V S$, where $V$ is orthogonal and $S \succ 0$. Thus, $A = S^{-1}V^{-1}LV S$.
Denote $V^{-1}LV = M$. Clearly, ${\|M\|}_2 = {\|L\|}_2$. Finally, write the polar decomposition: $M = UN$
with $U$ orthogonal, $N \succ 0$, and ${\|N\|}_2 = {\|M\|}_2$. We have $A = S^{-1}UNS$, which completes the proof.
\end{proof}


In view of Theorem~\ref{thm:mainresult}, the set $\mathbb S_d^{n,n}$ of stable matrices can be
characterized as the set of matrices that admit a SUN form, or equivalently, we can parametrize the set of stable matrices using a matrix triple $(S,U,N)$ as follows
\begin{eqnarray*}
\mathbb S_d^{n,n}=\Big\{S^{-1}UNS \in \R^{n,n}~\big|~S\succ 0,~U~\text{orthogonal},
~N\succeq 0~\text{with}~
{\|N\|}_2\leq 1
\Big\}.
\end{eqnarray*}
This characterization allows to reformulate the 
nearest stable matrix problem~\eqref{eq:probdef} as follows 
\begin{eqnarray}\label{eq:thm_reform}
\inf_{X\in \mathbb S_d^{n,n}}{\|A-X\|}_{F}^2 \; = \; 
\inf_{S\succ 0,~U\,\text{orthogonal},~N \succeq 0,~ {\|N\|}_2 \leq 1}{\|A-S^{-1}UNS\|}_{F}^2.
\end{eqnarray}

\subsection{Optimization algorithms}
An advantage of this reformulation is that the feasible set is rather simple and therefore it is relatively easy to project onto it. As a result, methods like BCD,  PGD (Algorithm~\ref{algo:gradient}) or projected FGM (Algorithm~\ref{algo:fgm}) can be used to tackle~\eqref{eq:thm_reform}. 

\subsubsection{Gradient}
The gradient of $f(S,U,N) = {\|A-S^{-1}UNS\|}_{F}^2$ with respect to $S$ is given by
\[
\nabla_S f(S,U,N) = 2 \, S^{-\top} [R^{\top} (R-A)  - (R-A)R^{\top}],
\]
where $R = S^{-1}UNS$.
For $U$ and $N$, we have
\[
\nabla_U f(S,U,N) = - 2 S^{-1} (A-R) S N^{\top} \quad \text{and}\quad
\nabla_N f(S,U,N) = - 2 U^{\top} S^{-1} (A-R) S ,
\]
see~\cite{gillis2019approximating} for more details. 

\subsubsection{Projection onto the feasible set} 
The projection of a solution $(S,U,N)$ onto the feasible set of~\eqref{eq:thm_reform} can be computed in closed form. 

\paragraph{Projections for $S$ and $N$}  
 In order to calculate the projection of a square matrix onto the set of positive semidefinite contractions, let us introduce some notation.
For a  symmetric matrix $H \in \R^{n,n}$ with eigenvalues $\lambda_k$ ($1 \leq k \leq n$) and eigenvalue decomposition $H=V{\rm diag}(\lambda_1,\ldots, \lambda_n)V^{\top}$, we set
$g(H)=V{\rm diag}(g(\lambda_1),\ldots, g(\lambda_n))V^{\top}$, where $g$ is any complex valued function defined on the spectrum of $H$. The matrix $g(H)$ does not depend on the particular orthogonal matrix $V$ since it is easily verified that $g(H)=q(H)$, where $q$ is any polynomial that maps each $\lambda_k$ to its value $g(\lambda_k)$.  For a general matrix $X \in \R^{n,n}$, we consider functions of its symmetric part, $g^s(X):=g((X+X^{\top})/2)$. For an interval $[a,b]\subset \R \cup \{\infty\}$ and $\lambda\in \R$ let
\[
p_{a,b}(\lambda):=\max\{a,\min\{b, \lambda\}\}=
\begin{cases}
a &\text{if }\lambda< a,\\
\lambda &\text{if }\lambda \in [a,b],\\
b&\text{if }b< \lambda.
\end{cases}
\]
Then $p_{a,b}(\lambda)$ is the nearest point projection of $\lambda$ onto $[a,b]$, that is,
$|\lambda-p_{a,b}(\lambda)|={\rm argmin}_{h \in [a,b]}|\lambda-h|$.
%
%
\begin{proposition}
The matrix $p_{a,b}^s(X)$ is the nearest point projection of $X \in \R^{n,n}$ with respect to the Frobenius norm onto the set
${\mathcal I}_{a,b}=\{\, H \in \R^{n,n}\, |\; H=H^{\top},\; a \, I\preceq H\preceq b\, I\}$, that is,
 \[
 p_{a,b}^s(X)={\rm argmin}_{H \in{\mathcal I}_{a,b}} {\|X-H\|_F}.
 \]
\end{proposition}
\begin{proof}
Let $(X+X^{\top})/2=V{\rm diag}(\lambda_1,\ldots, \lambda_n)V^{\top}$ with orthogonal $V$. Let $H\in {\mathcal I}_{a,b}$, and let
$\tilde H=V^\top  H V=[\tilde h_{ij}]$. Then $\tilde H\in {\mathcal I}_{a,b}$ and therefore $\tilde h_{ii}
\in[a,b]$ for all $i=1, \ldots , n$. By orthogonality between symmetric and skew symmetric matrices
and the orthogonal invariance of the Frobenius norm we have
\begin{eqnarray}\label{eqtemp1}
{\|X-H\|}_F^2
&=&{\left\|\frac{X-X^{\top}}{2}\right\|}_F^2+{\left\|\frac{X+X^{\top}}{2}-H\right\|}_F^2 \nonumber \\
&=& {\left\|\frac{X-X^{\top}}{2}\right\|}_F^2+{\|{\rm diag}(\lambda_1,\ldots, \lambda_n)-\tilde H\|}_F^2 \nonumber \\
  &= & {\left\|\frac{X-X^{\top}}{2}\right\|}_F^2
  + \sum_i(\lambda_i-\tilde h_{ii})^2+\sum_{i \not =j}\tilde h_{ij}^2.
\end{eqnarray}
The sum is minimized by
$\tilde H={\rm diag}( p_{a,b}(\lambda_1), \ldots , p_{a,b}(\lambda_n))$. Thus,
$H=p_{a,b}^s(X)$.
\end{proof}
Since for a positive semidefinite matrix the inequality ${\|N\|}_2 \leq \alpha$ is equivalent to $N\preceq \alpha I_n$
we have the corollaries below.
\begin{corollary}
 The nearest point projection of $X \in \R^{n,n}$ onto the set of positive semidefinite
contractions with respect to Frobenius norm  is $p_{0,1}^s(X)$, that is,
\[
p_{0,1}^s(X) \; = \; {\rm argmin}_{N \succeq 0, {\|N\|}_2 \leq 1} {{\|X-N\|}_F}.
\]
\end{corollary}
\begin{corollary}{\rm \cite{Hig88b}}
 The nearest point projection of $X \in \R^{n,n}$ onto the cone of $n\times n$ positive semidefinite
matrices with respect to Frobenius norm is $p_{0,\infty}^s(X)$, that is,
\[
p_{0,\infty}^s(X) \; = \; {\rm argmin}_{S \succeq 0} {{\|X-S\|}_F}.
\]
\end{corollary}

\paragraph{Projections for $U$} 

Before we give the projection onto the set of orthogonal matrices, we provide another closely related projection that will be useful to obtain initializations in Section~\ref{subsec:int}.
These results require the polar decomposition.

\begin{proposition}\label{thm:optimUB}
Let $X\in \R^{n,n}$ and let $X=VH$ be the polar decomposition of $X$, where $V\in \R^{n,n}$ is orthogonal and $H \in \R^{n,n}$ satisfies $ H\succeq 0$. Then
$$
{\argmin}_{(U,N), U^{\top}U=I_n, N \succeq 0, {\|N\|}_2 \leq 1}
{\| X - UN\|}_F^2 =
\left( V, p_{0,1}(H)\right),
$$
\end{proposition}
%
%
\begin{proof}
Let $H=Q\,{\rm diag}(\lambda_1, \ldots, \lambda_n)Q^{\top}$ be a diagonalization of $H$ with orthogonal $Q$.
Let $U,N\in \R^{n,n}$ be such that $U^{\top}U=I_n$ and $N\succeq 0$ with ${\|N\|}_2\leq 1$.
Then
\begin{eqnarray}\label{eq:optUB1}
{\|X-UN\|}_F^2&=&{\|VH-UN\|}_F^2={\|H-V^{\top}UN\|}_F^2 \nonumber\\
&=& {\|Q\, {\rm diag}(\lambda_1, \ldots, \lambda_n)Q^{\top} -V^{\top} UN\|}_F^2 \nonumber\\
&=&
{\|{\rm diag}(\lambda_1, \ldots, \lambda_n)-Q^{\top}V^{\top}UNQ\|}_F^2 \nonumber\\
&\geq & \sum_{i}(\lambda_i-p_{-1,1}(\lambda_i))^2\label{eq:esti5}
\\
&= & \sum_{i}(\lambda_i-p_{0,1}(\lambda_i))^2.\nonumber
\end{eqnarray}
The last equation holds since all $\lambda_i$'s are nonnegative.
The inequality (\ref{eq:esti5}) follows from the fact that all diagonal
entries of $Q^{\top}V^{\top}UNQ$ are contained in $[-1,1]$
since ${\|Q^{\top}V^{\top}UNQ\|}_2={\|N\|}_2\leq 1$.
Equality holds in (\ref{eq:esti5}) if and only if
$Q^{\top} V^{\top} UNQ={\rm diag}( p_{0,1}(\lambda_1), \ldots , p_{0,1}(\lambda_n))$.
The latter is equivalent to $UN = Vp_{0,1}(H)$.
\end{proof}
%
%
\begin{proposition}
 Denoting $\mathcal{P}_{\bot}(X)$ the projection of $X$ onto the set of $n\times n$ orthogonal matrices, we have
$
\mathcal{P}_{\bot}(X)
= \argmin_{U^{\top}U = I_n} {\|X-U\|}_F
= V$, where $X = V H$ is the polar decomposition of $X$.
\end{proposition}

\subsubsection{Initialization}~\label{subsec:int}
The algorithms BCD,  PGD or FGM used to solve the reformulation~\eqref{eq:thm_reform} are highly sensitive to  the starting points. Three initializations are proposed in~\cite{gillis2019approximating}. 
\paragraph{Standard initialization} Set $S = I_n$, for which the optimal values of $U$ and $N$ can be computed using the polar decomposition of $A$, see Proposition~\ref{thm:optimUB}:
\[
{\rm argmin}_{(U,N), U^{\top}U=I_n, N \succeq 0, {\|N\|}_2 \leq 1}
{\| A - UN\|}_F =
\left( V, p_{0,1}(H)\right),
\]
where $A = VH$ is the polar decomposition of $A$.

\paragraph{LMI-based initialization} Let
$\mu = \max(1,\rho(A))$ so that $A' = \frac{A}{\mu}$ is stable.
Then there exists a Lyapunov solution $P \succ 0$ to the system ${A'}^{\top}PA'-P \preceq 0$ (one can use the Matlab function \texttt{dlyap(A,eye(n)}). 
By setting $A'^{\top}PA'-P=-Q$ for some $Q \succeq 0$, and $A'=P^{-1/2}RP^{1/2}$ where $R=P^{1/2}A'P^{-1/2}$, we get
\begin{align*}
A'^{\top}PA'-P = -Q  
& \iff P^{1/2}R^{\top}RP^{1/2}-P  = -Q  \\
& \iff P^{1/2}\left(R^{\top}R-I_n\right)P^{1/2} = -Q \\
& \iff I_n-R^{\top}R = P^{-1/2}QP^{-1/2}  \\
& \iff \underbrace{I_n-P^{-1/2}QP^{-1/2}}_{=:H} = R^{\top}R.
\end{align*}
This implies that $H\succeq 0$ and we can write $R=UN$, where $N=H^{1/2}$ and $U$ is an orthogonal matrix.
Setting $S=P^{1/2}$ we have $A=P^{-1/2}RP^{1/2}=S^{-1}U N S$. 
Since $Q \succeq 0$ and $R^\top R \succeq 0$, we have ${\|H\|}_2 \leq 1$ which implies that ${\|B\|}_2 \leq 1$.

\paragraph{Random initialization} Generate each entry of $S$ using the normal distribution (in Matlab, \texttt{randn(n)}).
Then, replace $S$ with $SS^{\top}+I_n$ which is positive definite.
Ideally, one would like to compute the corresponding optimal $(U,N)$, that is, minimize ${\|A - S^{-1} UN S\|}_F$. However, it is not clear how to do this efficiently, and instead one can take $U$ and $N$ as the optimal solution of
\[
\min_{ U \text{ orthogonal}, N \succeq 0, {\|N\|}_2 \leq 1} \; {\|S A S^{-1} - UN\|}_F,
\]
that is, $(U,N)$ is the polar decomposition of $S A S^{-1}$ and $N$ is replaced with $p_{0,1}(H)$; see Proposition~\ref{thm:optimUB}.
The motivation is that if $S A S^{-1} \approx  UN$ then $A \approx  S^{-1} UN S$. 

A good strategy is to generate many initial random points, perform a few iterations of FGM, and keep the best solution to be refined with more FGM iterations. We refer to this approach as mRand-FGM.

\subsection{Numerical examples} \label{sec:numexpdstable} 

Let us illustrate the use of our proposed algorithm with some examples from the paper~\cite{GugP18}.

\paragraph{Example 2: 3-by-3 matrix} 

We consider
\[
A =\left( \begin{array}{ccc}
0.6 & 0.4&  0.1 \\
 0.5 & 0.5 & 0.3\\
 0.1 & 0.1 & 0.7
 \end{array} \right)
 \]
\normalsize  with $\rho(A) = 1.096$,  for which \cite{GugP18} shows that the nearest stable nonnegative matrix is
 \[
X =\left( \begin{array}{ccc}
0.5640 & 0.3599 & 0.0850 \\
 0.4716 & 0.4684 & 0.2881 \\
 0.0643 & 0.0602 & 0.6851
 \end{array} \right).
 \]
 \normalsize FGM for any initialization strategy converge to the same solution.
This is because, as shown in~\cite{GugP18} for nonnegative matrices, if a local minimum to problem~\eqref{eq:probdef}
is component-wise positive, then it is a global minimizer.

\paragraph{Example from \cite[Section 4.4]{GugP18}} 
We consider
\begin{equation} \label{matrixAGP} 
A =
\left( \begin{array}{ccccc}
 0.7 &  0.2 &  0.1 &  0.5 &  1 \\
 0.3 &  0.6 &  0.2 &  0.8 &  0.3 \\
 0.5 &  0.7 &  0.9 &  1 &  0.5 \\
 0.1 &  0.1 &  0.3 &  0.8 &  0.3\\
 0.8 &  0.2 &  0.9 &  0.3 &  0.2 \\
\end{array} \right)
\end{equation} 
\normalsize
with $\rho(A) = 2.4$. The nonnegative solution provided by Guglielmi and Protasov~\cite{GugP18} with their algorithm is
\[
X_+ =
\left( \begin{array}{ccccc}
 0.3796 &  0.1797 &  0 &  0.5 &  0.7343 \\
 0 &  0.5791 &  0.0069 &  0.8 &  0.0274 \\
 0.0580 &  0.6719 &  0.6403 &  1 &  0.1334 \\
 0 &  0 &  0 &  0.8 &  0 \\
 0.4204 &  0.1759 &  0.6770 &  0.3 &  0 \\
\end{array} \right)
\]
\normalsize with relative error 
${\|A-X_+\|}_F/\|A\|_F = 38.38\%$ (which is not necessarily optimal). (Recall that here nonnegativity is enforced which is not the case in our approach.)  

Depending on the initialization, FGM converges to different solutions: 
Stand-FGM, LMI-FGM and mRand-FGM converge to
three different solutions with relative errors 
26.31\%,  26.88\%, and 26.19\%, respectively. 
Interestingly, mRand-FGM provides the best solution, given by 
\[
S^{-1}UNS = 
\left( \begin{array}{ccccc} 
 0.5823 &  0.1491 &  -0.0759 &  0.5356 &  0.8519 \\ 
 0.2192 &  0.5868 &  0.0951 &  0.8549 &  0.1967 \\ 
 0.4583 &  0.6789 &  0.7997 &  1.0395 &  0.4186 \\ 
 -0.0493 &  -0.1106 &  -0.1899 &  0.7968 &  0.0094 \\ 
 0.7087 &  0.1482 &  0.7384 &  0.3266 &  0.0741 \\ 
\end{array} \right),  
\] 
whose eigenvalues are 
\[
\{ 
-0.5003,  
   1,  
   0.9920 \pm 0.0399i, 
   0.3561
\}. 
\] 
mRand-FGM performs better (but required additional costs) as it relies on generating several randomly generated starting point (namely, 100 in this experiment).

Figure~\ref{GP2018_dstable} shows the position of the eigenvalues of the different solutions. 
\begin{figure}[ht!]
\begin{center}
\includegraphics[width=0.9\textwidth]{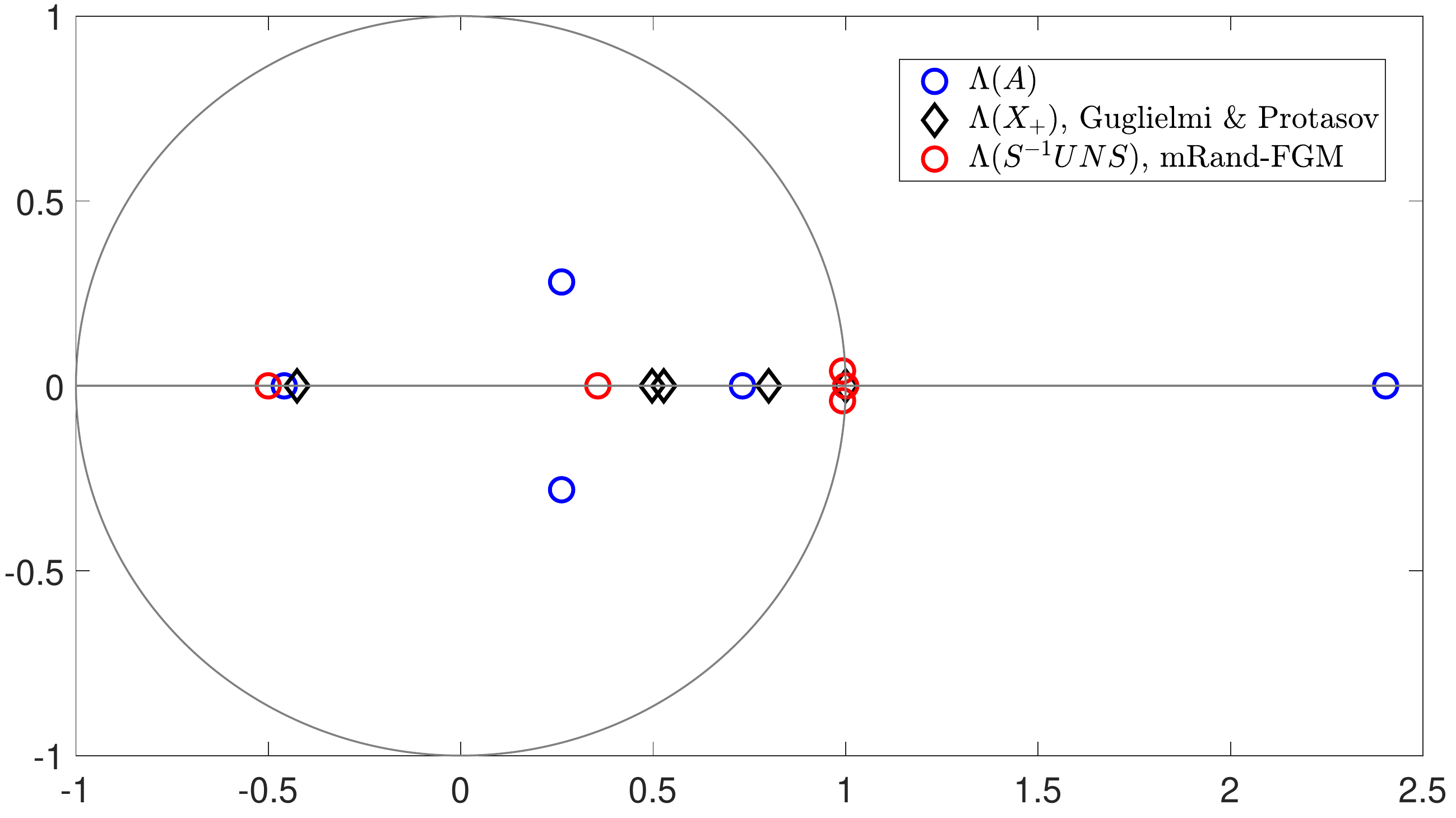}
\caption{Eigenvalues of $A$, its stable approximation $X_+$ computed by the algorithm from~\cite{GugP18}, 
and $S^{-1}UNS$ computed by  mRand-FGM.   \label{GP2018_dstable}}
\end{center}
\end{figure}

\section{Nearest stable matrix pair}  \label{sec:nsmp} 
The  matrix pair $(E,A) \in (\R^{n,n})^2$ is said to be \emph{discrete-time} \emph{stable} (resp.\ \emph{asymptotically stable})
if all the finite eigenvalues of $zE-A$ are
in the closed (resp.\ open) unit ball and those on the
unit circle are semisimple.
The  matrix pair $(E,A)$ is said to be discrete-time \emph{admissible} if it is regular, of index at most one, and discrete-time stable.

In this section, we discuss the discrete-time counter part of the continuous-time nearest stable matrix pair problem; see Problem~\ref{prob:contstabpair}. 

\begin{problem}\label{prob:discrstabpair}{\rm
For a given  pair $(E,A) \in (\R^{n,n})^2$
find the nearest discrete-time admissible matrix pair $(M,X)$. In other words,  if $\mathbb S_d$ is the set of
matrix pairs $(M,X)\in (\R^{n,n})^2$ that are regular, of index at most one, and have all finite eigenvalues inside closed unit ball,
then compute
\begin{equation}\label{mainprobdiscpair}
\inf_{(M,X) \in \mathbb S_d} \{{\|E-M\|}_F^2+{\|A-X\|}_F^2\}.\tag{$\mathcal{P}$}
\end{equation}
}
\end{problem}

To the best of our knowledge, this problem was discussed for the first time in~\cite{gillis2020note}, where the problem is found to be very difficult due to 
\begin{itemize}
\item the properties of the spectral radius as a function of the input matrix;
\item the set $\mathbb S_d$ is nonconvex, and is neither open nor closed;
\item not being able to reformulate the set $\mathbb S_d$, unlike the continuous-time case, Problem~\ref{prob:contstabpair}. 
\end{itemize}

In~\cite{gillis2020note}, authors considered instead a \emph{rank-constrained nearest stable matrix pair problem}. For this, let $r (\leq n) \in \mathbb Z_{+}$ and define a subset ${\mathbb S}^r_d$ of ${\mathbb S}_d$ by
\[
{\mathbb S}^r_d :=\left\{(M,X)\in {\mathbb S}_d:~\text{rank}(M)=r\right\}.
\]
For a given unstable matrix pair $(E,A)$, the \emph{rank-constrained nearest stable matrix pair problem} requires to compute the smallest perturbation $(\Delta_E,\Delta_A)$
with respect to Frobenius norm such that $(E+\Delta_E,A+\Delta_A)$ is admissible with $\text{rank}(E+\Delta_E)=r$, or equivalently, solve the following
optimization problem
\begin{equation}\label{restprob}
\inf_{(M, X)\in{\mathbb S}^r_d} {\|E-M\|}_F^2+{\|A-X\|}_F^2 \tag{$\mathcal{P}_r$}.
\end{equation}

The advantages of the rank-constraint admissible pair $(E,A)$ are the following: 
\begin{itemize}
\item It allows to parameterize the set ${\mathbb S}^r_d$ in terms of the matrix quadruple $(T,W,U,N)$, where $T,W\in \R^{n,n}$ are invertible, $U\in \R^{r,r}$ is orthogonal, and $N\in \R^{r,r}$ is a positive semidefinite contraction, see Section~\ref{discpair_reform}.
\item  The set ${\mathbb S}_d$ of admissible pairs can be written as
\begin{eqnarray*}
{\mathbb S}_d
&=&\bigcup_{r=1}^n {\mathbb S}_d^r.
\end{eqnarray*}
This implies that
\[
\eqref{mainprobdiscpair} = \min_{r=1,2,\ldots,n}\eqref{restprob}.
\]
Therefore, to compute a solution of~\eqref{mainprobdiscpair},
a possible way is therefore to solve $n$ rank-constrained problems~\eqref{restprob}.
\item Such situations may be useful when  descriptor systems are directly generated from data
where the constraints are added as a second step.  One practical example is of a circuit or power net, where one discretizes the flow and adds the Kirchhoff laws afterward.
\end{itemize}

\subsection{Reformulation of Problem~\eqref{restprob}}\label{discpair_reform}
The idea of parametrizing the set of discrete-time stable matrices (Theorem~\ref{thm:mainresult}) can be generalized to the set of rank- constrained admissible pairs, ${\mathbb S}_d^r$. It was noted in~\cite{gillis2020note} that in the proof of Theorem~\ref{thm:mainresult}, only the invertibility of matrix
$S$ is needed and the condition of symmetry on $S$ can be relaxed. 
The corresponding characterization of stable matrices is as follows. 
\begin{theorem}\label{thm:newstabmatchar}
Let $A\in \R^{n,n}$. Then $A$ is stable if and only if 
 $A=S^{-1}UNS$ for some $S,U,N \in \R^{n,n}$ such that $S$ is invertible, $U^{\top}U=I_n$, $N\succeq 0$, and ${\|N\|}_2\leq 1$.
\end{theorem}
\begin{theorem}\label{thm:reform1}
Let $E,A\in \R^{n,n}$ be such that $\text{rank}(E)=r$. Then
$(E,A)$ is admissible if and only if
there exist matrices $T,W \in \R^{n,n}$, $S,U,N\in \R^{r,r}$ such that the matrices $T,W,S$ are invertible,
$U^{\top}U=I_r$, $N\succeq 0$, ${\|N\|}_2\leq 1$ such that
\begin{equation}\label{eq:firstreform}
E=W \begin{bmatrix}   I_r  & 0 \\   0 & 0 \end{bmatrix}T, \quad
\text{and}
\quad
A=W\begin{bmatrix}  S^{-1}UNS  & 0 \\  0 & I_{n-r}\end{bmatrix}T.
\end{equation}
\end{theorem}
\begin{proof} For a regular index one pair $(E,A)$,
there exist invertible matrices $W,T\in \R^{n,n}$ such that
\begin{equation}
E=W \begin{bmatrix}   I_r  & 0 \\   0 & 0 \end{bmatrix}T \qquad \text{ and } \qquad
A=W\begin{bmatrix}  \tilde A  & 0 \\  0 & I_{n-r}\end{bmatrix}T,
\end{equation}
see~\cite{Gan59a}.
Further, the finite eigenvalues of $(E,A)$ and $\tilde A$ are the same because $\text{det}(\lambda E-A)=0$
if and only if
 $\text{det}(\lambda I_r-\tilde A)=0$.
Thus, by stability of $\tilde A$ and Theorem~\ref{thm:newstabmatchar}, it follows that 
there exist $S,U,N \in \R^{r,r}$ such that
$S$ is invertible, $U^{\top}U=I_r$, $N\succeq 0$, ${\|N\|}_2\leq 1$, and $\tilde A =S^{-1}UNS$.\\
Conversely, it is easy to see that any matrix pair $(E,A)$ in the form~\eqref{eq:firstreform} is regular and of index one.
The stability of $(E,A)$ follows from Theorem~\ref{thm:newstabmatchar} as the matrix $S^{-1}UNS$ is stable.
\end{proof}

For a standard pair $(I_n,A)$ (with $E=I_n$), Theorem~\ref{thm:reform1} coincides with
Theorem~\ref{thm:newstabmatchar} as in this case $W$ and $T$ can be chosen to be the identity matrix which yields $A=S^{-1}UNS$.
Note that the matrix $S$ is invertible in Theorem~\ref{thm:reform1} and therefore it can be absorbed in $W$ and $T$. The advantage is that this reduces the number of variables in the corresponding optimization problem.
\begin{corollary}\label{thm:reform2}
Let $E,A\in \R^{n,n}$ be such that $\text{rank}(E)=r$. Then
$(E,A)$ is admissible if and only if
there exist invertible matrices $T,W \in \R^{n,n}$, and $U,N\in \R^{r,r}$ with
$U^{\top}U=I_r$, $N\succeq 0$ and ${\|N\|}_2\leq 1$ such that
\begin{equation}\label{eq:secreform}
E=W \begin{bmatrix}   I_r  & 0 \\   0 & 0 \end{bmatrix}T, \quad
\text{and}
\quad
A=W\begin{bmatrix}  UN & 0 \\  0 & I_{n-r}\end{bmatrix}T.
\end{equation}
\end{corollary}

In view of Corollary~\ref{thm:reform2}, the set ${\mathbb S}^r_d$ of restricted rank admissible pairs
can be characterized in terms of matrix pairs~\eqref{eq:secreform}, that is,
\begin{eqnarray*}
&{\mathbb S}^r_d = \Bigg \{\left(W \mat{cc}  I_r  & 0 \\   0 & 0 \rix T,
W\mat{cc}  UN  & 0 \\  0 & I_{n-r}\rix T\right):~\text{invertible}~ T,W \in \R^{n,n},\\
& \hspace{6cm} U,N\in\R^{r,r},U^{\top}U=I_r,N\succeq 0,{\|N\|}_2\leq 1 \Bigg \}.
\end{eqnarray*}
This parametrization allows us to reformulate problem~\eqref{restprob} as 
\begin{equation}\label{eq:reform_prob}
 \inf_{W,T \in \R^{n,n},\,U, N\in \R^{r,r},\,U^{\top}U=I_r,\, {\|N\|}_2\leq 1} \;  f(W,T,U,N),
\end{equation}
where
\[
f(W,T,U,N)=
{\left\|E-W \mat{cc}  I_r  & 0 \\   0 & 0 \rix T\right\|}_F^2+
{\left\|A-W\mat{cc}  UN  & 0 \\  0 & I_{n-r}\rix T\right\|}_F^2 .
\]
An advantage of this reformulation over~\eqref{restprob} is that it is relatively easy to project onto the feasible set of~\eqref{eq:reform_prob}.
Thus methods like  BCD,  PGD (Algorithm~\ref{algo:gradient}) or FGM (Algorithm~\ref{algo:fgm}) can be used to estimate~\eqref{eq:reform_prob}, see~\cite{gillis2020note} for more details on optimization methods. 

For the initialization, we restrict ourselves to the identity initialization: We take $W=T=I_n$ and $(U,N)$ as the optimal solution of
\[
\min_{(U,N) \text{ s.t. } U^\top U=I_r, \|N\| \leq 1} {\|A_{1:r,1:r}-UN\|}_F^2.
\]
In this particular case, it can be computed explicitly using the polar decomposition of $A_{1:r,1:r}$~\cite{gillis2020note}.

\subsection{Numerical example}
 
 Let us consider the 5-by-5 matrix $A$ from~\eqref{matrixAGP}. 
As we had seen in Section~\ref{sec:numexpdstable}, 
The nearest stable matrix to $A$ had a relative error of 26.19\%, with $\| A - S^{-1} UN S\|_F^2 = 0.5672$. 

Now, allowing us to approximate the standard system $(I_n,A)$ with a descriptor system with \mbox{$r=n$}, we obtain a nearby descriptor system $(M,X)$ with 
$\| I_n - M\|_F^2 + \| A - X \|_F^2 = 0.1869$, which is significantly smaller than when imposing $M=I_n$. 
We have  
\[ M = \left( \begin{array}{ccccc} 
 1.0263 &  0.0053 &  0.0442 &  -0.0281 &  0.0414 \\ 
 0.0411 &  0.9884 &  0.0072 &  -0.0262 &  0.0322 \\ 
 0.0145 &  -0.0097 &  0.9852 &  -0.0043 &  0.0022 \\ 
 0.0859 &  0.0430 &  0.2248 &  0.8851 &  0.1779 \\ 
 0.0187 &  0.0046 &  0.0342 &  -0.0208 &  1.0309 \\ 
\end{array} \right), 
\] 
and 
\[ 
X = \left( \begin{array}{ccccc} 
 0.6874 &  0.1915 &  0.0602 &  0.5186 &  0.9703 \\ 
 0.2588 &  0.6115 &  0.1923 &  0.8264 &  0.2675 \\ 
 0.4792 &  0.7111 &  0.9125 &  1.0087 &  0.4922 \\ 
 0.0886 &  0.0395 &  0.0997 &  0.8635 &  0.1864 \\ 
 0.7920 &  0.1928 &  0.8693 &  0.3134 &  0.1783 \\ 
\end{array} \right), 
\]  
whose generalized eigenvalues are 
\[
\{ -0.4761, 
   0.9937 \pm 0.0518i, 
   0.9126, 
   0.4192  \}. 
\]

\section{Applications: data driven system identification} \label{sec:appldiscretesys}

	In the paper \cite{mamakoukas2020learning}, Mamakoukas et al.\ consider the problem of learning a discrete LTI system, of the form, 
	\[
	y_t := x_{t+1} = A x_t + B u_t , \text{ for } t=0,1,2,\dots 
	 \] 
	from observations. More precisely, given a set of observations, $(x_t,y_t)$ for $t=0,1,2,\dots,T$, the goal is to recover $A$ and $B$.  
Defining the matrices $X = [x_0, x_1, \dots, x_T]$, and similarly for $Y$ and $U$, the least squares solution to that problem is the optimal solution of 
	\[
	\min_{A, B} \| Y - AX - BU \|_F^2. 
	\]
	This is an unconstrained least squares problem with a closed-form solution, namely $[A_{\text{ls}}, B_{\text{ls}}] = Y [X;U]^\dagger$, 
	where $\dagger$ denotes the Moore-Penrose inverse. 
However, this solution does not take into account prior information, such as stability of the sought system. 
	In~\cite{mamakoukas2020learning}, authors used the reformulation from Section~\ref{sec:nsm} to reformulate the above problem by taking stability explicitly into account: Substituting 
	$A = S^{-1}U N S$, we obtain the following problem that naturally takes stability into account via a convex feasible set: 
	\[
	\min_{S \succ 0,U \text{orthogonal}, N \succeq 0, \|N\| \leq 1} 
	\| Y - S^{-1}U N S X - BU \|_F^2. 
	\]
	Using an FGM lead them to a new algorithm for data-driven system identification that outperform the state of the art, achieving order-of-magnitude improvement, both in terms of reconstruction error and computational load (time and memory requirements); see the numerical experiments in~\cite{mamakoukas2020learning}. A video presenting the paper is available from \url{https://slideslive.com/38936948}, while the code, and videos of illustrative examples are available from \url{https://github.com/giorgosmamakoukas/MemoryEfficientStableLDS}. 
 
	In a follow-up work, Mamakoukas et al.\  \cite{mamakoukas2020learning2} learned  data-driven stable Koopman operators (which are infinite-dimensional linear representations of general nonlinear systems) relying on the same characterization; 
	see \url{https://sites.google.com/view/learning-stable-koopman} for more details.



\newpage

\small 

\bibliographystyle{spmpsci}
\bibliography{Bibliographymatrixnearness}

\begin{thebibliography}{100}
\providecommand{\url}[1]{{#1}}
\providecommand{\urlprefix}{URL }
\expandafter\ifx\csname urlstyle\endcsname\relax
  \providecommand{\doi}[1]{DOI~\discretionary{}{}{}#1}\else
  \providecommand{\doi}{DOI~\discretionary{}{}{}\begingroup
  \urlstyle{rm}\Url}\fi

\bibitem{AlaBKMM11}
Alam, R., Bora, S., Karow, M., Mehrmann, V., Moro, J.: Perturbation theory for
  {H}amiltonian matrices and the distance to bounded-realness.
\newblock SIAM Journal on Matrix Analysis and Applications \textbf{32}(2),
  484--514 (2011)

\bibitem{AndV73}
Anderson, B., Vongpanitlerd, S.: Network Analysis and Synthesis.
\newblock Prentice-Hall, Englewood Cliffs, New Jersey (1973)

\bibitem{Ant05}
Antoulas, A.C.: Approximation of large-scale dynamical systems.
\newblock SIAM (2005)

\bibitem{Antm97}
Antsaklis, P.J., Michel, A.N.: Linear systems.
\newblock McGraw-Hill, New York (1997)

\bibitem{AstOV10}
Astolfi, A., Ortega, R., Venkatraman, A.: A globally exponentially convergent
  immersion and invariance speed observer for mechanical systems with
  non-holonomic constraints.
\newblock Automatica \textbf{46}(1), 182--189 (2010)

\bibitem{BeaMX15_ppt}
Beattie, C., Mehrmann, V., Xu, H.: Port-{H}amiltonian realizations of linear
  time invariant systems.
\newblock Preprint 23-2015, Institut f\"ur Mathematik, TU Berlin (2015)

\bibitem{BeaMXZ17_ppt}
Beattie, C., Mehrmann, V., Xu, H., Zwart, H.: Linear port-{H}amiltonian
  descriptor systems.
\newblock Math. Control Signals Syst. \textbf{30}(17) (2018).
\newblock \doi{10.1007/s00498-018-0223-3}

\bibitem{Bertsekas99b}
Bertsekas, D.: Corrections for the book nonlinear programming: Second edition
  (1999).
\newblock Available at \url{http://www.athenasc.com/nlperrata.pdf}

\bibitem{Bertsekas99}
Bertsekas, D.: Nonlinear Programming: Second Edition.
\newblock Athena Scientific, Massachusetts (1999)

\bibitem{blondel2000survey}
Blondel, V.D., Tsitsiklis, J.N.: A survey of computational complexity results
  in systems and control.
\newblock Automatica \textbf{36}(9), 1249--1274 (2000)

\bibitem{bolte2014proximal}
Bolte, J., Sabach, S., Teboulle, M.: Proximal alternating linearized
  minimization for nonconvex and nonsmooth problems.
\newblock Mathematical Programming \textbf{146}(1-2), 459--494 (2014)

\bibitem{boumal2014manopt}
Boumal, N., Mishra, B., Absil, P.A., Sepulchre, R.: Manopt, a matlab toolbox
  for optimization on manifolds.
\newblock The Journal of Machine Learning Research \textbf{15}(1), 1455--1459
  (2014)

\bibitem{boumal2016non}
Boumal, N., Voroninski, V., Bandeira, A.: The non-convex burer-monteiro
  approach works on smooth semidefinite programs.
\newblock Advances in Neural Information Processing Systems \textbf{29},
  2757--2765 (2016)

\bibitem{BoyGFB94}
Boyd, S., El~Ghaoui, L., Feron, E., Balakrishnan, V.: Linear Matrix
  Inequalities in System and Control Theory.
\newblock Society for Industrial and Applied Mathematics (1994).
\newblock \doi{10.1137/1.9781611970777}.
\newblock \urlprefix\url{http://epubs.siam.org/doi/abs/10.1137/1.9781611970777}

\bibitem{brock1968optimal}
Brock, J.E.: Optimal matrices describing linear systems.
\newblock AIAA Journal \textbf{6}(7), 1292--1296 (1968)

\bibitem{Broc15}
Brockett, R.W.: Finite Dimensional Linear Systems.
\newblock Society for Industrial and Applied Mathematics, Philadelphia, PA
  (2015).
\newblock \doi{10.1137/1.9781611973884}

\bibitem{BruS13}
Brull, T., Schr\"oder, C.: Dissipativity enforcement via perturbation of
  para-{H}ermitian pencils.
\newblock IEEE Transactions on Circuits and Systems I: Regular Papers
  \textbf{60}(1), 164--177 (2013)

\bibitem{burer2003nonlinear}
Burer, S., Monteiro, R.D.: A nonlinear programming algorithm for solving
  semidefinite programs via low-rank factorization.
\newblock Mathematical Programming \textbf{95}(2), 329--357 (2003)

\bibitem{Bye88}
Byers, R.: A bisection method for measuring the distance of a stable to
  unstable matrices.
\newblock SIAM J. on Scientific and Statistical Computing \textbf{9}, 875--881
  (1988)

\bibitem{ByeHM98}
Byers, R., He, C., Mehrmann, V.: Where is the nearest non-regular pencil?
\newblock Linear Algebra Appl. \textbf{285}(1-3), 81--105 (1998).
\newblock \doi{10.1016/S0024-3795(98)10122-2}

\bibitem{ByeN93}
Byers, R., Nichols, N.: On the stability radius of a generalized state-space
  system.
\newblock Linear Algebra and its Applications \textbf{188}, 113--134 (1993)

\bibitem{Cam80}
Campbell, S.L.: Singular Systems of Differential Equations.
\newblock Pitman, London (1980)

\bibitem{CheMH19}
Cherifi, K., Mehrmann, V., Hariche, K.: Numerical methods to compute a minimal
  realization of a port-{Hamiltonian} system.
\newblock arXiv preprint arXiv:1903.07042  (2019)

\bibitem{ChilG96}
Chilali, M., Gahinet, P.: {${\rm H_{\infty}}$} design with pole placement
  constraints: an {LMI} approach.
\newblock IEEE Transactions on Automatic Control \textbf{41}(3), 358--367
  (1996)

\bibitem{choudhary2020approximating}
Choudhary, N., Gillis, N., Sharma, P.: On approximating the nearest
  $\omega$-stable matrix.
\newblock Numerical Linear Algebra with Applications \textbf{27}(3), e2282
  (2020)

\bibitem{CoePS99}
Coelho, C., Phillips, J., Silveira, L.: Robust rational function approximation
  algorithm for model generation.
\newblock In: Proceedings 1999 Design Automation Conference (Cat. No.
  99CH36361), pp. 207--212 (1999)

\bibitem{conn2000trust}
Conn, A.R., Gould, N.I., Toint, P.L.: Trust region methods.
\newblock SIAM (2000)

\bibitem{cvx}
CVX~Research, I.: {CVX}: Matlab software for disciplined convex programming,
  version 2.0.
\newblock \url{http://cvxr.com/cvx} (2012)

\bibitem{DesV75}
Desoer, C., Vidyasagar, M.: Feedback Systems: Input-Output Properties.
\newblock Academic Press, Orlando, FL, USA (1975).
\newblock \doi{10.1137/1.9780898719055}.
\newblock \urlprefix\url{http://epubs.siam.org/doi/abs/10.1137/1.9780898719055}

\bibitem{DuLM13}
Du, N., Linh, V., Mehrmann, V.: Robust stability of differential-algebraic
  equations.
\newblock In: Surveys in Differential-Algebraic Equations {I}, pp. 63--95.
  Berlin: Springer (2013).
\newblock \doi{10.1007/978-3-642-34928-7_2}

\bibitem{Dua10}
Duan, G.R.: Analysis and Design of Descriptor Linear Systems.
\newblock Springer-Verlag, New York (2010)

\bibitem{Eisr84}
Eising, R.: The distance between a system and the set of uncontrollable
  systems.
\newblock In: P.A. Fuhrmann (ed.) Mathematical Theory of Networks and Systems,
  pp. 303--314. Springer Berlin Heidelberg, Berlin, Heidelberg (1984)

\bibitem{EmmV13}
Emmrich, E., Mehrmann, V.: Operator differential-algebraic equations arising in
  fluid dynamics.
\newblock Comput. Methods Appl. Math. \textbf{13}, 443--470 (2013).
\newblock \doi{https://doi.org/10.1515/cmam-2013-0018}

\bibitem{FreJ11}
Freund, R.: The {SPRIM} algorithm for structure-preserving order reduction of
  general {RLC} circuits.
\newblock in Model Reduction for Circuit Simulation, Springer, New York (2011)

\bibitem{FreJ04}
Freund, R., Jarre, F.: An extension of the positive real lemma to descriptor
  systems.
\newblock Optimization Methods and Software \textbf{19}(1), 69--87 (2004)

\bibitem{FreJ07}
Freund, R., Jarre, F., Vogelbusch, C.H.: Nonlinear semidefinite programming:
  sensitivity, convergence, and an application in passive reduced-order
  modeling.
\newblock Mathematical Programming \textbf{109}(2-3), 581--611 (2007)

\bibitem{FujSS12}
Fujimoto, K., Sakai, S., Sugie, T.: Passivity based control of a class of
  {Hamiltonian} systems with nonholonomic constraints.
\newblock Automatica \textbf{48}(12), 3054--3063 (2012)

\bibitem{gangsaas1986application}
Gangsaas, D., Bruce, K., Blight, J., Ly, U.L.: Application of modem synthesis
  to aircraft control: Three case studies.
\newblock IEEE Transactions on Automatic Control \textbf{31}(11), 995--1014
  (1986)

\bibitem{Gan59a}
Gantmacher, F.: The Theory of Matrices I.
\newblock Chelsea Publishing Company, New York, NY (1959)

\bibitem{gillis2020book}
Gillis, N.: Nonnegative Matrix Factorization.
\newblock SIAM, Philadelphia (2020)

\bibitem{gillis2019approximating}
Gillis, N., Karow, M., Sharma, P.: Approximating the nearest stable
  discrete-time system.
\newblock Linear Algebra and its Applications \textbf{573}, 37--53 (2019)

\bibitem{gillis2020note}
Gillis, N., Karow, M., Sharma, P.: A note on approximating the nearest stable
  discrete-time descriptor systems with fixed rank.
\newblock Applied Numerical Mathematics \textbf{148}, 131--139 (2020)

\bibitem{gillis2018computing}
Gillis, N., Mehrmann, V., Sharma, P.: Computing the nearest stable matrix
  pairs.
\newblock Numerical Linear Algebra with Applications \textbf{25}(5), e2153
  (2018)

\bibitem{GilS16}
Gillis, N., Sharma, P.: On computing the distance to stability for matrices
  using linear dissipative {H}amiltonian systems.
\newblock Automatica \textbf{85}, 113--121 (2017)

\bibitem{gillis2018finding}
Gillis, N., Sharma, P.: Finding the nearest positive-real system.
\newblock SIAM Journal on Numerical Analysis \textbf{56}(2), 1022--1047 (2018)

\bibitem{GilS18}
Gillis, N., Sharma, P.: A semi-analytical approach for the positive
  semidefinite procrustes problem.
\newblock Linear Algebra and its Applications \textbf{540}, 112--137 (2018)

\bibitem{gillis2020minimal}
Gillis, N., Sharma, P.: Minimal-norm static feedbacks using dissipative
  {Hamiltonian} matrices.
\newblock Linear Algebra and its Applications \textbf{623}, 258--281 (2021).
\newblock Special issue in honor of {Paul Van Dooren}

\bibitem{GohLR82}
Gohberg, I., Lancaster, P., Rodman, L.: Matrix Polynomials.
\newblock Classics in Applied Mathematics. Society for Industrial and Applied
  Mathematics, Philadelphia (1982).
\newblock \urlprefix\url{https://books.google.be/books?id=KwEItnMvwbgC}

\bibitem{GolSBM03}
Golo, G., {van der}~Schaft, A., Breedveld, P., Maschke, B.: {H}amiltonian
  formulation of bond graphs.
\newblock In: A.R. R.~Johansson (ed.) Nonlinear and Hybrid Systems in
  Automotive Control, pp. 351--372. Springer-Verlag, Heidelberg, Germany (2003)

\bibitem{Tal04}
Grivet-Talocia, S.: Passivity enforcement via perturbation of {H}amiltonian
  matrices.
\newblock IEEE Transactions on Circuits and Systems I: Regular Papers
  \textbf{51}(9), 1755--1769 (2004)

\bibitem{GugPBS12}
Gugercin, S., Polyuga, R., Beattie, C., Van~der Schaft, A.:
  Structure-preserving tangential interpolation for model reduction of
  port-{Hamiltonian} systems.
\newblock Automatica \textbf{48}(9), 1963--1974 (2012)

\bibitem{GugKL15}
Guglielmi, N., Kressner, D., Lubich, C.: Low rank differential equations for
  {H}amiltonian matrix nearness problems.
\newblock Numerische Mathematik \textbf{129}(2), 279--319 (2015)

\bibitem{guglielmi2017matrix}
Guglielmi, N., Lubich, C.: Matrix stabilization using differential equations.
\newblock SIAM Journal on Numerical Analysis \textbf{55}(6), 3097--3119 (2017)

\bibitem{GugLM16}
Guglielmi, N., Lubich, C., Mehrmann, V.: On the nearest singular matrix pencil.
\newblock SIAM Journal on Matrix Analysis and Applications \textbf{38}(3),
  776--806 (2017)

\bibitem{GugP18}
Guglielmi, N., Protasov, V.Y.: On the closest stable/unstable nonnegative
  matrix and related stability radii.
\newblock SIAM J. Matrix Anal. Appl. \textbf{39}(4), 1642--1669 (2018)

\bibitem{GusS01}
Gustavsen, B., Semlyen, A.: Enforcing passivity for admittance matrices
  approximated by rational functions.
\newblock IEEE Transactions on Power Systems \textbf{16}(1), 97--104 (2001)

\bibitem{HadB91}
Haddad, M., Bernstein, D.: Explicit construction of quadratic lyapunov
  functions for the small gain, positivity, circle and popov theorems and their
  application to robust stability.
\newblock In: Proc. of the 30th IEEE Conf. on Decision and Control, pp.
  2618--2623 vol.3 (1991)

\bibitem{hien2019extrapolNMF}
Hien, L.T.K., Gillis, N., Patrinos, P.: Inertial block proximal methods for
  non-convex non-smooth optimization.
\newblock In: Proceedings of the 37th International Conference on Machine
  learning (ICML) (2020)

\bibitem{hien2020inertial}
Hien, L.T.K., Phan, D.N., Gillis, N.: An inertial block majorization
  minimization framework for nonsmooth nonconvex optimization.
\newblock arXiv preprint arXiv:2010.12133  (2020)

\bibitem{Hig88b}
Higham, N.: Computing a nearest symmetric positive semidefinite matrix.
\newblock Linear Algebra and its Applications \textbf{103}, 103--118 (1988)

\bibitem{Hig88a}
Higham, N.: Matrix nearness problems and applications.
\newblock In: M.~Gover, e.~S.~Barnett (eds.) Applications of Matrix Theory, pp.
  1--27. Oxford University Press (1989)

\bibitem{HinP86}
Hinrichsen, D., Pritchard, A.: Stability radii of linear systems.
\newblock Systems Control Lett. \textbf{7}, 1--10 (1986)

\bibitem{HorJ85}
Horn, R., Johnson, C.: Matrix Analysis.
\newblock Cambridge University Press, Cambridge (1985)

\bibitem{HuaIMS99}
Huang, C.H., Ioannou, P., Maroulas, J., Safonov, M.: Design of strictly
  positive real systems using constant output feedback.
\newblock IEEE Trans. on Automatic Control \textbf{44}(3), 569--573 (1999)

\bibitem{IdaB97}
Ida, N., Bastos, P.A.: Electromagnetics and Calculation of Fields.
\newblock Springer-Verlag, New York (1997)

\bibitem{IonT87}
Ioannou, P., Tao, G.: Frequency domain conditions for strictly positive real
  functions.
\newblock IEEE Trans. on Automatic Control \textbf{32}(1), 53--54 (1987)

\bibitem{JacZ12}
Jacob, B., Zwart, H.: Linear Port-{Hamiltonian} Systems on Infinite-dimensional
  Spaces.
\newblock Springer, Berlin (2012)

\bibitem{Jos89}
Joshi, S.: Control of Large Flexible Space Structures (Lecture Notes in Control
  and Information Sciences), vol. 131.
\newblock Springer-Verlag Berlin Heidelberg (1989)

\bibitem{Kai80}
Kailath, T.: Linear Systems.
\newblock Prentice-Hall, Englewood Cliffs, NJ (1980)

\bibitem{KauNC89}
Kautsky, J., Nichols, N., Chu, E.W.: Robust pole assignment in singular control
  systems.
\newblock Linear Algebra and its Applications \textbf{121}, 9--37 (1989)

\bibitem{Kha92}
Khalil, H.K.: Nonlinear Systems.
\newblock Macmillan, New York (1992)

\bibitem{Kle13}
Kleijn, C.: 20-sim 4c 2.1 reference manual.
\newblock Controlab Products B.V  (2013)

\bibitem{KunM06}
Kunkel, P., Mehrmann, V.: Differential-Algebraic Equations: Analysis and
  Numerical Solution.
\newblock EMS textbooks in mathematics. European Mathematical Society (2006)

\bibitem{LanT85}
Lancaster, P., Tismenetsky, M.: The Theory of Matrices, 2nd edn.
\newblock Academic Press, Orlando (1985)

\bibitem{leibfritz2004compleib}
Leibfritz, F.: Compleib, constraint matrix-optimization problem library-a
  collection of test examples for nonlinear semidefinite programs, control
  system design and related problems.
\newblock Dept. Math., Univ. Trier, Trier, Germany, Tech. Rep  (2004)

\bibitem{LozBEM13}
Lozano, R., Brogliato, B., Egeland, O., Maschke, B.: Dissipative systems
  analysis and control: theory and applications.
\newblock Springer Science \& Business Media (2013)

\bibitem{LozJ90}
Lozano-Leal, R., Joshi, S.: Strictly positive real transfer functions
  revisited.
\newblock IEEE Trans. on Automatic Control \textbf{35}(11), 1243--1245 (1990)

\bibitem{mamakoukas2020learning2}
Mamakoukas, G., Abraham, I., Murphey, T.D.: Learning data-driven stable
  {Koopman} operators.
\newblock arXiv preprint arXiv:2005.04291  (2020)

\bibitem{mamakoukas2020learning}
Mamakoukas, G., Xherija, O., Murphey, T.D.: Learning memory-efficient stable
  linear dynamical systems for prediction and control.
\newblock In: 34th Conference on Neural Information Processing Systems
  (NeurIPS). Vancouver, Canada (2020)

\bibitem{Markovsky2019}
Markovsky, I.: Low-Rank Approximation: Algorithms, Implementation,
  Applications, 2nd edition edn.
\newblock Springer (2019).
\newblock
  \urlprefix\url{http://homepages.vub.ac.be/~imarkovs/publications.html}

\bibitem{MasSB92}
Maschke, B., {Van Der Schaft}, A., Breedveld, P.: An intrinsic {Hamiltonian}
  formulation of network dynamics: non-standard poisson structures and
  gyrators.
\newblock Journal of the Franklin Institute \textbf{329}(5), 923--966 (1992)

\bibitem{MasKAS97}
Masubuchi, I., Kamitane, Y., Ohara, A., Suda, N.: ${H}_{\infty}$ control for
  descriptor systems: A matrix inequalities approach.
\newblock Automatica \textbf{33}(4), 669--673 (1997)

\bibitem{MehMS16}
Mehl, C., Mehrmann, V., Sharma, P.: Stability radii for linear {H}amiltonian
  systems with dissipation under structure-preserving perturbations.
\newblock SIAM Journal on Matrix Analysis and Applications \textbf{37}(4),
  1625--1654 (2016)

\bibitem{MehMS16b}
Mehl, C., Mehrmann, V., Sharma, P.: Stability radii for real linear
  {Hamiltonian} systems with perturbed dissipation.
\newblock BIT Numerical Mathematics \textbf{57}(3), 811--843 (2017)

\bibitem{MehMW15}
Mehl, C., Mehrmann, V., Wojtylak, M.: On the distance to singularity via low
  rank perturbations.
\newblock Operators and Matrices \textbf{9}, 733--772 (2015)

\bibitem{MehMW18}
Mehl, C., Mehrmann, V., Wojtylak, M.: Linear algebra properties of dissipative
  {Hamiltonian} descriptor systems.
\newblock SIAM Journal on Matrix Analysis and Applications \textbf{39}(3),
  1489--1519 (2018)

\bibitem{MehMW21}
Mehl, C., Mehrmann, V., Wojtylak, M.: Distance problems for dissipative
  {Hamiltonian} systems and related matrix polynomials.
\newblock Linear Algebra and its Applications \textbf{623}, 335--366 (2021).
\newblock \doi{https://doi.org/10.1016/j.laa.2020.05.026}.
\newblock Special issue in honor of Paul Van Dooren

\bibitem{Meh91}
Mehrmann, V.: The Autonomous Linear Quadratic Control Problem: Theory and
  Numerical Solution.
\newblock Lecture Notes in Control and Information Sciences. Springer Berlin
  Heidelberg (1991).
\newblock \urlprefix\url{https://books.google.be/books?id=VAKrAAAAIAAJ}

\bibitem{MehS11}
Mehrmann, V., Schr${\rm \ddot{o}}$der, C.: Nonlinear eigenvalue and frequency
  response problems in industrial practice.
\newblock J. Math. Industry \textbf{1}(18) (2011)

\bibitem{MehV20}
Mehrmann, V., Van~Dooren, P.: Optimal robustness of port-{Hamiltonian} systems.
\newblock SIAM Journal on Matrix Analysis and Applications \textbf{41}(1),
  134--151 (2020)

\bibitem{MosL91}
Moses, R., Liu, D.: Determining the closest stable polynomial to an unstable
  one.
\newblock IEEE Trans. on Signal Processing \textbf{39}(4), 901--906 (1991).
\newblock \doi{10.1109/78.80912}.
\newblock \urlprefix\url{http://dx.doi.org/10.1109/78.80912}

\bibitem{nes83}
Nesterov, Y.: A method of solving a convex programming problem with convergence
  rate o(1/k2).
\newblock Soviet Mathematics Doklady \textbf{27}(2), 372--376 (1983)

\bibitem{Nes04}
Nesterov, Y.: Introductory lectures on convex optimization: A basic course,
  vol.~87.
\newblock Springer Science \& Business Media (2004)

\bibitem{nesterov1994interior}
Nesterov, Y., Nemirovskii, A.: Interior-point polynomial algorithms in convex
  programming.
\newblock SIAM, Philadelphia (1994)

\bibitem{NesP20}
Nesterov, Y., Protasov, V.Y.: Computing closest stable nonnegative matrix.
\newblock SIAM Journal on Matrix Analysis and Applications \textbf{41}(1),
  1--28 (2020)

\bibitem{noferini2020nearest}
Noferini, V., Poloni, F.: Nearest {\rm $\Omega$}-stable matrix via {Riemannian}
  optimization.
\newblock Numerische Mathematik  (2021).
\newblock \doi{10.1007/s00211-021-01217-4}

\bibitem{o2015adaptive}
O’donoghue, B., Candes, E.: Adaptive restart for accelerated gradient
  schemes.
\newblock Foundations of computational mathematics \textbf{15}(3), 715--732
  (2015)

\bibitem{ONV13}
Orbandexivry, F.X., Nesterov, Y., Van~Dooren, P.: Nearest stable system using
  successive convex approximations.
\newblock Automatica \textbf{49}(5), 1195--1203 (2013)

\bibitem{OveV05}
Overton, M., Van~Dooren, P.: On computing the complex passivity radius.
\newblock In: Proceedings of the 44th IEEE Conference on Decision and Control,
  pp. 7960--7964 (2005).
\newblock \doi{10.1109/CDC.2005.1583449}

\bibitem{PacD93}
Packard, A., Doyle, J.: The complex structured singular value.
\newblock Automatica \textbf{29}(1), 71--109 (1993)

\bibitem{PolS10}
Polyuga, R., Van~der Schaft, A.: Structure preserving model reduction of
  port-{Hamiltonian} systems by moment matching at infinity.
\newblock Automatica \textbf{46}(4), 665--672 (2010)

\bibitem{Pop73}
Popov, V.: Hyperstability of Control Systems.
\newblock Springer-Verlag New York, Inc., Secaucus, NJ, USA (1973)

\bibitem{PraS21a}
Prajapati, A., Sharma, P.: Estimation to structured distances to singularity
  for matrix pencils with symmetry structures: A linear algebra-based approach.
\newblock arXiv preprint arXiv:2105.13656  (2021)

\bibitem{razaviyayn2013unified}
Razaviyayn, M., Hong, M., Luo, Z.Q.: A unified convergence analysis of block
  successive minimization methods for nonsmooth optimization.
\newblock SIAM Journal on Optimization \textbf{23}(2), 1126--1153 (2013)

\bibitem{Sch2014port}
van~der Schaft, A., Jeltsema, D., et~al.: Port-{H}amiltonian systems theory: An
  introductory overview.
\newblock Foundations and Trends{\textregistered} in Systems and Control
  \textbf{1}(2-3), 173--378 (2014)

\bibitem{Sch06}
Schaft, A.v.: Port-{H}amiltonian systems: an introductory survey.
\newblock In: J.V. M.~Sanz-Sole, J.~Verdura (eds.) Proc. of the International
  Congress of Mathematicians, vol. III, Invited Lectures, pp. 1339--1365.
  Madrid, Spain (2006)

\bibitem{SchT07}
Schr{\"o}der, C., Stykel, T.: Passivation of {LTI} systems.
\newblock Preprint \textbf{368} (2007)

\bibitem{singh2008unified}
Singh, A.P., Gordon, G.J.: A unified view of matrix factorization models.
\newblock In: Joint European Conference on Machine Learning and Knowledge
  Discovery in Databases, pp. 358--373. Springer (2008)

\bibitem{suffridge1993approximation}
Suffridge, T., Hayden, T.: Approximation by a hermitian positive semidefinite
  toeplitz matrix.
\newblock SIAM Journal on Matrix Analysis and Applications \textbf{14}(3),
  721--734 (1993)

\bibitem{SunKS94}
Sun, W., Khargonekar, P., Shim, D.: Solution to the positive real control
  problem for linear time-invariant systems.
\newblock IEEE Trans. on Automatic Control \textbf{39}(10), 2034--2046 (1994)

\bibitem{SyrADG97}
Syrmos, V.L., Abdallah, C.T., Dorato, P., Grigoriadis, K.: Static output
  feedback: A survey.
\newblock Automatica \textbf{33}(2), 125 -- 137 (1997)

\bibitem{toh1999sdpt3}
Toh, K.C., Todd, M., T{\"u}t{\"u}nc{\"u}, R.: {SDPT3}--a matlab software
  package for semidefinite programming, version 1.3.
\newblock Optimization methods and software \textbf{11}(1-4), 545--581 (1999)

\bibitem{Ferrarinotes}
Trecate, G.F.: Nonlinear systems state feedback control.
\newblock
  \urlprefix\url{http://sisdin.unipv.it/labsisdin/teaching/courses/ails/files/6-State_feedback_control_handout.pdf}

\bibitem{tutuncu2003solving}
T{\"u}t{\"u}nc{\"u}, R., Toh, K., Todd, M.: Solving
  semidefinite-quadratic-linear programs using {SDPT3}.
\newblock Mathematical programming \textbf{95}(2), 189--217 (2003)

\bibitem{Udell2016lowrank}
Udell, M., Horn, C., Zadeh, R., Boyd, S.: Generalized low rank models.
\newblock Foundations and Trends in Machine Learning \textbf{9}(1), 1--118
  (2016)

\bibitem{Sch13}
{van der}~Schaft, A.: Port-{H}amiltonian differential-algebraic systems.
\newblock In: Surveys in Differential-Algebraic Equations, pp. 173--226.
  Springer (2013)

\bibitem{SchM95}
{van der}~Schaft, A., Maschke, B.: The {H}amiltonian formulation of energy
  conserving physical systems with external ports.
\newblock Arch. Elektron. {\"U}bertragungstech. \textbf{45}, 362--371 (1995)

\bibitem{SchM02}
{van der}~Schaft, A., Maschke, B.: {H}amiltonian formulation of
  distributed-parameter systems with boundary energy flow.
\newblock J. Geom. Phys. \textbf{42}, 166--194 (2002)

\bibitem{SchM13}
{van der}~Schaft, A., Maschke, B.: Port-{H}amiltonian systems on graphs.
\newblock SIAM J. Control Optim. \textbf{51}, 906--937 (2013)

\bibitem{Vandenberghe1996}
Vandenberghe, L., Boyd, S.: Semidefinite programming.
\newblock SIAM Review \textbf{38}(1), 49--95 (1996)

\bibitem{Var95}
Varga, A.: On stabilization methods of descriptor systems.
\newblock Systems \& Control Letters \textbf{24}(2), 133--138 (1995)

\bibitem{VoiB11}
Voigt, M., Benner, P.: Passivity enforcement of descriptor systems via
  structured perturbation of {H}amiltonian matrix pencils.
\newblock In: Talk at Meeting of the GAMM Activity Group Dynamics and Control
  Theory, Linz (2011)

\bibitem{waldspurger2020rank}
Waldspurger, I., Waters, A.: Rank optimality for the burer--monteiro
  factorization.
\newblock SIAM journal on Optimization \textbf{30}(3), 2577--2602 (2020)

\bibitem{WanC96}
Wang, H.S., Chang, F.R.: The generalized state-space description of positive
  realness and bounded realness.
\newblock In: Proc. of the 39th Midwest Symp. on Circuits and Systems, vol.~2,
  pp. 893--896 (1996)

\bibitem{WanZKPW10}
Wang, Y., Zhang, Z., Koh, C., Pang, G., Wong, N.: {PEDS}: Passivity enforcement
  for descriptor systems via {H}amiltonian-symplectic matrix pencil
  perturbation.
\newblock In: 2010 IEEE/ACM Int. Conf. on Computer-Aided Design (ICCAD), pp.
  800--807 (2010)

\bibitem{Wen88}
Wen, J.: Time domain and frequency domain conditions for strict positive
  realness.
\newblock IEEE Trans. on Automatic Control \textbf{33}(10), 988--992 (1988)

\bibitem{Wil84}
Wilkinson, J.: Sensitivity of eigenvalues.
\newblock Utilitas Math. \textbf{25}, 5--76 (1984)

\bibitem{wolkowicz2012handbook}
Wolkowicz, H., Saigal, R., Vandenberghe, L.: Handbook of semidefinite
  programming: theory, algorithms, and applications, vol.~27.
\newblock Springer Science \& Business Media (2012)

\bibitem{wright1999numerical}
Wright, S., Nocedal, J.: Numerical optimization.
\newblock Springer Science  (1999)

\bibitem{wright2015coordinate}
Wright, S.J.: Coordinate descent algorithms.
\newblock Mathematical Programming \textbf{151}(1), 3--34 (2015)

\bibitem{xu2013block}
Xu, Y., Yin, W.: A block coordinate descent method for regularized multiconvex
  optimization with applications to nonnegative tensor factorization and
  completion.
\newblock SIAM Journal on Imaging Sciences \textbf{6}(3), 1758--1789 (2013)

\bibitem{xu2017globally}
Xu, Y., Yin, W.: A globally convergent algorithm for nonconvex optimization
  based on block coordinate update.
\newblock Journal of Scientific Computing \textbf{72}(2), 700--734 (2017)

\bibitem{Yips81}
Yip, E., Sincovec, R.: Solvability, controllability, and observability of
  continuous descriptor systems.
\newblock IEEE Transactions on Automatic Control \textbf{26}(3), 702--707
  (1981)

\bibitem{yurtsever2021scalable}
Yurtsever, A., Tropp, J.A., Fercoq, O., Udell, M., Cevher, V.: Scalable
  semidefinite programming.
\newblock SIAM Journal on Mathematics of Data Science \textbf{3}(1), 171--200
  (2021)

\bibitem{ZhaLX02}
Zhang, L., Lam, J., Xu, S.: On positive realness of descriptor systems.
\newblock IEEE Transactions on Circuits and Systems I: Fundamental Theory and
  Applications \textbf{49}(3), 401--407 (2002)

\bibitem{Zho11}
Zhou, T.: On nonsingularity verification of uncertain matrices over a
  quadratically constrained set.
\newblock IEEE Trans. on Automatic Control \textbf{56}(9), 2206--2212 (2011)

\end{thebibliography}

\newpage 

\normalsize


\end{document}